\documentclass[reqno,11pt]{article}

\usepackage{a4wide}
\usepackage{geometry}
\usepackage{hyperref}
\usepackage{graphicx}
\usepackage{subfigure}
\usepackage[toc,page]{appendix}
\usepackage{xcolor}
\usepackage{pst-all}
\usepackage[english]{babel}
\usepackage{pdfsync}
\usepackage{amsmath,amssymb,amsthm}
\usepackage{mdwlist}
\usepackage{mathrsfs}
\usepackage{mathtools}
\usepackage{bm}
\usepackage{dsfont}
\usepackage{bbm}

\usepackage{tikz}
\usepackage[framemethod=TikZ]{mdframed}
\newcommand{\RCD}{\mathsf{RCD}}
\newcommand{\CD}{\mathsf{CD}}
\newcommand{\m}{\mathfrak{m}}
\newcommand{\n}{\mathfrak{n}}
\newcommand{\dis}{\mathsf{d}}
\newcommand{\di}{\mathsf{d}}
\newcommand{\N}{\mathbb N}

\newcommand{\PX}{\mathscr{P}}
\newcommand{\de}{\mathrm{d}}
\newcommand{\supp}{\text{supp}}

\newcommand{\X}{\mathsf{X}}

\newcommand{\M}{\mathscr{M}}
\newcommand{\Leb}{\mathcal{L}}
\newcommand{\p}{\mathsf p}
\newcommand{\g}{\mathit g}

\newcommand{\dint}{\mathsf{d}_{\mathsf{iKRW}}}
\newcommand{\Sin}{\mathcal{S}}
\newcommand{\weakto}{\overset{\ast}{\rightharpoonup}}

\DeclarePairedDelimiter{\abs}{\lvert}{\rvert}
\DeclarePairedDelimiter{\norm}{\lVert}{\rVert}

\DeclareFontFamily{U}{matha}{\hyphenchar\font45}
\DeclareFontShape{U}{matha}{m}{n}{
  <-6> matha5 <6-7> matha6 <7-8> matha7
  <8-9> matha8 <9-10> matha9
  <10-12> matha10 <12-> matha12
  }{}
\DeclareSymbolFont{matha}{U}{matha}{m}{n}
\DeclareMathSymbol{\Lt}{3}{matha}{"CE}

\def\Xint#1{\mathchoice 
  {\XXint\displaystyle\textstyle{#1}}%
  {\XXint\textstyle\scriptstyle{#1}}%
  {\XXint\scriptstyle\scriptscriptstyle{#1}}%
  {\XXint\scriptscriptstyle\scriptscriptstyle{#1}}%
  \!\int} 
\def\XXint#1#2#3{{\setbox0=\hbox{$#1{#2#3}{\int}$} 
  \vcenter{\hbox{$#2#3$}}\kern-.5\wd0}} 
\def\-int{\Xint -}

\numberwithin{equation}{section}

\newcommand{\diam}{\rm diam}

\newcommand{\R}{\mathbb{R}}

\newcommand{\mm}{{\mbox{\boldmath$m$}}}

\newcommand{\Kliminf}{K\kern-3pt-\kern-2pt\mathop{\rm lim\,inf}\limits}  
\renewcommand{\d}{{\mathrm d}}

\newcommand{\restr}[1]{\lower3pt\hbox{$|_{#1}$}} 

\newcommand{\nchi}{{\raise.3ex\hbox{$\chi$}}}

\renewcommand{\mm}{\mathfrak m}                                

\renewenvironment{proof}{\removelastskip\par\medskip   
\noindent{\em Proof.} \rm}{\penalty-20\null\hfill$\square$\par\medbreak}

\numberwithin{equation}{section}

\newtheorem{thm}{Theorem}[section]
\newtheorem{prop}[thm]{Proposition}
\newtheorem{lemma}[thm]{Lemma}
\newtheorem{cor}[thm]{Corollary}
\newtheorem{dfn}[thm]{Definition}

\newtheorem{example}[thm]{Example}

\newtheorem{clm*}{Claim}
\newtheorem{result*}{Useful Result}

\theoremstyle{remark}
\newtheorem{rmk}[thm]{Remark}

\title{Convergence of metric measure spaces satisfying the CD condition for negative values of the dimension parameter}
\author{Mattia Magnabosco\thanks{Institut f\"ur Angewandte Mathematik, Universit\"at Bonn. Email: magnabosco@iam.uni-bonn.de}, Chiara Rigoni\thanks{Institut f\"ur Angewandte Mathematik, Universit\"at Bonn. Email: rigoni@iam.uni-bonn.de}, Gerardo Sosa}

\makeindex

\begin{document}

\UseRawInputEncoding

\maketitle

\begin{abstract}
We study the problem of whether the curvature-dimension condition with negative values of the generalized dimension parameter is stable under a suitable notion of convergence. To this purpose, first of all we introduce an appropriate setting to introduce the $\CD(K, N)$-condition for $N < 0$, allowing metric measure structures in which the reference measure is quasi-Radon. Then in this class of spaces we introduce the distance $\dis_{\mathsf{iKRW}}$, which extends the already existing  notions of distance between metric measure spaces. Finally, we prove that if a sequence of metric measure spaces satisfying the $\CD(K, N)$-condition with $N < 0$ is converging with respect to the distance $\dis_{\mathsf{iKRW}}$ to some metric measure space, then this limit structure is still a $\CD(K, N)$ space.
\end{abstract}

\tableofcontents

\section{Introduction}
In the last years, the class of metric measure spaces satisfying the synthetic curvature-dimension condition has been a central object of investigation.  These spaces, in which a lower bound on the curvature formulated in terms of optimal transport holds, have been introduced by Sturm in \cite{Sturm06I, Sturm06II} and independently by Lott and Villani in \cite{Lott-Villani09}. For a metric measure space $(\X, \dis, \mm)$, the curvature-dimension condition $\CD(K, N)$ depends on two parameters $K \in \R$  and $N \in [1, \infty]$ and it relies on a suitable convexity property of the entropy functional defined on the space of probability measures on $\X$: the $\CD(K, N)$-condition for finite $N$ is an appropriate reformulation of the $\CD(K, \infty)$ one introduced as the $K$-convexity of the relative entropy with respect to $\m$. Spaces satisfying the curvature-dimension condition are Riemannian manifolds \cite{Sturm06I, Sturm06II}, Finsler spaces \cite{Ohta09} and Alexandrov spaces \cite{Petrunin11, ZhangZhu10}. In particular, in the case of a weighted Riemannian manifold, namely a Riemannian manifold $(\mathit M, \g)$ equipped with a weighted measure $\mm = e^{-\psi} \text{vol}_\g$ which leads to a weighted Ricci curvature tensor $\mathrm{Ric}_N$,  being a $\CD(K, N)$ space is equivalent to the condition $\mathrm{Ric}_N \ge K$ that can be regarded as the combination of a lower bound by $K$ on the curvature and an upper bound by $N$ on the dimension. Moreover, in the setting of Riemannian manifolds, it turns out that for $N > 0$ it is possible to characterize the $\CD(K, N)$-condition in terms of a property of the relative entropy, as in the case of metric measure spaces satisfying the  $\CD(K, \infty)$-condition, for brevity, the class of $\CD(K, \infty)$ spaces: the required property is the {$(K, N)$-convexity} introduced in \cite{Erbar-Kuwada-Sturm13}. This notion reinforces the one of $K$-convexity, and can be generalized to the case of metric measure spaces.\\

In the Euclidean setting a direct application of the results in \cite{BraLieb} ensures that given any convex measure $\mu$, with full dimensional convex support and $C^2$ density $\Psi$, the space $(\R^n, \dis_{\text{Eucl.}}, \mu)$ satisfies the $\CD(0, N)$-condition for $1/N \in [-\infty, 1/n]$ (i.e. $N \in (-\infty, 0) \cup [n, \infty]$), in the sense that $\mathrm{Ric}_N \ge 0$. This class of measures, introduced by Borell in \cite{Borell}, extends the set of the so-called log-concave ones and it has been largely studied for example in \cite{Bobkov07, BobLed, Kol13}. In particular, following the terminology adopted by Bobkov, the case $N \in (-\infty, 0)$ corresponds to the ``heavy-tailed measures'' (see also \cite{BCR05}), identified by the condition that $1/ \Psi^{1/(n-N)}$ is convex. An explicit example of these measures is given by the family of Cauchy probability measures on $(\R^n, \dis_{\text{Eucl.}})$
\begin{equation}\label{eq:intro1}
\mu^{n, \alpha} = \dfrac{c_{n, \alpha}}{\big(1+\abs{x}^2\big)^{\frac{n + \alpha}{2}}} \, \d x, \quad \alpha > 0,
\end{equation}
where $c_{n, \alpha} > 0$ is a normalization constant. Hence, $(\R^n, \dis_{\text{Eucl.}}, \mu^{n, \alpha})$ is a $\CD(0, -\alpha)$ space.\\ 

Admitting $N < 0$ may sound strange if one thinks to $N$ as an {upper bound on the dimension}; however, as explained in \cite{Ohta-Takatsu1} and \cite{Ohta-Takatsu2}, in the case of weighted Riemannian manifolds, it is useful to consider a generalization of the entropy, called {$m$-relative entropy} $H_m(\cdot | \nu)$, stemming from the {Bregman divergence} in {information geometry}, which is closely related to the {R\'enyi entropies} in {statistical mechanics}. More precisely, in these papers Ohta and Takatsu prove that if $(\mathsf M, \omega)$ is a weighted Riemannian manifold and $\nu = \exp_m (\Psi) \omega$ is a conformal deformation of $\omega$ in terms of the $m$-exponential function, then the fact that $H_m(\cdot | \nu) \ge K$ in the Wasserstein space $(\PX^2(\M), W_2)$ is equivalent to the fact that $\text{Hess} \Psi \ge K$ and $\mathrm{Ric}_N \ge 0$ with $N = 1/(1-m)$, where $\mathrm{Ric}_N$ is the weighted Ricci curvature tensor associated with $(\mathsf M, \omega)$. In this setting, depending on the choice of the particular entropy, i.e., the value of $m$, the value of the dimension $N$ can be negative. 
Hence they show that the bounds $\text{Hess} \Psi \ge K$ and $\mathrm{Ric}_N \ge 0$ imply appropriate variants of the Talagrand, HWI, logarithmic Sobolev and the global Poincar\`e inequalities as well as the concentration of measures. Moreover, using similar techniques as in \cite{JKO98, Vil03, OS09}, they prove that the gradient flow of $H_m(\cdot | \nu)$ produces a weak solution to the porous medium equation (for $m > 1$) or the fast diffusion equation (for $m < 1$) of the form
\begin{equation}\label{eq:intro2}
\dfrac{\partial \rho}{\partial t} = \dfrac{1}{m} \Delta^\omega (\rho^m) + \text{div}_\omega(\rho \nabla \Psi),
\end{equation}
$\Delta^\omega$ and $\text{div}_\omega$ being the Laplacian and the divergence associated with the measure $\omega$. This is result was demonstrated also by Otto \cite{Otto01} in the case in which the reference measure $\nu$ in the {$m$-relative entropy} $H_m(\cdot | \nu)$ is given by the family of $m$-Gaussian measures, which is in turn closely related to the Barenblatt solution to \eqref{eq:intro2} without drift (see \cite{OW, Ta2}).\\

In \cite{Ohta16}, the author extends the range of admissible ``dimension parameters'' to negative values of $N$ in the theories of $(K, N)$-convex functions, of tensors $\mathrm{Ric}_N$ and of the $\CD(K, N)$-condition in the more general setting of metric measure spaces.  In particular, it is proved that the $(K, N)$-convexity for $N < 0$ is weaker than the $K$-convexity, thus it covers a wider class of functions. This means that the class of metric measure spaces satisfying the $\CD(K, N)$-condition for negative values of $N$ includes all $\CD(K, \infty)$ ones; in particular, since a metric measure space which satisfies the $\CD(K, N)$-condition for some $N > 0$ is also a $\CD(K, \infty)$ space, it follows that:

\vspace{0.37cm}

\noindent \hspace{-0.15cm} \framebox{\parbox[c]{4.45cm}
{ \vspace{0.1cm}\small{ $(\X, \dis, \mm)$ is a $\CD(K, N)$ space,}\\ \vspace{-0.5cm} \begin{center} \small{for some $N > 0$.} \vspace{0.1cm} \end{center}\vspace{-0.2cm}}} $\Rightarrow$ {\fboxsep=0.7em\fbox{\parbox[c]{4.385cm} {\small{$(\X, \dis, \mm)$ is a $\CD(K, \infty)$ space}}}} $\Rightarrow$ \framebox{\parbox[c]{4.45cm}{\vspace{0.07cm} \small{$(\X, \dis, \mm)$ is a $\CD(K, N)$ space,}\\ \vspace{-0.5cm}\begin{center} \small{for any $N < 0$.} \end{center}\vspace{-0.2cm}}} \vspace{0.37cm}

The curvature-dimension condition for negative values of the dimension has not been largely studied up to now. In the setting of metric measure spaces, the only paper devoted to the study of this notion is the aforementioned work by Ohta \cite{Ohta16}. Therein, many direct consequences are extracted from the definitions, as in the case of the standard curvature-dimension bounds theory and a number of results valid in the case of $N > 0$ are generalized to these spaces, including the Brunn-Minkowski inequality and some other functional ones. 

Nevertheless, most of the results on this topic are obtained in the case of weighted Riemannian manifolds. A first example of a model space is provided in \cite{Milman-NegSph}: it is therein proved that the {$n$-dimensional unit sphere} equipped with the {harmonic measure}, namely the {hitting distribution}  by the {Brownian motion} started at $x \in \mathbb{S}^n$, $|x| < 1$ (which can be equivalently described as  the probability measure whose density is proportional to $\mathbb{S}^n \ni y \to 1 / | y - x |^{n+1}$) is a $\CD\big(n-1-(n+1)/4, -1 \big)$ space. More generally, Milman provides an equivalent to the family \eqref{eq:intro1} of Cauchy measures in $\R^n$, showing that the family of probability measures on the $n$-dimensional unit sphere having density proportional to
\[
\mathbb{S}^n \ni y \mapsto \dfrac{1}{\abs{y-x}^{n + \alpha}}
\]
satisfies the curvature-dimension condition $\CD(n-1-\frac{n+\alpha}{4}, -\alpha)$ for all $\abs{x} < 1$, $\alpha \ge -n$ and $n \ge 2$. In \cite{Milman-Neg}, the author studies the isoperimetric, functional and concentration properties of $n$-dimensional weighted Riemannian manifolds satisfying a uniform bound from below on the tensor $\text{Ric}_N$, when $N \in (-\infty, 1)$,  providing a new one-dimensional model-space under an additional diameter upper bound (namely, a positively curved sphere of possibly negative dimension). In this setting, many other rigidity results have been obtained, such as the splitting theorem for weighted Lorentzian and Riemannian manifolds satisfying the $\CD(K, N)$-condition for any $N < 1$ (see \cite{Wylie17, Sakurai19, Sakurai20}). Other interesting geometric results have been proved when the tensor $\text{Ric}_N$ for $N \in (-\infty, 0]$ is uniformly bounded from below: for example, in the paper \cite{KolMilm}, Kolesnikov and Milman prove various Poincar\'e-type inequalities on the manifolds and their boundaries (making use of the Bochner's inequality and of the Reilly formula, when the boundary is nonempty).

Finally, let us underline that Bochner's inequality, generalized to the setting of weighted Riemannian manifolds satisfying the $\CD(K, N)$-condition for $N < 0$ in \cite{Ohta16} and in \cite{KolMilm}, independently, does not yet have a corresponding in the nonsmooth setting of metric measure spaces. We recall that for $N > 0$, this important inequality has been extended to the setting of singular spaces in a series of works, precisely in \cite{Ohta-SturmBoch} for Finsler manifolds, in \cite{Gigli-Kuwada-Ohta10} and \cite{ZhangZhu10} for Alexandrov spaces and in \cite{AmbrosioGigliSavare12} for $\RCD(K, \infty)$ spaces.\\

Despite the progress made in \cite{Ohta16}, some fundamental questions remain open. The objective of this paper is to address the question of whether the curvature-dimension condition with negative value of generalized dimension is stable under convergence in a suitable topology. Special attention has to be payed to establishing an appropriate setting. In fact, inspired by some of the results found in \cite{Ohta16}, we prove that for any $N < -1$ the interval $I := [-\pi/2, \pi/2]$ equipped with the Euclidean distance and the weighted measure $\d \mm(x) := \cos^N(x) \d\mathcal{L}^1(x)$, $\mathcal{L}^1$ being the 1-dimensional Lebesgue measure on $I$, is a $\CD(N, N)$ space. This fundamental example shows that the natural setting to introduce this curvature-dimension condition cannot be the one of complete and separable metric (Polish, in short) spaces equipped with Radon measures as in the case of $\CD(K, N)$ spaces with $N > 0$, but rather the one of Polish spaces endowed with  {quasi-Radon} measures, i.e., measures which are Radon outside a negligible set. In fact, roughly speaking, the information that the weighted measure $ \cos^N(x) \d\mathcal{L}^1(x)$ is the right one to consider in order to have a space with negative dimension comes from the theory of $(K, N)$-convex functions (see \cite[Section 2]{Ohta16}). However, despite the fact that the ``natural'' domain for the function $\cos^N(x)$ with $N < 0$ would be the open interval  $(-\pi/2, \pi/2)$, the theory of optimal transport forces us to consider the underling metric space to be complete and separable, in order to ensure that also the Wasserstein space $(\PX^2(\X), W_2)$ enjoys the same properties.  Furthermore, we prove that also the space obtained by gluing together $n$-copies of the interval $(I, \dis_{\text{Eucl.}}, \mm)$  introduced above still satisfies the $\CD(N, N)$-condition: this in particular shows that the negligible set of points in which the reference measure explodes is not just appearing in the ``boundary'' of our space, but also in the interior of it.

In this new and more general setting, a sequence of spaces satisfying   the curvature-dimension condition for negative dimension parameter may fail to be stable under the standard measured Gromov-Hausdorff convergence of metric measure spaces. For example, it can be the case that a well-defined limit of a sequence of $\CD(K,N)$ spaces does not exist due to failure of convergence of the metric or the measure. We present some examples of this kind of behavior for metric measure spaces whose reference measures are quasi-Radon:
\begin{itemize}
\item[1)] (\emph{$\sigma$-finiteness lost in limit}) Consider the sequence of compact metric measure spaces given by $\{([0,1],\dis_{\text{Eucl.}},\m_n :=x^{-n}dx)\}_{n\in\mathbb{N}}$. Since these measures are unbounded, it is not clear a priori in which way we want the measures to converge. One possibility, however, is the following: note that for every neighborhood $U$ of $0$ the measure $\m_n|_{[0,1]\setminus U}$ is finite, therefore, up to being cautious with boundaries, one could ask for the weak$^*$-convergence of the restricted finite measures $\m_n|_{X\setminus U}$ to some measure $\m_\infty^U$, for every neighborhood $U$ of the origin. Then using an extension theorem we would obtain a unique measure $\m_\infty$ defined on the whole interval $[0, 1]$. However this construction leads to a measure $\m_\infty$ which is infinite for every measurable subset $A$ of $[0,1]$ with $\Leb^1(A) > 0$, losing thus any regularity. 
\item[2)] (\emph{Bounded measures to unbounded measures}) Consider now the sequence of metric measure spaces given by $\big\{([2^{-n}, +\infty), |\cdot|, x^N \Leb^1)\big\}_{n\in \N}$ for some fixed $N < -1$. For each $n \in \N$, the reference measure of the space is  Radon, but the limit measure is such that every neighborhood of $0$ has infinite mass.
\end{itemize}

We recall that in the setting of metric measure spaces, a suitable notion of convergence, called measured Gromov-Hausdorff convergence,  was  introduced by Fukaya in \cite{Fukaya} as a natural variant of the purely metric Gromov-Hausdorff one. Then the stability of the $\CD(K, N)$-condition for $N \in [1, \infty)$  was proved following these two approaches:
\begin{itemize}
\item[$\bullet$] Lott and Villani proved that the $\CD(K, N)$-condition is stable under pointed measured Gromov-Hausdorff convergence in the class of proper pointed metric measure spaces. Roughly speaking, this means that for any $R > 0$ there is a measured Gromov-Hausdorff convergence of balls of radius $R$ around the given points of the spaces;
\item[$\bullet$] Sturm worked in the setting of Polish spaces equipped with a probability with finite second moment as reference measure. In this class of spaces he defined a distance $\mathbb D$ by putting
\[
\mathbb D \Big( (\X_1, \dis_1, \mm_1), (\X_2, \dis_2, \mm_2) \Big) := \inf W_2\big((\iota_1)_\sharp \mm_1, (\iota_2)_\sharp \mm_2 \big),
\]
the infimum being taken among all complete and separable metric spaces $(\X, \dis)$ and all the isometric embeddings $\iota_i \colon (\text{supp}(\mm_i), \dis_i) \to (\X, \dis)$, $i =1, 2$. He then showed that the curvature-dimension condition is stable with respect to this $\mathbb D$-convergence.
\end{itemize}

In particular, these two techniques produce the same convergence in the case of compact and doubling metric measure spaces. Hence, in \cite{GigMonSav} Gigli, Mondino and Savar\'e introduce a notion of convergence of metric measure spaces, called pointed measured Gromov convergence, which works without any compactness assumptions on the metric structure and for more general Radon measures which are finite on bounded sets. Moreover, they prove that lower Ricci bounds are stable with respect to this convergence.\\

The first achievement of this paper we propose a suitable setting to introduce the curvature-dimension condition for negative values of the dimension parameter, extending and complementing the work by Ohta \cite{Ohta16}. We then propose an appropriate notion of distance, that we call intrinsic pointed Kantorovich-Rubinstein-Wasserstein distance $\dis_{\mathsf{iKRW}}$, and we prove that the curvature-dimension bounds with negative values of the dimension are stable with respect to the $\dis_{\mathsf{iKRW}}$-convergence. In particular, this distance extends the one introduced in \cite{GigMonSav} to the set of equivalence classes of metric measure spaces with more general $\sigma$-finite measures, allowing us to analyze sequences of metric measure spaces in which the reference measures may ``explode'' in some points and which are not necessarily finite on bounded sets (we underline that also in this setting we do not require the locally compactness assumption on the metric structure). 

More specifically, the structures we work with are isomorphism classes of \emph{pointed generalized metric measure spaces}  $(\X,\dis,\m, \mathcal C, p)$ where:
\begin{itemize}
\item $(\X,\dis)$ is a complete separable  metric space,
\item $\m \in \M^{qR}(\X)$ is a quasi-Radon measure, $\m \not= 0$,
\item $\mathcal C \subset \X$ is a closed set with empty interior and $\mm(\mathcal C) = 0$,
\item $p\in\supp(\m) \subset \X$ is a distinguished point,
\end{itemize}
and $(\X_1,\dis_1,\m_1, \mathcal C_1, p_1)$ is said to be isomorphic to $(\X_2,\dis_2,\m_2, \mathcal C_2, p_2)$ if there exists
\[
\text{an isometric embedding } i \colon \supp(\mm_1) \to \X_2 \text{ such that } i(\mathcal C_1) = \mathcal C_2, i_\sharp \mm_1 = \mm_2 \text{ and } i(p_1) = p_2.
\]
Intuitively, here for ``quasi-Radon measure'' $\mm$ on $(\X, \dis)$ (following the terminology introduced in \cite[Volume 4]{Fremlin}) we mean a complete $\sigma$-finite measure with the following properties:
\begin{itemize}
\item there exists a closed negligible set with empty interior $\mathcal{S}_{\m} \subset \X$ such that $\m(U) = \infty$ for every open neighborhood $U$  of $ x \in \mathcal{S}_{\m}$
\item the restricted measure $\mm|_{\X \setminus \mathcal{S}_{\m}}$ is Radon on the open set $\X \setminus \mathcal{S}_{\m}$.
\end{itemize}
In this class of spaces we then introduce the intrinsic distance $\dis_{\mathsf{iKRW}}$. This is constructed by taking partitions of the space: each element of the partition has finite measure and can be then renormalized; hence, we measure the intrinsic Kantorovich-Rubinstein-Wasserstein distance between these renormalized elements. In doing so, we take inspiration from the ideas behind the construction of the distance $\mathrm{p}\mathbb{G}_W$  in \cite{GigMonSav}. However, in contrast to their setting, the lack of regularity of the measure becomes an obstacle to find a canonical and appropriate manner to partition the metric measure space. In particular, it turns out that a control on the Hausdorff distance of the singular sets in the definition of the $\dis_{\mathsf{iKRW}}$-distance is actually necessary in order to provide an extrinsic realization of the distance given an intrinsic one.

Then in this setting the $\CD(K, N)$-condition for negative values of $N$ is introduced requiring a suitable convexity property of the extended R\'enyi entropy functional defined on the space of probability measures on $\X$, as in the case $N > 0$. 

We prove that this notion is stable with respect to the $\dis_{\mathsf{iKRW}}$-distance: our main result (Theorem \ref{th:stab}) shows that if a sequence of pointed generalized metric measure spaces $\{(\X_n,\dis_n,\m_n, \mathcal C_n, p_n)\}_{n \in \N} $ satisfying the $\CD(K, N)$-condition for some $N < 0$ (and some other technical assumptions) is converging  in the $\dis_{\mathsf{iKRW}}$-distance to some generalized metric measure space $(\X_\infty,\dis_\infty, \m_\infty, \mathcal C_\infty, p_\infty)$, then this limit structure is still a $\CD(K, N)$ space.

This result in the case $N > 0$ strongly relies on the fact that the (standard) R\'enyi entropy functional is lower semicontinuous with respect to the weak topology in $\PX_2(\X)$. Unfortunately, the same property does not hold for the extended R\'enyi entropy functional $S_{N,\m}$ when the reference measure $\mm$ is quasi-Radon. Therefore we provide a new argument to prove this stability, which extends the proofs of Lott-Villani and Sturm when $N > 0$ and the one of Gigli-Mondino-Savar\`e when $N = \infty$ (in all these classes of spaces the reference measure $\mm$ is Radon, namely $\mathcal{S}_{\m} =\emptyset$). We show that $S_{N,\m}$ is weakly lower semicontinuous on the space
\begin{equation*}
    \PX^\mathcal {\mathcal S_\m}(\X):= \{\mu\in \PX_2(\X) \,: \, \mu(\mathcal S_\m)=0\}
\end{equation*}
and that this will be enough to prove the desired stability result, provided that each one of the spaces in the converging sequence $\{(\X_n,\dis_n,\m_n, \mathcal C_n, p_n)\}_{n \in \N} $ is not accumulating ``too much mass'' around any of the points in $\mathcal S_{\m_n}$ in a uniform way.

Finally, in Theorem \ref{th:ext} we manage to adapt this \emph{Stability Result} also in the case in which we only have at our disposal a distance which is not explicitly dependent on the behavior of the $\mm$-singular sets. Intuitively, one of the examples we would like to include in our theory consists in approximating the $\CD(0, N+1)$ space $([0, \infty), |\cdot|, x^N \Leb^1)$, where $N < -1$, making use of the sequence of metric measure spaces $([2^{-n}, +\infty), |\cdot|, x^N \Leb^1)$: clearly each space in this sequence is still a $\CD(0, N+1)$ space for $N < -1$ but now the singularity of the measure is ruled out from the domain, meaning that each metric space $([2^{-n}, +\infty), |\cdot|)$ is actually equipped with a Radon measure. Hence, we rely on an extrinsic approach to convergence which does not require any control on the Hausdorff distance between $\mm$-singular sets in the definition of the $\dis_{\mathsf{iKRW}}$-distance.\\

\noindent{\bf Acknowledgements:} The authors thank Professor Karl-Theodor Sturm for suggesting them the problem and for many valuable discussions. Many thanks are due to Lorenzo Dello Schiavo, for many enlightening conversations on measure theory. The second author gratefully acknowledges support by the European Union through the ERC-AdG 694405 RicciBounds. A large part of this work was written while the third author was employed in the group of Professor Karl-Theodor Sturm and furthermore he very gratefully acknowledges support by the European Union through the ERC-AdG 694405 RicciBounds.

\section{Metric spaces equipped with quasi-Radon measures}

\subsection{Measure theory background}\label{sec:preliminaries}
\subsubsection{Quasi-Radon measures}
We begin by introducing some notation and concepts of measure theory.
Let $\X$ be a set and $\mathcal{T},\,\Sigma,\,\m$ be, respectively, a topology, a $\sigma$-algebra, and a positive measure defined on $\Sigma$. The quadruple $(\X, \mathcal{T}, \Sigma, \m)$ is referred to as a \emph{topological measure space} if the condition $\mathcal{T} \subseteq \Sigma$ is satisfied. From now on we assume that $\m(\X)\not = 0$.
\begin{dfn}
Let $(\X,\mathcal{T},\Sigma, \m)$ be a topological measure space. We say that the measure $\m$ is:
\begin{enumerate}
\item[i)] \emph{locally finite} if for every $x \in \X$ there exists a neighborhood $U \in \mathcal{T}$ with $\m(U) < \infty$;
\item[ii)]  \emph{effectively locally finite} if for every $A\in \Sigma$ with $\m(A)>0$, there exists an open set $U\in \mathcal{T}$ with finite measure such that  $\m(A\cap U)>0$;
\item[iii)]  \emph{$\sigma$-finite} if there exists $\{A_i\}_{i\in\N} \subset \Sigma$ with $\m(A_i) < \infty$ for any $i \in \N$ such that $\X = \cup_{i\in\N}A_i$;
\item[iv)] \emph{inner regular with respect to a family of sets $\mathcal{F}\subset \Sigma$} if for any $E \in \Sigma$ we have that \[\m(E) = \sup\{\m(A) : A\in \mathcal{F} \text{ and }A\subset E \}.\]
In the particular case in which $\mathcal{F}$ is the family of compact sets of $\X$, this property is called \emph{tightness}.
\end{enumerate}
\end{dfn}

\noindent
We denote by $\mathcal{B}(\X)$ the smallest $\sigma$-algebra containing all open sets of a topological space $(\X,\mathcal{T})$. Recall that $\mathcal{B}(\X)$ is called  the Borel sigma algebra of $(\X,\mathcal{T})$ and that a measure defined on $\mathcal{B}(\X)$ is referred to as a Borel measure. In the following, with the expression \emph{``$\m$ is a measure on $\X$''} we refer to the fact that $\m$ is a measure defined on $\mathcal{B}(\X)$.  If $(\X,\mathcal{T})$ is a Hausdorff space, then compact subsets are in $\mathcal{B}(\X)$, since compact sets are closed in such spaces. An effectively locally finite Borel measure on a metric space  is inner regular with respect to closed sets: a proof of this fact can be found in \cite[Volume 4, Theorem 412E]{Fremlin}. 
Furthermore, a finite Borel measure on a complete and separable  metric (in short, a Polish)  space is tight (see \cite[Volume 2, Theorem 7.1.7]{Bogachev07}).

At this point we recall the following useful characterization of tight measures (see  \cite[Volume 4, Corollary 412B]{Fremlin} for a proof of this result):
\begin{prop}\label{prop:charatight}
Let $(\X,\dis)$ be a Polish space and $\m$ a Borel effectively locally finite measure on $\X$. Then the following are equivalent:
\begin{itemize}
\item[i)] the measure $\m$ is tight,
\item[ii)] for every measurable $A \in  \mathcal{B}(\X)$ with $\mm(A) > 0$ there exist a measurable compact set $K\subset A$ such that $\m(K) > 0$ .
\end{itemize}
\end{prop}
This in particular allows us to prove an important property of effectively locally finite measures defined on a Polish space:
\begin{lemma}\label{le:inn_reg}
Let $(\X,\dis)$ be a Polish space and $\m$ a Borel effectively locally finite measure on $\X$. Then the measure $\m$ is tight. 
\end{lemma}
\begin{proof}
In order to prove the tightness of $\mm$, we show the validity of point $ii)$ in Proposition \ref{prop:charatight} by exhibiting the existence of a compact set $K\subset \X$ with $\m(K)>0$ in the case in which $\X$ has finite measure. 
In fact, the assumption on $\m$ to be effectively locally finite  guarantees the existence of a set $B \subset A$ with finite positive measure for any $A\subset \Sigma$ with $\mm(A) > 0$. Now, since a Borel measure on a metric spaces is inner regularity with respect to closed sets, there exists also a closed subset $C\subset B$ of positive finite measure.
We can then conclude noticing that closed subsets of Polish spaces are Polish as well and recalling that a finite Borel measure on a Polish space is tight.
\end{proof}

The following classes of Borel measures are of central interest to us.
\begin{dfn}[Radon and quasi-Radon measures]\label{def:radon}
Let $(\X,\dis)$ be a separable metric space. We say that a  Borel measure $\m$ is
\begin{enumerate}
\item[i)] \emph{Radon} if it is complete, locally finite, and inner regular with respect to compact sets. 
\item[ii)]  \emph{quasi-Radon} if it is complete, effectively locally finite, and inner regular with respect to closed sets.
\end{enumerate}
\end{dfn} 
\begin{rmk}[Assumptions on inner regularity]\label{rmk:inreg}
We underline that in Definition \ref{def:radon} we can drop the assumptions on the inner regularity of the measure whenever $(\X,\dis)$ is complete. Indeed in this setting both locally finite and effectively locally finite measures are tight, in view of Lemma \ref{le:inn_reg}, and for Hausdorff spaces tight measures are inner regular with respect to closed sets.
\end{rmk}
\begin{rmk}[Compatibility of definitions]
For general topological measure spaces it is required to  enrich the defining properties of Radon and quasi-Radon measures; this is understandable, since generally this class of spaces have much less structure. A broad definition can be found in \cite[Volume 4, Definitions 411(a) - 411(b)]{Fremlin}. Restricted to the framework in which we work, it turns out that the just mentioned definitions are equivalent to ours. 
\end{rmk}

Let us now list some properties that these classes of measures enjoy.
\begin{prop}\label{pr:radon}
For any $(\X,\dis)$ separable metric space equipped with a measure $\m$, the following implications are valid:
\begin{itemize}
\item[i)] if $\m$ is a  Radon measure, then $\m$ is a quasi-Radon measure;
\item[ii)] assume that $\m$ is locally finite, quasi-Radon, and such that, for every Borel set $E$ with positive measure, there exists a measurable compact set $K\subset E$ with the property that $\m (E\setminus K) > 0$. Then $\m$ is a Radon measure. If $(\X, \dis)$ is a complete metric space, then the same conclusion holds just assuming that $\m$ is a locally finite quasi-Radon measure.
\end{itemize}
Moreover, if $\m$ is a quasi-Radon measure, then the following properties hold:
\begin{itemize}
\item[iii)] $\X$ can be covered up to a set of measure zero by countable many open sets of finite measure. Specifically, there exists a countable collection of open sets  $\{U_i\}_{i\in \N}$ with the property that $\m(U_i)<\infty$ for every $i\in N$ and $\m(\X\setminus \bigcup_{i\in \N} U_i)=0$.
\item[iv)] $\m$ is $\sigma$-finite. 
\item[v)] if $\m$ is inner regular with respect to compact sets, then there exists a closed set $E\subset \X$ with $\mm(E) = 0$ such that $\m|_{\X\setminus E}$ is a Radon measure on the open set $\X\setminus E$.
\end{itemize}
In particular, for any quasi-Radon measure $\mm$ defined on a Polish space $(\X, \dis)$ there exists a closed set $\mathcal S_\mm$ with $\mm(\mathcal S_\mm) = 0$ such that $\m|_{\X\setminus \mathcal S_\mm}$ is a Radon measure on the open set $\X\setminus \mathcal S_\mm$.
\end{prop}
\begin{proof}
The first point $i)$ follows from the fact that on a separable metric space $(\X, \dis)$ a locally finite tight measure is essentially locally finite (see  \cite[Volume 4, 416A]{Fremlin} for a proof of this result). \\
As for $ii)$,  the first part follows directly from the characterization of tightness in Proposition \ref{prop:charatight}. Hence Lemma \ref{le:inn_reg} implies the second part (see for example Remark \ref{rmk:inreg}).\\
Let us then prove $iii)$. We start recalling that a topological space $Y$ is called Lindelh\"of if for every open cover of $Y$ there exist a countable sub-cover. Moreover, $Y$ is called hereditary Lindelh\"of if the same property holds for every subset $\mathcal{V}\subset Y$. Now we note that $(\X, \dis)$ is second countable,  since it is separable, and that second countable topological spaces are Lindel\"of. Moreover since second countability is an hereditary property we have that a separable metric space is hereditary Lindel\"of. 
\\ We then define $\mathcal{V}$ to be the union of all open sets $V\subset \X$ which have finite measure and let $\{U_i\}_{i\in\N}$ be a countable sub-cover given by the hereditary Lindel\"of property of $\X$. We set $\mathcal{U}:=\cup_{i\in\N}U_i$.
Observe that $\m(E)=0$, for the complement set $E=X\setminus \mathcal{U}$. Indeed, the claim follows from the effectively locally finite property of the measure $\m$, since the intersection of $E$ with all open sets of finite measure is empty, i.e., $E\cap \mathcal{U}=\emptyset$.\\
At this point, the proof of $iv)$ and $v)$ are straightforward. In fact, the  sets $\{\{U_i\}_{i\in \N}, E\} $ defined in $iii)$ have finite measure and cover $\X$ while the set null set $E\subset X$ is closed, being the complement of a countable union of open sets. Lastly, note that $\m|_{X\setminus E}$ is locally finite and, by assumption, inner regular with respect to compact sets.\\
In the case of a Polish space $(\X, \dis)$, this last property $v)$ together with Lemma \ref{le:inn_reg} guarantees the existence of the closed set $\mathcal S_\mm$ with the property that $\mm(\mathcal S_\mm) = 0$ and $\m|_{\X\setminus \mathcal S_\mm}$ is a Radon measure.
\end{proof}

Our goal now is to state a general version of the Radon-Nikodym Theorem valid for quasi-Radon measures. With this aim, we first introduce another concept which is a strengthening of absolute continuity between measures.

\begin{dfn}[Absolutely continuous and truly continuous measures]
Let $(\X, \Sigma, \m)$ be a measurable space. We say that a measure $\mu$ on $\Sigma$ is
\begin{itemize}
\item[i)] \emph{absolutely continuous with respect to a measure $\m$} if for any $\varepsilon > 0$ there is $\delta > 0$ such that $\mu(E) \le \varepsilon$ for any measurable set $E \subseteq X$ with $\m(E) \le \delta$;
\item[ii)] \emph{truly continuous with respect to a measure $\m$} if for every $\epsilon > 0$ there exists $\delta > 0$ and a measurable set $E \subseteq X$ such that $\m(E) < \infty$ and $\mu(F) \le \epsilon$ for any measurable $F \subseteq X$ with $\mu(E \cap F) \le \delta$.
\end{itemize}
\end{dfn}

\begin{prop}\label{trac}
Let $(X, \Sigma, \m)$ be a measurable space and $\mu$ be a measure on $\Sigma$. Then $\mu$ is truly continuous with respect to $\m$ if and only if the following two conditions are satisfied:
\begin{itemize}
\item[i)] $\mu$ is absolutely continuous with respect to $\m$
\item[ii)] for any $E \in \Sigma$ with $\mu(E) > 0$ there is $F \in \Sigma$ such that $\m(F) < \infty$ and $\mu (E \cap F) > 0$.
\end{itemize}

Moreover, if the measure $\m$ is $\sigma$-finite, then $\mu$ is truly continuous with respect to $\m$ if and only if it is absolutely continuous.
\end{prop}

\begin{proof}
Let $\mu$ be a truly continuous measure. Directly from the definitions, it follows that $\mu$ is also absolutely continuous. Then we take $E \in \Sigma$ with $\mu(E) \neq 0$ and we observe that there exists $F \in \Sigma$, with $\m(F) < \infty$, such that $\mu(G) < \mu(E)$ whenever $G \in \Sigma$ and $G \cap F = \emptyset$. This in particular means that $\mu (E \setminus F) < \mu(E)$ and $\mu(E \cap F) > 0$.

Conversely, let $\mu$ be an absolutely continuous measure such that condition $ii)$ is satisfied. We want to prove that $\mu$ is truly continuous.
First of all we define $S := \sup\{ \mu(F) : F \in \Sigma, \m(F) < \infty  \}$ and we take a sequence $\{ F_n \} \subset \Sigma$ of sets with finite $\m$-measure such that $\lim_{n \to \infty} \mu(F_n)$. Then we define the set $F^* := \cup_{n \in \N} F_n$ and we observe that $\mu(F^*) > 0$ and that
\begin{equation}\label{Gzero} \text{for any } \,\, G \in \Sigma \,\, \text{ with } G \cap F^* = \empty \text{ it holds } \mu(G) = 0. \end{equation} 
Indeed, by contradiction, if $\mu(G) > 0$, then there would exist $H \in \Sigma$ with $\m(H) < \infty$ such that $\m(H \cap G) > 0$ and this would mean that $G \cap F^* \neq \empty$.

At this point for any $n \in \N$ we define the set $F_n^* := \cup_{k \le n} F_k$. We have that $\lim_{n \to \infty} \mu(F^* \setminus F_n^*) = 0$ and that $\m(F_n^*) < \infty$, since the sequence $\{ F_n \}$ is such that $\m(F_n) < \infty$ for any $n \in \N$.

Now we take any $\varepsilon > 0$. The fact that $\mu$ is absolutely continuous with respect to $\m$ ensures that there exists $\delta > 0$ such that whenever $E \in \Sigma$ with $\m(E) \le \delta$ it holds $\mu(E) \le \frac{1}{2} \varepsilon$. Hence let $n \in \N$ such that $\mu(F^* \setminus F^*_n) \le \frac{1}{2} \varepsilon$.

Therefore, if $F \in \Sigma$ with $\m(F \cap F^*) \le \delta$,
\[  \mu (F) = \mu (F \cap F^*_n) + \mu(F \cap (F^* \setminus F^*_n)) + \mu (F \setminus F^*) \le \epsilon,  \]
since $\mu(F \cap F^*_n) \le \frac{1}{2} \varepsilon$ by the absolutely continuity of $\mu$, $\mu(F \cap (F^* \setminus F^*_n)) \le \mu (F^* \setminus F^*_n) \le \frac{1}{2} \varepsilon$ and $\mu(F \setminus F^*) = 0$ by \eqref{Gzero}. As the choice of $\varepsilon$ is arbitrary, this proves that $\mu$ is truly continuous.

Finally, let $\m$ be a $\sigma$-finite measure, $\{X_n\}_{n \in \N}$ a non-decreasing sequence of sets of finite measure covering $\X$, and $\mu$ absolutely continuous with respect to $\m$. For any $E \in \Sigma$ such that $\mu(E) > 0$, we have that $\lim_{n \to \infty} \mu (E \cap X_n) > 0$, which means that there exists a $n \in \N$ with $\mu(E \cap X_n) > 0$. This means that $\mu$ satisfies condition $ii)$, and so it is truly continuous.
\end{proof}

\begin{thm}[Radon-Nikodym Theorem]\label{thm:RadonNikodym}
Let $(\X, \dis, \m)$ be a Polish space equipped with a quasi-Radon measure, and $\mu$ be a measure on $\X$ which is absolutely continuous with respect to $\m$. There exists a measurable function $f$ on $\X$ such that for any $B \in \mathcal{B}(\X)$ it holds 
\[\mu(B) = \int_B f \, \de \m.\]
\end{thm}

\begin{proof}
We refer to \cite[Volume 2, Section 232]{Fremlin} for a proof of the Radon-Nikodym theorem for measures $\mu$ which are truly continuous with respect to $\m$. Now, since in this setting Proposition \ref{pr:radon} ensures that the measure $\m$ is $\sigma$-finite, we can conclude just applying Proposition \ref{trac}.
\end{proof}


\subsubsection{ Convergence of quasi-Radon measures}\label{sec:Con_meas}

For a Polish space $(X,\dis)$ and a set $S\subset X$ we say that $U \in \mathcal{B}(X)$ is a neighborhood of $S$, if there exists an open $V\in \mathcal{B}(X)$, such that $S\subset V \subset U$. We write $\mathcal{N}_S$ for the set of all the neighborhoods of $S$ in $\X$. 
Let us fix a closed set with empty interior $S\subset \X$ to introduce the following classes of measures:
\[
\begin{split}
\mathscr{P}(\X)  &:= \big\{ \m : \m \text{ is a probability measure on } \X \big\}.\\
\mathscr{M}(\X) &:= \big\{ \m : \m \text{ is a finite measure on } \X \big\}; \\
\mathscr{M}_{loc}^R(\X) &:= \big\{ \m : \m \text{ is a Radon measure on } \X \text{ s.t. } \m (B) < \infty, \forall \, B \subset \X \text{ bounded}\big\};  \\
\M_S(\X) &:= \big\{ \m :  \m \text{ is a quasi-Radon  measure on } \X,  \, S \text{ is a $\m$-null set and } 
 \\& \;\;\;\;\;\;\;\;\;\;\;\;\;\ \m|_{X\setminus U} \in \mathscr{M}_{loc}^R(\X), \text{ for every }  U\in \mathcal{N}_S \big\}.
\end{split} 
 \]


The next class of measures is of central importance in our work,
\[
\begin{split}
\mathscr{M}^{qR}(\X) &:= \big\{ \m :  \m \text{ is a quasi-Radon  measure on } \X  \text{ for which there exists } S \subset \X\\ & \qquad \quad\,\,\, \text{ closed  with the property that } \, \m(S) = 0 \text{ and } \m \in \M_{S}(\X) \big\}.
\end{split}
\]

Notice that we have the following chain of inclusions: $\mathscr{P}(\X)\subset \M(\X)\subset\M_{loc}^{R} (\X) \subset\M^{qR}(\X)$.\\

The effective study of quasi-Radon measures will require us to monitor their singularities.  Intuitively said, given a closed set $S \subset \X$ with empty interior, in the definition above we isolate the set of singular points of a quasi-Radon measure inside $S$. Thus one should regard $\M_{S}(X)$ as the set of quasi-Radon measures which are locally finite and concentrated in $X \setminus S$. Observe that the effectively locally finiteness implies that all singular sets $S$ of quasi-Radon measures have empty interior, that is, $\M_{S}(\X) = \emptyset$ if  $\text{int}(S)\not = \emptyset$. The locally finiteness guarantees as well that $S$ is nowhere dense.
Moreover, Proposition \ref{pr:radon} proves that for every $\m\in\M^{qR}(\X)$ there exists a singular set $S_\m \subset \X$, closed with empty interior, providing that $\m\in\M_{S_\m}(\X)$. Finally, note that, in particular, $\mathscr{M}_{loc}^R(X) \subset \M_\emptyset^{qR}(X)$.\\

Let us now introduce the following sets of functions
\[
\begin{split}
C_{bs}(\X) &:=  \big\{ \text{bounded continuous functions with bounded support on } \X\big\}, \\
C_b(\X) &:=\big\{ \text{bounded continuous functions on } \X\big\},  \\
C_S(\X) &:= \big\{ \text{continuous functions on $\X$ which vanish on some neighborhood of } S\big\},
\end{split}
\]
where $S$ is a closed set with empty interior, and proceed to define  a convergence on $\M_S(\X)$ in duality with functions in $C_{bs}(\X) \cap C_S(\X)$. In detail, we say that
\begin{dfn}[Weak convergence for quasi-Radon measures]
The sequence of measures $\{\m_n\}_{n\in\N} \subset \M_S(\X)$  converges weakly to $\m_\infty\in \M_S(\X)$,  written $\m_n \weakto \m_\infty$, if
\begin{eqnarray}\label{def:wconv}
\lim_{n\to\infty} \int f \m_n = \int f \m_\infty \qquad \text{ for every } f \in C_{bs}(\X) \cap C_S(\X). 
\end{eqnarray}
\end{dfn}

We wish to emphasize that many useful properties, which enjoy Radon measures, are not necessarily valid, certainly not a priori, in the setting of quasi-Radon measures. An example is presented in the Lemma below, the proof of which can be found in  \cite[ Theorem 3.14]{Rudin}.
\begin{lemma}\label{lemma:dens}
Let $(\X, \dis)$ be a complete and separable metric space equipped with a Radon measure $\mm$ and let $\mathit{C}_b(\X)$ be the set of continuous and bounded functions on $\X$. Then $\mathit{C}_b(\X) \cap L^1(\mm)$ is dense in $L^1(\mm)$.
\end{lemma}

The following proposition substantiates our choice of convergence.
\begin{prop}
Let $(\X, \dis)$ be a Polish space, $S \subset \X$ be a closed set with empty interior, and $\mu, \nu \in \M_S(\X)$ two quasi-Radon measures on $\X$ such that $\int f \, \de \mu = \int f \, \de \nu$, for every function $f \in \mathit{C}_{bs} (\X)\cap C_S(\X)$. Then $\mu = \nu$.
\end{prop}
\begin{proof}
According to \cite[Volume 4, Proposition 415I]{Fremlin}, if $\m,\mathfrak{n}\in \M^{qR}(\X)$ are such that $\int f \, \de \m = \int f \, \de \mathfrak{n}$, for every function $f \in  C_b(\X) \cap L^1(\m) \cap L^1(\mathfrak{n})$, then $\m = \mathfrak{n}$. In particular this is valid for measures $\m,\mathfrak{n}\in\M_S(\X) \subset \M^{qR}(\X)$. The conclusion is attained using an approximating argument.

Let $x_0\in \X\setminus S$ and, for any $n \in \N$, consider a sequence of Lipschitz functions $g_n \colon \X\to[0,1]$ with the property that
\begin{equation}
g_n = \begin{cases} 
1 \quad \text{on } B_{2^n}(x_0) \cap \{x \in \X\, : \, \dis(x,S)\geq 2^{-n}\},\\
0 \quad \text{on } \X \setminus B_{2^{n+1}}(x_0) \cap \{x \in \X\, : \, \dis(x,S)\leq 2^{-(n+1)}\}.\end{cases}
\end{equation} 
Now, for every $f \in  C_b(\X) \cap L^1(\m) \cap L^1(\mathfrak n)$, the sequence $f_n := g_n f$ is such that
\begin{eqnarray*}
\{f_n\}_{n\in\N}\subset {C}_{bs} (\X)\cap C_S(\X), \quad \lim_{n\to\infty}f_n{=}f \quad \m,\mathfrak n \text{-a.e.},\text{ and } \abs{f_n}\leq \abs{f}, 
\end{eqnarray*}
since $\m(S)=\mathfrak n(S)=0$. We can then conclude applying the dominated convergence theorem. 
%
%
\end{proof}
 
Our definition of weak convergence for quasi-Radon measures turns out to be well-fitted for our purposes. Indeed, we have tailored it precisely with this goal. So let us then conclude this Subsection giving some observations regarding the corresponding topology.
 
\begin{rmk}\label{re:weak_narrow}
\begin{itemize} 
\item[i)] For our purposes, we would like to have a disposal a notion of convergence for quasi-Radon measures without making any a priori assumption on the uniformity of singular sets. However, this seems out of reach: indeed, without having any control on the singular sets of a given sequences, we would be able to generate unfavorable limiting singular sets and thus, for instance, obtain that  $C_{S_\infty}(\X) = \{f \equiv 0\}$. In this case, the weak convergence is trivial. As an example, consider a dense and countable collection of points $P = \{p_m\}_{m\in \N}\subset \X $ in a complete and separable space, and non-atomic measures $\nu_n\subset \M^{qR}(\X)$ such that for any neighborhood $U^n \subset \X$ of the set of the first $n$-points $P_n:=\{p_1, \dots , p_n\}$, $\nu_n(U^n) = \infty$, while $\nu_n(\X\setminus U^n) < \infty$. Letting $n\to\infty$, we would expect a limit measure having $P$ as a singular set but, for the reason given above, convergence defined against any meaningful subclass of continuous functions turns out to be trivial. Furthermore, note that such a limit measure would fall outside the realm of quasi-Radon measures.


\item[ii)] \textit{Consistency.} Let us underline  that by considering $S = \emptyset$ and by restricting the topology to $\M_{loc}^R(\X)$ the above definition coincides with the weak$^\ast$ topology (induced in duality with $C_{bs}(\X)$); by further restricting the topology to $\M(X)$, the weak topology agrees with the narrow topology (defined in duality with $C_b(\X)$). 
\end{itemize}
\end{rmk}


\subsection{Pointed generalized metric measure spaces and their convergence}
\subsubsection{Metric spaces equipped with quasi-Radon measures}
In the following we say that $(\X, \dis, \mm)$ is a \emph{metric measure space} if $(\X, \dis)$ is a Polish space equipped with a quasi-Radon measure $\mm$. We will refer to a \emph{generalized metric measure space} meaning a structure $(\X,\dis,\m, \mathcal C)$ where:
\begin{itemize}
\item $(\X,\dis)$ is a complete separable length metric space,
\item $\m \in \M^{qR}(\X)$ is a quasi-Radon measure, $\m \not= 0$,
\item $\mathcal C \subset \X$ is a closed set with empty interior and $\mm(\mathcal C) = 0$.
\end{itemize}
A \emph{pointed generalized metric measure space} is then the structure $(\X,\dis,\m,  \mathcal C, p)$ consisting of a generalized metric measure space with a distinguished point $p\in\supp(\m) \subset \X$. \\
Two generalized metric measure spaces $(\X_i,\dis_i,\m_i, \mathcal C_i)$, $i = 1, 2$ are called \emph{isomorphic} if there exists
\[
\text{an isometric embedding } i \colon \supp(\mm_1) \to \X_2 \text{ such that } i(\mathcal C_1) = \mathcal C_2 \text{ and } i_\sharp \mm_1 = \mm_2
\]
and, in the case of pointed metric measure spaces $(\X_i,\dis_i,\m_i, \mathcal C_i, p_i)$, $i = 1, 2$, we further require that $i(p_1) = p_2$. Any such $i$ is called \emph{isomorphism} from $\X_1$ and $\X_2$.\\
We denote by $\mathbb X := [\X, \d, \mm, \mathcal C, p]$ the equivalence class of the given pointed generalized metric measure space $(\X, \d, \mm, \mathcal C, p)$ and by $\mathcal M^{qR}$ the collection of all equivalence classes of pointed generalized metric measure spaces.\\
In particular, the portion of the space outside the support of the measure can be neglected since $(\X, \d, \mm, \mathcal C)$ (resp. $(\X, \d, \mm, \mathcal C, p)$) is isomorphic to $(\supp(\mm), \d, \mm, \mathcal C)$ (resp. $(\supp(\mm), \d, \mm, \mathcal C, p)$). Hence, we will assume that $\supp(\mm)=\X$, except when considering the associated \emph{$k$-th cuts, $\mathbb{X}^k$}, of a metric measure space, which we now turn to define.

For a quasi-Radon measure $\mm\in \M^{qR}(X)$, let $\mathcal{S}_{\m}\subset \X$ be  the \emph{$\m$-singular set}, or singular set in short, namely the set of all points in $\X$ for which every open neighborhood has infinite measure
\begin{equation}\label{def:Sm}
\mathcal{S}_{\m}:=\big\{x \in \X\, : \, \m(U) = \infty \text{ for every open neighborhood } U \text{ of } x\big\}.
\end{equation}
Recall that from Proposition \ref{pr:radon} we have that $\mathcal{S}_{\m}$ is a closed set with $\mm(\mathcal{S}_{\m}) = 0$. Moreover $\mathcal{S}_{\m} = \emptyset$ if and only if the measure $\mm$ is Radon. In particular, to any metric measure space $(\X, \dis,\mm)$ we can associate a generalized metric measure space in a canonical way by considering $(\X, \dis, \mm, \mathcal S_\mm)$. Now we fix once and for all a cut-off Lipschitz function $f_{\mathsf{cut}} \colon [0,\infty]\to [0,1]$ such that
\[\begin{cases}
\,f_{\mathsf{cut}}(x) = 1 \quad &\text{ for  } 0\leq x \leq 1, \\ 
\,f_{\mathsf{cut}}(x)=0\quad &\text{ for }   2\leq x
\end{cases}\]
and for $k\in\N$ we define the $k$-th cut of $\mathbb X$ as the generalized metric measure space $\mathbb X^k := (\X,\dis,\m^k, \mathcal C, p)$ where the measure is given by
\begin{eqnarray}\label{eq:k-cut}
\m^k:=f^k \, \m, \text{ where } \,
f^k(x):=
\begin{cases} 
 f_{\mathsf{cut}}( \dis(x,p) 2^{-k}  )\big(1-f_{\mathsf{cut}}(\dis(x,\mathcal{S}_\m)2^k )\big) \, &\text{ if } \Sin_\m\not= \emptyset, \label{eq:cut}\\
f_{\mathsf{cut}}( \dis(x,p) 2^{-k}  )  &\text{ if }  \Sin_\m= \emptyset. \label{eq:fk}
\end{cases}
\end{eqnarray}
Intuitively,  the $k$-th cut $(\X,\dis,\m^k, \mathcal C, p)$ \emph{resembles} more $\mathbb{X}$ as $k$ grows (see Remark \ref{rk:krw} below). 

\begin{rmk}[Regularity of the measure $\mm$]\label{ass:reg}
We point out that since we are considering metric measure spaces $(\X, \dis, \mm)$ endowed with measures $\m \in \M^{qR}(\X)$, it holds that
\begin{itemize} 
\item $\mm^k(X) < \infty $ for any $k \in \N$, and that
\item there exists a $\tilde k \in \N$ such that  for any $k \ge \tilde k$ it holds $\mm^k(\X) > 0$.
\end{itemize}
We say that $(\X, \dis, \mm)$ is a metric measure space with $\mm$-regularity parameter $\tilde k$ if the aforementioned condition is satisfied for $\tilde k \in \N$.
\end{rmk}

Finally, for a metric measure space $(\X, \dis, \mm)$, we define its $k$-th $\m$-regular set, or \emph{$k$-regular set} in short, as 
\begin{equation}\label{def:Rk}
\mathcal{R}^k := B_{2^{k+1}}(p) \setminus \mathcal{N}_{2^{-(k+1)}}(\mathcal{S}_{\m}), \quad \text{ for any } k \in \N,
\end{equation}
where $ \mathcal{N}_{2^{-(k+1)}}(\mathcal{S}_{\m}) := \cup_{x \in \mathcal{S}_{\m}} B_{2^{-(k+1)}} (x)$. Observe that $\mm|_{\mathcal{R}^k}$ is a finite measure and that actually $\text{supp}(\m^k) = \mathcal{R}^k$.

\subsubsection{Convergence of pointed metric measure spaces}
First of all, we recall what is the intrinsic Kantorovich-Rubinstein-Wasserstein ($\mathsf{iKRW}$, in short) distance between two metric measure spaces of finite mass. For this aim, we start fixing a cost function ${\mathsf c}$, that is, 
\begin{equation}\label{def:cost}
{\mathsf c} \in C( [0,\infty]) \text{ is non-constant and concave with } {\mathsf c}(0)=0 \text{ and }\lim_{d \to \infty} {\mathsf c}(d) < \infty
\end{equation}
 (e.g., ${\mathsf c}(d) = \tanh(d)$ or ${\mathsf c}(d) = d \wedge 1$). Then the $\mathsf{iKRW}$-distance between two probability measures $\m,\n \in\mathscr{P}(\X)$ on a complete and separable metric space $(\X, \dis)$ is given by
\begin{equation}
W_{\mathsf c}(\m,\mathfrak{n}):=\inf_{\gamma \in \mathsf{Adm}(\m,\mathfrak{n})} \int_{\X\times \X} {\mathsf c}(\dis(x,y)) \, \d\gamma(x,y).
\end{equation}
Observe that the distance $W_{\mathsf c}$ allows us to deal with all measures in $\mathscr{P}(\X)$, rather than with the ones in the restricted set $\mathscr{P}_2(\X)$. Moreover, regardless of the choice of $\mathsf c$ as in \eqref{def:cost}, $(\mathscr{P}(\X),W_{\mathsf c})$ is a complete and separable metric space and  the convergence with respect to the weak topology of probability measures is equivalent to the convergence provided by the $W_{\mathsf c}$-distance (see \cite[Chapter 6]{Villani09});  the last claim is a consequence of the fact that $c \circ \dis$ defines a bounded complete distance on $\X$, whose induced topology coincides with the one induced by $\dis$.\\

In the same spirit as Sturm's $\mathbb {D}$ distance, the $\mathsf{iKRW}$-distance is used to define  an  intrinsic  complete separable distance $\mathsf{d}_{\mathsf{iKRW}}^{fm}$ between pointed metric measure spaces with finite mass \cite{GigMonSav}.  Let $\mathbb{X}_1:=(\X_1,\dis_1,\m_1, \mathcal C_1, p_1),$  $\mathbb{X}_2 := (\X_2,\dis_2,\m_2, \mathcal C_2, p_2)\in \mathcal{M}^{qR} $ be generalized metric measure spaces with finite mass, then we set
\begin{equation}\label{def:iKRW}
\begin{split}
\mathsf{d}_{\mathsf{iKRW}}^{fm}&(\mathbb{X}_1,\mathbb{X}_2) := \\  & \left | \log\bigg(\frac{\m_1(\X_1)}{\m_2(\X_2)}\bigg) \right | + \inf \Big  \{ \dis\big (i_1(p_1),i_2(p_2) \big ) + { \dis_H \big (i_1(\mathcal{C}_{1}),i_2(\mathcal{C}_{2}) \big )}+ W_{\mathsf c}\big ((i_1)_\sharp \bar{\m}_1, (i_2)_\sharp  \bar{\m}_2 \big )  \Big\},
\end{split}
\end{equation} 
where the infimum is taken over all isometric embeddings $i_j \colon (\X_j,\dis_j) \to (\X,\dis)$ onto a  complete separable metric space,  $\bar{\m}_j:= \frac{\m_j}{\m_j(\X_j)}$ is a normalization of the measure $\m_j$, for $j\in\{1,2\}$ and $\dis_H$ is the Hausdorff distance between the two closed sets $i_1(\mathcal{C}_{1})$ and $i_2(\mathcal{C}_{2})$. In the following we set $\d_H(\emptyset, A) := +\infty$ if $A \neq \emptyset$ while $\d_H(\emptyset, \emptyset) := 0$.\\

Notice that the distance $\mathsf{d}_{\mathsf{iKRW}}^{fm}$ is defined only in the case in which the total mass of the two measures $\m_1$ and $\m_2$ is finite (and strictly positive). Therefore, in order to define a distance between two generalized metric measure spaces in $\mathcal{M}^{qR}$, we recover the space making use of the $k$-cuts and we sum up the contributions given by the $\mathsf{d}_{\mathsf{iKRW}}^{fm}$-distance between them.

In particular, we need the mass of the $k$-cuts to be strictly positive: for that purpose, given any  $\bar k \in \N$, we introduce the following class of spaces
\[
\mathcal{M}^{qR}_{\bar k} := \Big\{ (\X, \dis, \m, \mathcal C, p) \in \mathcal{M}^{qR} : \m^{\bar k}(\X) > 0 \Big\}
\]

Let us observe that for any finite family of generalized metric measure spaces in $\mathcal{M}^{qR}$, there exists a $\bar k \in \N$ such that the whole family is contained in $\mathcal{M}^{qR}_{\bar k}$ (in particular, it is sufficient to take $\bar k := \max \, \tilde k_i$, where  $\tilde k_i$ is the regularity parameter of the $i$-th space). Nevertheless, for a sequence in $\mathcal{M}^{qR}$ it is necessary to assume the existence of a common regularity parameter  in order to introduce a meaningful distance. Hence, in the following, we will restrict ourself to the class $\mathcal{M}^{qR}_{\bar k}$ for some $\bar k \in \N$.


\begin{dfn}[Intrinsic pointed Kantorovich-Rubinstein-Wasserstein distance]\label{def:dis}
For any couple of metric measure spaces $\mathbb{X}_i:=(\X_i,\dis_i,\m_i, \mathcal C_i, p_i)\in \mathcal{M}_{\bar k}^{qR}$, $i\in\{1,2\}$, $\bar k \in \N$, we define the  pointed $\mathsf{iKRW}$-distance as
\begin{eqnarray*}
\dis_{\mathsf{iKRW}}(\mathbb{X}_1,\mathbb{X}_2):= \displaystyle \sum_{k \ge \bar k} \dfrac{1}{2^k} \min \Big\{ 1, \dis_{\mathsf{iKRW}}^{fm}(\mathbb{X}_1^k,\mathbb{X}_2^k) \Big\},
\end{eqnarray*} 
where $\mathbb{X}_i^k = (\X_i, \dis_i, \m_i^k, \mathcal C_i, p_i)$ is the $k$-th cut of $\mathbb{X}_i$,  for $i\in\{1,2\}$.
\end{dfn}

Notice that the distance $\dis_{\mathsf{iKRW}}$ depends on the common regularity parameter $\bar k$, but we drop this dependence, since it will be clear from the context.

\begin{dfn}[Converging sequence of pointed generalized metric measure spaces]\label{de:intr}
We say that the sequence of pointed generalized metric measure spaces $\{\mathbb{X}_n\}_{n\in \N } \subset \mathcal{M}^{qR}_{\bar k}$, for some $\bar k \in \N$, is $\mathsf{iKRW}$-converging to $\mathbb{X}_\infty \in \mathcal{M}^{qR}_{\bar k}$ if 
\[\lim_{n\to \infty}\dis_{\mathsf{iKRW}}(\mathbb{X}_n,\mathbb{X}_\infty)=0.\]
\end{dfn}

Observe that the fact that $\dis_{\mathsf{iKRW}}^{fm}$ is a distance function guarantees that also $\dis_{\mathsf{iKRW}} \colon  \mathcal{M}_{\bar k}^{qR}  \to \R^+ \cup \{0\}$  defines a finite distance function.
\begin{rmk}\label{rk:krw} Directly from the definitions of $\dis_{\mathsf{iKRW}}$ and $\dis_{\mathsf{iKRW}}^{fm}$, it follows that
\begin{equation}\label{eq:convfm}
 \lim_{n\to \infty}\dis_{\mathsf{iKRW}}(\mathbb{X}_n,\mathbb{X}_\infty)=0 \quad \text{if and only if} \quad \lim_{n\to \infty} \dis_{\mathsf{iKRW}}^{fm}(\mathbb{X}_n^k,\mathbb{X}_\infty^k)=0 \text{ for every } k \ge \bar k,
 \end{equation}
where $\bar k$ is the common regularity parameter associated to the converging sequence.
\end{rmk}

In the next result we prove an \emph{extrinsic approach to convergence}. From now on we assume that the generalized metric measure space $(\X,\dis,\m, \mathcal C)$ is the canonical one associated to $(\X,\dis,\m)$, namely $\mathcal C = \mathcal S_\mm$ is the $\mm$-singular set.
\begin{prop}\label{pr:krw}
Let $\{\mathbb{X}_n\}_{n\in \N \cup \{\infty\}} \subset \mathcal{M}^{qR}_{\bar k}$, $\mathbb{X}_n = (\X_n,\dis_n,\m_n, \mathcal S_{\mm_n}, p_n)$ be a sequence of pointed generalized metric measure spaces, $\bar k \in \N$. Then the following statements are equivalent:
\begin{itemize}
\item[{i)}] $\lim_{n\to\infty} \dis_{\mathsf{iKRW}}(\mathbb{X}_n,\mathbb{X}_\infty)=0$,
\item[{ii)}] there exists a complete and separable metric space $(Z,\dis_Z)$, and isometric embeddings $i_n\colon \X_n\to Z$, $n \in \N$, for which 
\begin{equation}\label{eq:realiz}
\begin{split}
\left | \log\bigg(\frac{\m_n^k(\X_n)}{\m_\infty^k(\X_\infty)}\bigg) \right | + \dis_Z&\big (i_n(p_n), i_\infty(p_\infty) \big ) + \big(\dis_Z\big)_H \big (i_n(\mathcal{S}_{\mm_n}),i_\infty(\mathcal{S}_{\mm_\infty})\big)\\ &+ W_c\big((i_n)_\sharp \bar{\m}_n^k, (i_\infty)_\sharp  \bar{\m}_\infty^k \big )   \overset{n \to \infty}{\to} 0,
\end{split}
\end{equation}
for any $k \ge \bar k$.
\end{itemize}

Such a space and embeddings  $\big( (Z,\dis_Z),\{i_n\}_{n\in\N}\big )$ are referred to as an \emph{effective realization for the convergence} of  $\{\mathbb{X}_n\}_{n \in \N}$ to $\mathbb{X}_\infty$.
\end{prop}

\begin{proof}
$i) \Rightarrow ii)$ We start assuming that $\dint(\mathbb{X}_n,\mathbb{X}_\infty) \to 0$. In this case, the metric space $(Z,\dis_Z)$ as well as the isometric embeddings $\{i_n\}_{n \in \N}$ are constructed relying on a  twofold gluing argument. Roughly speaking, the strategy is the following: for any fixed $k \ge \bar k$ we use a ``gluing'' procedure to construct a common space $Z^k$ equipped with the metric that makes all the $k$-th cuts $\{\mathbb{X}_n^k\}_{n\in\N\cup\{\infty\}}$ be isometrically embedded. Next, we show that a certain compatibility condition holds between the spaces $\{Z^k\}_{k\in \mathbb{N}}$: this allows us to ``glue'' one more time, and obtain the desired common complete and separable metric space $(Z,\dis_Z)$ in which we will embed  the sequence $\{\mathbb{X}_n\}_{n\in\N\cup\{\infty\}}$. In the following we present the detailed argument, which is a suitable adaptation of \cite[Theorem 3.15]{GigMonSav}.\\
For every fixed $k \ge \bar k$, \eqref{eq:convfm} ensures the existence of a complete and separable metric spaces $\{(Z_n^k,\dis_{Z_n^k})\}_{n \in \N}$, and of two sequences of isometric embeddings $i_n^k \colon \X_n^k \to Z_n^k$ and $i_{\infty,n}^k\colon\X_\infty^k \to Z_n^k$, $n \in \N$, with the property that 
\begin{equation}\label{eq:extrinsic}
\begin{split}
\left | \log \bigg(\frac{\m_n^k(\X_n)}{\m_\infty^k(\X_\infty)}\bigg) \right | &+ \dis_{Z_n^k}\left (i_n^k(p_n),i_{\infty,n}^k(p_\infty) \right ) + \left(\dis_{Z_n^k}\right)_H\left ({i}_n^k(\mathcal S_{\mm_n}),{i}_{\infty, n}^k(\mathcal S_{\mm_\infty}) \right ) \\&+ W_c\left ((i_n^k)_\sharp \bar{\m}_n^k, (i_{\infty,n}^k)_\sharp  \bar{\m}_{\infty}^k \right )   \overset{n \to \infty}{\to} 0.
\end{split}
\end{equation}
We then define the set $Z^k=\sqcup_{n\in \N} Z_n^k$ and the function $\dis_{Z^k} \colon Z^k\times Z^k \to [0,\infty)$ by setting
\begin{eqnarray*}
\dis_{Z^k}(x,y) {:=}
\begin{cases}
\, \dis_{Z_n^k}(x,y) & \text{ if } (x,y) \in Z_n^k \times Z_n^k, \; \exists n \in \N,    \\
 \inf\limits_{w\in X_\infty^k} \dis_{Z_n^k}\big(x,i_{\infty,n}^k(w)\big) +\dis_{Z_m^k}\big(i_{\infty,m}^k(w),y\big)  &\text{ if } (x,y) \in Z_n^k \times Z_m^k, \; \exists n \neq m.
\end{cases}
\end{eqnarray*}
Thus, we can define an equivalence relation $\sim$ on $ Z^k$ saying that  $v \sim w$ if and only if $\dis_{Z^k}(v,w)=0$, for  $v,w\in Z^k$: we quotient $Z^k$ by this relation and we complete it. We denote by $\tilde{Z}^k$ the resulting space. Note that $\dis_{Z^k}$ canonically  induces a distance function on $\tilde{Z}^k\times \tilde{Z}^k$, which we still denote by $\dis_{Z^k}$, and that the operations made so far preserve the separability of the space. Thus, the pair $(\tilde{Z}^k,\dis_{Z^k})$ is a complete and separable metric space. By construction, for $n\in \N$, the composition 
\begin{equation*}
\mathsf{i}_{n}^k :{=}p^k \circ j_n^k\circ i_{n}^k \, \colon \, \X_n^k \to \tilde{Z}^k
\end{equation*}
is an isometric embedding, where $j_n^k \colon Z_n^k \to Z^k$ is the canonical inclusion and $p^k \colon Z^k \to  \tilde{Z}^k$ the projection map. Moreover, the fact that the set $j_n^k(i_{\infty,n}^k(\X_\infty))$ is identified under the equivalence relation with $j_m^k (i_{\infty,m}^k(\X_\infty))$, for every $m,n\in\N$, implies that the maps 
\begin{eqnarray*}
p^k \circ j_n^k\circ i_{\infty,n}^k\colon\X_\infty^k\to \tilde{Z}^k \qquad \text{and} \qquad
p^k \circ j_m^k\circ i_{\infty,m}^k\colon\X_\infty^k\to \tilde{Z}^k
\end{eqnarray*}
coincide, for every $n,m\in\N$. In this manner, we see that also $ p^k \circ j_m^k\circ i_{\infty,m}^k: X_\infty^k \to \tilde{Z}^k$ is an isometric embedding, which is independent of $m$. Let us denote it by $\mathsf{i}_\infty^k$. Convergence \eqref{eq:extrinsic} yields the convergence
\begin{equation} \label{eq:conv_Zk}
 \left | \log\left (\frac{\m_n^k(\X_n)}{\m_\infty^k(\X_\infty)}\right ) \right | + \dis_{Z_n^k}\left (\mathsf{i}_n^k(p_n),\mathsf{i}_{\infty}^k(p_\infty) \right ) + \left(\dis_{Z_n^k}\right)_H\left (\mathsf{i}_n^k(\mathcal S_{\mm_n}),\mathsf{i}_{\infty}^k(\mathcal S_{\mm_\infty}) \right ) \overset{n\to\infty}{\to} 0.
 \end{equation}
To finish the first step of the argument we note that the pushforward of a coupling under the map $ (p^k\circ j_n^k)\times (p^k\circ j_n^k):Z_n^k \times Z_n^k \to \tilde{Z}^k\times \tilde{Z}^k$,  is again a coupling between the pushforward of the original marginal measures, namely 
\[
\text{if } \, \pi\in \mathsf{Adm}\big( (i_n^k)_\sharp \bar{\m}_n^k ,(i_{\infty, n}^k)_\sharp \bar{\m}_{\infty}^k\big),  \text{then } \, \tilde{\pi}:{=} ( (p^k\circ j_n^k)^2)_\sharp \pi \in \mathsf{Adm}\big( (\mathsf{i}_n^k)_\sharp \bar{\m}_n^k ,(\mathsf{i}_\infty^k)_\sharp \bar{\m}_\infty^k\big).
\]
Therefore, if we choose $\pi \in \mathsf{Opt}\big( (i_n^k)_\sharp \bar{\m}_n^k ,(i_\infty^k)_\sharp \bar{\m}_\infty^k\big)$, we get
\begin{equation*}
W_c^{\tilde{Z}^k}\Big((\mathsf{i}_n^k)_\sharp \bar{\m}_n^k , (\mathsf{i}_\infty^k)_\sharp \bar{\m}_\infty^k\Big) \leq W_c^{Z_n^k}\Big( (i_n^k)_\sharp \bar{\m}_n^k ,(i_\infty^k)_\sharp \bar{\m}_\infty^k\Big),
\end{equation*}
since $p^k\circ j_n^k : Z_n^k \to \tilde{Z}^k$ is an isometry. Jointly with the last term in \eqref{eq:extrinsic}, this inequality implies the convergence $(\mathsf{i}_n^k)_\sharp \bar{\m}_n^k \to (\mathsf{i}_\infty^k)_\sharp \bar{\m}_\infty^k$ in $\big(\PX(\tilde{Z}^k),W_c^{\tilde{Z}^k}\big)$. We have hereby shown the existence of a complete separable metric space, $(Z^k,\dis_{Z^k})$, and of a sequence of isometric embeddings,  $i_n^k \colon \X_n\to Z^k$ for every $n\in\N\cup \{\infty\}$, which provide  a realization of the convergence $\mathbb{X}_n^k\to \mathbb{X}_\infty^k$ for any $k \ge \bar k$. \\

For the second part of the argument, first of all we prove that for any $\bar k \le j < k-1$, the embeddings $i_n^k \colon \X_n\to Z^k$ serve as an effective realization for the convergence $\mathbb{X}_n^{j}\to \mathbb{X}_\infty^{j}$. To that purpose, let us consider the function
\begin{eqnarray*}
\begin{split}
g_n^{j,k} \colon Z^k &\to \, [0,1] \\
y\;\,&\mapsto \,
\begin{cases} 
 f_{\mathsf{cut}}\left( \, \dis_{Z^k}\left(y,i_n^k(p_n)\right) \; 2^{-(j+1)}\right)\left(1-f_{\mathsf{cut}}\left(\dis_{Z^k}\left(y,i_n^k \left(\mathcal{S}_{\m_n}\right)\right)2^{j+1} \right)\right) \, &\text{ if } \Sin_{\m_n}\not= \emptyset, \label{eq:cut}\\
f_{\mathsf{cut}}\left( \, \dis_{Z^k}\left(y,i_n^k(p_n)\right) \; 2^{-(j+1)}\right)  &\text{ if }  \Sin_{\m_n}= \emptyset. \label{eq:fk}
\end{cases}
\end{split}
\end{eqnarray*}

The Lipschitz continuity of  the cut-off function $f_{\mathsf{cut}}$, together with the convergence of  $\{i_n^k(p_n)\}_{n \in \N}$  to $i_\infty^k(p_\infty)$ by \eqref{eq:extrinsic}, ensures that the sequence $f_{\mathsf{cut}}\left( \, \dis_{Z^k}\left(y,i_n^k(p_n)\right) \; 2^{-(j+1)}\right)$ is uniformly converging to $ f_{\mathsf{cut}}\,\big(\dis_{Z^k}(y,i_\infty^k(p_\infty))\;  2^{-(j+1)}\big)$ as $n \to \infty$. In the same way, the triangular inequality ensures that 
\[\displaystyle\abs{\dis_{Z^k}\left(y,i_n^k \left(\mathcal{S}_{\m_n}\right)\right) - \dis_{Z^k}\left(y,i_\infty^k \left(\mathcal{S}_{\m_\infty}\right)\right)}  \le \left(\dis_{Z^k}\right)_H \left(  i_n^k \left(\mathcal{S}_{\m_n}\right), i_\infty^k \left(\mathcal{S}_{\m_\infty}\right) \right).
\]
and the convergence \eqref{eq:extrinsic} guarantees that the sequence $f_{\mathsf{cut}}\left(\dis_{Z^k}\left(y,i_n^k \left(\mathcal{S}_{\m_n}\right)\right)2^{j+1} \right)$ uniformly converges to $f_{\mathsf{cut}}\left(\dis_{Z^k}\left(y,i_\infty^k \left(\mathcal{S}_{\m_\infty}\right)\right)2^{j+1} \right)$ as $n \to \infty$. Hence, for every $y \in Z^k$ the sequence $\{g_n^{j,k}(y)\}_{n \in \N}$ uniformly converges as $n \to \infty$ to
\[
 g^{j,k}(y) := \begin{cases} 
  f_{\mathsf{cut}}\,\left(\dis_{Z^k}(y,i_\infty^k(p_\infty))\;  2^{-(j+1)}\right)\left(1-f_{\mathsf{cut}}\left(\dis_{Z^k}\left(y,i_\infty^k \left(\mathcal{S}_{\m_\infty}\right)\right)2^{j+1} \right)\right) \, &\text{ if } \Sin_{\m_\infty}\not= \emptyset, \label{eq:cut}\\
 f_{\mathsf{cut}}\,\left(\dis_{Z^k}(y,i_\infty^k(p_\infty))\;  2^{-(j+1)}\right)  &\text{ if }  \Sin_{\m_\infty}= \emptyset.
\end{cases}
\]
This in particular implies the weak convergence of the sequence of measures
\begin{equation*} 
\big(i_n^k\big)_\sharp\bar{\m}_n^j = \big(g_n^{j,k} \circ i_n^k\big)_{\sharp}\bar{\m}^k_n \,\,\,\overset{n\to\infty}{\rightharpoonup} \,\,\, \big(i_\infty^k\big)_\sharp\bar{\m}^j = \big(g^{j,k} \circ i_\infty^k\big)_\sharp\bar{\m}^k
\end{equation*}
(note that we ask for $\bar k \le j<k-1$ so that the equalities in the above equation are accurate).
The former convergence, together with  \eqref{eq:conv_Zk} and the fact that $\left | \log\Big(\frac{\m_n^j(X_n)}{\m_\infty^j(X_\infty)}\Big) \right |\to 0$, shows that the embeddings $i_n^{k}\colon \X_n\to Z^k$ realize the convergence of the sequence of $j$-cuts, for every $\bar k \le j < k-1$.

The given argument indicates that for different values of $k$ the necessary compatibility between embeddings $\{i_n^k\}_{n\in \N \cup \{\infty\}}$ and spaces $Z^k$ exists to perform a second ``gluing'' argument: we can then construct a common complete and separable metric space $Z := \sqcup_k Z^k$ and a sequence of embeddings $i_n\colon \X_n\to Z$, $n\in\N\cup\{\infty\}$. The pair $\left (Z, \{i_n\}_{n\in\N\cup\{\infty\}}\right )$ is the desired effective realization of the $\mathsf{iKRW}$-convergence $\mathbb{X}_n\to\mathbb{X}_\infty$. We omit the details of the construction, since we can follow exactly the same lines of the one presented above.

%
%
$ii) \Rightarrow i)$ Note that the existence of an effective realization of the convergence $\mathbb{X}_n \overset{n \to \infty}{\to} \mathbb{X}_\infty$ implies that $\dint^{fm}(\mathbb{X}_n^k , \mathbb{X}_\infty^k )\overset{n \to \infty}{\to} 0$, for all $k \ge \bar k$. Then, we can conclude by using \eqref{eq:convfm}.
\end{proof}

In some situations, it would be practical to have to our disposal a metric which is not explicitly dependent  on the behavior of the $\m$-singular sets. For instance, we could gain flexibility by not asking for a control on the Hausdorff distance between $\m$-singular sets in the definition of the $\dis_{\mathsf{iKRW}}$-distance. However, as we just saw, this term is necessary to provide an extrinsic realization of the distance given an intrinsic one. Therefore, the following definition turns out to be useful.
\begin{dfn}[Extrinsic convergence]\label{de:ext}
We say that the sequence of pointed generalized metric measure spaces $\{\mathbb{X}_n\}_{n\in \N } \subset \mathcal{M}^{qR}_{\bar k}$ \emph{converges extrinsically} to $\mathbb{X}_\infty \in \mathcal{M}^{qR}_{\bar k}$, $\bar k \in \N$, if  there exists a complete and separable metric space $(Z,\dis_Z)$, and isometric embeddings $i_n\colon \X_n\to Z$, $n \in \N$, for which 
\begin{equation}\label{eq:realiz}
\begin{split}
\left | \log\bigg(\frac{\m_n^k(\X_n)}{\m_\infty^k(\X_\infty)}\bigg) \right | + \dis_Z&\big (i_n(p_n), i_\infty(p_\infty) \big ) +  W_c\big((i_n)_\sharp \bar{\m}_n^k, (i_\infty)_\sharp  \bar{\m}_\infty^k \big )   \overset{n \to \infty}{\to} 0,
\end{split}
\end{equation}
for any $k \ge \bar k$.
\end{dfn}
Note that we dropped the assumption on the Hausdorff distance between singular sets at the cost of presenting ourselves a space providing the extrinsic realization. Furthermore, by Proposition \ref{pr:krw} we know that a $\mathsf{iKRW}$-converging sequence converges also in the extrinsically manner. 

We finish the Section with some remarks. 
\begin{rmk}
\emph{Connection to Gigli-Mondino-Savare's $p\mathbb{G}_W$ distance.} In \cite{GigMonSav} the authors define a distance between $\mathbb{X}_1$ and $\mathbb{X}_2$ metric measure spaces endowed with Radon measures giving finite mass to bounded sets. Indeed, this is what inspired us to propose the definition of $\dis_{\mathsf{iKRW}}$. As a matter of fact, in this case the $\m$-singular set of metric measure spaces in such family is the empty set, and thus our definition coincides with theirs. 
\end{rmk}

\begin{rmk}
We recall that in \cite[Theorem 3.17]{GigMonSav} the authors prove that the class of all metric measure spaces equipped with Radon measures is complete with respect to the $p\mathbb{G}_W$ distance. It is worth to underline that in this context we cannot hope for a similar completeness result. The main reason is that the set all of closed sets with empty interior is not closed for the Hausdorff distance. Hence, intuitively, we cannot prevent a sequence of quasi-Radon measures from converging to a measure which is not quasi-Radon.
\end{rmk}

%
%

\section{$\CD$-condition for negative generalized dimension}
\subsection{Basic definitions and properties}\label{sec:CD1}
Let $(\X, \d)$ be a complete and separable metric space. We denote by $\mathscr{P}(\X)$ the set of all Borel probability measures on $\X$ and by $\mathscr{P}_2(\X) \subset \mathscr{P}(\X)$ the set of all probability measures with finite second moment. On the space $\mathscr{P}_2(\X)$ we introduce the 2-Wasserstein distance
\begin{equation}\label{eq:W2}
    W_2^2(\mu,\nu):=\inf_{\gamma \in \mathsf{Adm}(\mu,\nu)} \int_{\X\times \X}\dis(x,y)^2 \, \d\gamma(x,y).
\end{equation}
The infimum in \eqref{eq:W2} is always realized. The plans $\gamma\in\mathsf{Adm}(\mu,\nu)$ such that $\int \dis(x, y)^2\,\de \gamma(x, y) = W_2^2(\mu,\nu)$ are called optimal couplings, or optimal transport plans, and the set that contains them all is denoted by $\mathsf{Opt}(\mu,\nu)$. It is well known that $W_2$ is a complete and separable distance on $\PX_2(\X)$. Moreover the convergence with respect to the Wasserstein distance is characterized in the following way (cfr. \cite[Section 7.1]{AmbrosioGigliSavare08}):
\begin{equation*}
    \mu_n \overset{W_2}{\longrightarrow}\mu \quad \iff \quad \mu_n \rightharpoonup \mu\, \text{ and } \int \di(x_0,x)^2 \de \mu_n \to \int \di(x_0,x)^2 \de \mu \quad \forall x_0\in \X.
\end{equation*}
This description shows in particular that, for a sequence of probability measures with uniformly bounded support, the $W_2$-convergence is equivalent to the weak one and, consequently, to the $W_c$-convergence. Hence, in this case,  \eqref{eq:realiz} of Proposition \ref{pr:krw} remains valid when we replace the $W_c$-distance with the $W_2$-one.\\

We introduce the \emph{R\'enyi entropy} $S_{N,\m}$ for $N < 0$ with respect to the reference measure $\mm$ as the functional defined on $\mathscr{P}(\X)$ by 
\begin{equation*}
S_{N,\m}(\mu):= \begin{cases}
\displaystyle \int_{\X} \rho(x)^{\frac{N-1}{N}} d\m(x) &\qquad \text{if } \, \mu \ll  \m, \, \mu = \rho \m,\smallskip\\
+\infty &\qquad \text{otherwise},
\end{cases}
\end{equation*}
where $\rho = {\d\mu}/{\d\m}$ is the Radon-Nikodym derivative of $\mu$ with respect to $\m$, whose existence is guaranteed by Theorem \ref{thm:RadonNikodym}. 
\begin{rmk}
We point out that for $N \ge 1$ the ``standard'' R\'enyi entropy is defined as $S_{N,\m}(\rho \mm) := - \int_{\X} \rho(x)^{\frac{N-1}{N}} d\m(x)$. For $N < 0$, the minus sign is omitted to impose the convexity of the function $h(s) = s^{(N-1)/N}$. Note that, for $N \ge 1$, it suffices to define the ``standard'' R\'enyi entropy on Polish spaces equipped with Radon reference measures. In this case, under a volume growth condition on the reference measure $\m$, the functional $S_{N,\m}(\cdot)$ is lower semicontinuous with respect to the weak topology and, in particular, it is lower semicontinuous with respect to 2-Wasserstein convergence in $\PX_2(\X)$. Unfortunately, the same property is not necessarily true, for negative values of $N <0 $ and quasi-Radon reference measures $\m$. However, we prove in Proposition \ref{prop:almostlscofSN} that $S_{N,\m}(\cdot)$ is lower semicontinuous with respect to weak convergence, on the subspace
\begin{equation*}
    \PX^{\mathcal S_\m}(\X):= \{\mu\in \PX_2(\X) \,: \,S_\m \text{ is the $\m$-singular set and } \mu(\mathcal S_\m)=0\}.
\end{equation*}
\end{rmk}
In fact, we show a more general result stating that the R\'enyi entropy functional $S_{N, \n}(\nu)$ is a lower semicontinuous function of $(\n,\nu)\in \M_S(\X)\times\PX^{\mathcal S}(\X)$, where the convergence of the first coordinate is intended to be the  weak convergence of quasi-Radon measures.

In order to give the definition of curvature-dimension bounds, we need also to introduce the following distortion coefficients for $N < 0$:
\begin{equation}\label{E:sigma}
\sigma_{K,N}^{(t)}(\theta):= 
\begin{cases}
\infty, & \textrm{if}\ K\theta^{2} \leq N\pi^{2}, \crcr
\displaystyle  \frac{\sin(t\theta\sqrt{K/N})}{\sin(\theta\sqrt{K/N})} & \textrm{if}\  N\pi^{2} < K\theta^{2} <  0, \crcr
t & \textrm{if}\ 
K \theta^{2}=0,  \crcr
\displaystyle   \frac{\sinh(t\theta\sqrt{-K/N})}{\sinh(\theta\sqrt{-K/N})} & \textrm{if}\ K\theta^{2} > 0
\end{cases}
\end{equation}
and 
\begin{equation}\label{eq:deftau}
    \tau_{K, N}^{(t)}(\theta):=t^{1 / N} \sigma_{K /(N-1)}^{(t)}(\theta)^{(N-1) / N}.
\end{equation}

\noindent

\begin{dfn}
For any couple of measures $\mu_0,\mu_1 \in \mathscr{P}^{ac}(\X)$, $\mu_i = \rho_i \m$, we denote by $\pi\in \mathscr{P}(\X\times \X)$ a coupling between them, and by $T_{K,N}^t(\pi | \m)$  the functional defined by
\[
\begin{split}
T_{K,N}^{(t)}(\pi | \m)&:= \int_{\X\times \X} \Big[ \tau^{(1-t)}_{K,N} \big(\dis(x,y) \big) \rho_{0}(x)^{-\frac{1}{N}} +    \tau^{(t)}_{K,N} \big(\dis(x,y) \big) \rho_{1}(y)^{-\frac{1}{N}} \Big]   d \pi(x,y).
\end{split}
\]
\end{dfn}

We are ready to introduce the definition of metric measure space satisfying a curvature-dimension condition for negative values.
\begin{dfn}[$\CD$-condition]\label{def:curvcond}
For fixed  $K\in \R, N \in (-\infty,0)$, we say that a metric measure space $(\X, \dis, \m)$ satisfies the  \emph{$\CD(K,N)$-condition}  if, for each pair $\mu_{0}=\rho_{0} \m,\, \mu_{1}=\rho_{1} \m \in \PX_{2}^{ac}(\X)$, there exists an optimal coupling $\pi\in \mathsf{Adm}(\mu_0,\mu_1)$ and a $W_{2}$-geodesic  $\{\mu_t\}_{t \in [0, 1]} \subset \PX_{2}(\X)$ such that
\begin{align}\label{def:CD}
S_{N', \mm}(\mu_t) \leq T_{K,N'}^{(t)}(\pi | \m)
\end{align}
holds, for every $t\in [0,1]$, and every $N'\in [N, 0)$, provided that $S_{N', \m}(\mu_0)$, $S_{N', \mm}(\mu_1) < \infty$.  
\end{dfn}
\begin{rmk}\label{re:bounded}
Note that the $\CD$-inequality becomes trivial when $K < 0$ and
\[
\pi \Big(\big\{ (x, y) \in \X \times \X : \d(x, y) \ge \pi \sqrt{(N'-1)/K}  \big\}  \Big) > 0.
\]
Furthermore, observe that, if $K\geq 0$, or if $\diam(\X)<\sqrt{\pi (N-1)/K}$ when $K<0$, the coefficients $\tau_{K,N}^{(t)}(\cdot)$ are bounded.  Now notice that Jensen's inequality shows that $S_{N,\m}(\mu_0), \,S_{N,\m}(\mu_1)$ are finite, if the entropies $S_{N',\m}(\mu_0), \,S_{N',\m}(\mu_1)$ are finite for some $N' \in [N, 0)$. Therefore, observe that in this case the $\CD(K,N)$-condition guarantees that the Wasserstein geodesics along which inequality \ref{def:CD} holds are absolutely continuous with respect to $\m$.
\end{rmk}
\begin{rmk} 
Take notice that Definition \ref{def:curvcond} 
restricts the domain in which it is required to verify inequality   \eqref{def:CD} to $\mathcal{D}(S_{N',\m}) := \{\mu \,: \,S_{N',\m}(\mu)<\infty\}$. This is a common assumption made when dealing with convexity-like properties of functions which take values in the extended real numbers. Nevertheless, we wish to comment in more detail with regard to this.  

It turns out that it is not necessary to make this restriction in  the classical theory of curvature-dimension bounds for $N \ge 1$ since the R\'enyi Entropy is bounded for absolutely continuous measures on $\PX_2(\X)$, as a consequence of the fact that $\CD(K,N)$ space posses reference measures with a controlled volume growth. A proof of the finiteness of some entropy functionals, in particular of the  R\'enyi entropy, under volume growth assumptions can be found for example in \cite[Proposition E.17]{Lott-Villani09}. 
However, this is not necessarily the case when dealing with negative parameters $N$ and, therefore, we require the finiteness of the entropies at the marginal measures. 

We remark that  for  terminal marginals  with bounded supports, Definition \ref{def:curvcond} coincides with the one introduced by Ohta in \cite{Ohta16}. Indeed, if the supports of $\mu_0$ and $ \mu_1$ are bounded in $(\X, \di)$ the coefficients $\tau^{(1-t)}_{K,N}(\cdot)$ are bounded below away from $0$ on the support of any coupling $\pi \in \mathsf{Opt}(\mu_0, \mu_1)$, for fixed $0<t<1$. Thus, if for some $N'\in [N,0)$ one of the terminal measures has unbounded entropy $S_{N',\m}$, then $T_{K,N'}^{(t)}(\pi | \m) = \infty$, for any $t \in(0, 1)$, and inequality \eqref{def:CD} is always satisfied. 

Lastly, we emphasize as well that  the assumption on the entropy in Definition \ref{def:curvcond} is consistent with the standard definitions of curvature-dimension conditions for positive values of $N$, in which the requirement is not explicitly made.
\end{rmk}

We underline that, as in the case $N \ge 1$, the definition of curvature-dimension condition is invariant under standard transformations of metric measure structures. Precisely, the $\CD$-condition is stable under isomorphisms, scalings, and restrictions to convex subsets of metric measure spaces (this can be proved using the same techniques as in \cite[Proposition 1.4]{Sturm06I} and in \cite[Proposition 4.12, 4.13 and 4.15]{Sturm06II}). We also point out, that the ``hierarchy property'' of $\CD(K, N)$ spaces, with $N < 0$, remains valid. Specifically, 
\begin{prop}
If $(\X, \dis, \mm)$ satisfies the curvature-dimension condition $\CD(K, N)$ for some $K \in \R, N < 0$, then it also satisfies the curvature-dimension condition $\CD(K', N')$ for any $K' \le K$ and $N' \in [N, 0)$.
\end{prop}
\begin{proof}
The monotonicity in $N$ follows directly from Definition \ref{def:curvcond}, while the monotonicity in $K$ follows from the fact that the coefficient $\sigma_\kappa^{(t)}(\theta)$ is non-decreasing in $\kappa$ once $t$ and $\theta$ are fixed (see \cite[Remark 2.2]{BacherSturm10}).
\end{proof}
Let us conclude by recalling that the $\CD(K, N)$-condition is weaker than the $\CD(K,\infty)$-condition \cite{Ohta16} and, therefore, it follows that $\CD(K, \infty)$ spaces are $\CD(K, N)$ space for every $N < 0$.

\subsection{Examples}

In this section we present some examples of negative dimensional CD spaces, referring to \cite{Milman-NegSph, Milman-Neg, KolMilm} for other model spaces satisfying the $\CD(K, N)$-condition with $N < 0$. Moreover, we show that singular points of the reference measure in negative dimensional CD spaces can appear as inner points of geodesics. This fact motivates us to present the definitions of \emph{approximate $\CD$-condition} and \emph{$\omega$-uniform convexity}, object of Section \ref{sec:assumption}, which will enable us to deal with this kind of behavior in the proof of our Stability Theorem.

A fundamental notion in the presentation of the examples is the one of $(K,N)$-convexity of a function on a metric space, for a negative value of $N$. This definition is the natural counterpart of the one with positive $N$, and it was introduced by Ohta in \cite{Ohta16}.

\begin{dfn}[$(K,N)$-convexity]
In a metric space $(\X,\di)$, for every fixed $K \in \R$ and $N \in (-\infty,0)$, a function $f:\X\to  \bar{\R}$ is said to be $(K,N)$-convex if for every $x_0,x_1 \in \{f< +\infty\}$, with $d:=\di(x_0,x_1)<\pi\sqrt{N/K}$ when $K<0$, there exists a constant speed geodesic $\gamma$ connecting $x_0$ and $x_1$, such that
\begin{equation}\label{eq:convinequality}
    f_N(\gamma_t) \leq \sigma_{K/N}^{(1-t)}(d) f_N(x_0) + \sigma_{K/N}^{(t)}(d) f_N(x_1) \qquad \forall t\in [0,1],
\end{equation}
where $f_N(x)=e^{-f(x)/N}$.
\end{dfn}

The following result (\cite[Corollary 4.12]{Ohta16}) is used to produce  examples of $\CD(K, N)$ spaces with negative values of the generalized dimension.
\begin{prop}\label{prop:ohtaexamples}
Let $M$ be a $n$-dimensional complete Riemannian manifold with Riemannian distance $\d_g$ and Riemannian volume $\text{vol}_g$. Let us then consider a weighted volume $\mm = e^{-\psi} \text{vol}_g$, and let numbers $K_1, K_2 \in \R$, $N_2 \ge n$ and $N_1 < - N_2$ be given.

Then if $(M, \d_g, \text{vol}_g)$ satisfies the $\CD(K_2, N_2)$-condition, the weighted space $(M, \d_g, e^{-\psi} \mm)$ satisfies the $\CD(K_1+K_2, N_1 +N_2)$-condition provided that $\psi \in C^2(M)$ is $(K_1, N_1)$-convex.
\end{prop}

\begin{example}[1-dimensional models]\label{ex:nice} In the following we will denote by $ |\cdot|$ the Euclidean distance and by $\Leb^1$ the 1-dimensional Lebesgue measure.\\
{\bf (i)} For any pair of real numbers $K > 0, N < -1$ the weighted space $(\R, |\cdot|, V \Leb^1)$ with
\[
V(x) = \cosh\bigg(  x \sqrt{- \dfrac{K}{N}} \bigg)^N
\]
satisfies the curvature-dimension condition $\CD(K, N+1)$ with no singular set, i.e. $\mathcal{S}_{V \Leb^1} =\emptyset$.\\
{\bf (ii)} For any pair of real numbers $K > 0, N < -1$ also the weighted space $([0, \infty), |\cdot|, V \Leb^1)$ with
\[
V(x) = \sinh\bigg(  x \sqrt{- \dfrac{K}{N}} \bigg)^N
\]
satisfies the curvature-dimension condition $\CD(K, N+1)$ with singular set $\mathcal{S}_{V \Leb^1} = \{ 0 \}$.\\
{\bf (iii)} For any $N < -1$ the space $([0, \infty), |\cdot|, x^{N} \Leb^1)$ is a $\CD(0, N+1)$ space with singular set $\mathcal{S}_{x^{N} \Leb^1} = \{ 0 \}$.\\
{\bf (iv)} For any pair of real numbers $K < 0, N < -1$ the weighted space
\[
\Bigg(\Bigg[ - \dfrac{\pi}{2} \sqrt{\dfrac{K}{N}},  \dfrac{\pi}{2} \sqrt{\dfrac{K}{N}} \Bigg], |\cdot |,  \cos \bigg(  x \sqrt{\dfrac{K}{N}}\bigg) ^N \, \Leb^1 \Bigg)
\]
satisfies the curvature-dimension condition $\CD(K, N+1)$ with singular set  given by
\[\mathcal{S}_{\cos \big(  x \sqrt{K/N} \big)^N \, \Leb^1} = \bigg\{   - \dfrac{\pi}{2} \sqrt{\dfrac{K}{N}},  \dfrac{\pi}{2} \sqrt{\dfrac{K}{N}} \bigg\}.\]
\end{example}

Example \ref{ex:nice} provides negative dimensional $\CD$ spaces, whose set of singular points is a subset of their topological boundary. Unfortunately, this is not a general behavior and we proceed now to show this. With this goal in mind, we will rely on a modification of Proposition \ref{prop:ohtaexamples}, whose proof needs a preliminary result.

\begin{lemma}
Let $f:I\to \bar \R$ be a function on the interval $I:=[a,b]\subset \R$. Assume that there exists $c\in(a,b)\cap \{f<+\infty\}$ such that ${f}|_{[a,c]}$ and ${f}|_{[c,b]}$ are $(K,N)$-convex and for every $x_0\in[a,c]$, $x_1\in[c,b]$ it holds that
\begin{equation}\label{eq:midpointconv}
    f_N(c) \leq \sigma_{K/N}^{\big(\frac{x_1-c}{x_1-x_0}\big)}(x_1-x_0) f_N(x_0) + \sigma_{K/N}^{\big(\frac{c-x_0}{x_1-x_0}\big)}(x_1-x_0) f_N(x_1).
\end{equation}
Then $f$ is $(K,N)$-convex.
\end{lemma}

\begin{proof}
We have to prove the convexity inequality \eqref{eq:convinequality} for every $x_0,x_1\in I$ in the domain $\{f<\infty\}$. However, this holds by hypothesis, if $x_0,x_1\in [a,c]$ or $x_0,x_1\in [c,b]$, thus it is sufficient to consider the case where $x_0\in[a,c)$ and $x_1\in(c,b]$. Without loss of generality, we can assume that $x_t\in[a,c)$, then the $(K,N)$-convexity of $f|_{[a,c]}$ yields that
\begin{equation*}
    f_N(x_t)\leq \sigma_{K/N}^{\big(\frac{c-x_t}{c-x_0}\big)}(c-x_0) f_N(x_0) + \sigma_{K/N}^{\big(\frac{x_t-x_0}{c-x_0}\big)}(c-x_0) f_N(c).
\end{equation*}
Combining this last inequality with \eqref{eq:midpointconv} we obtain
\begin{equation}\label{eq:intermidiate}
\begin{split}
    f_N(x_t)\leq  \bigg[\sigma_{K/N}^{\big(\frac{c-x_t}{c-x_0}\big)}(c-x_0)+ \sigma_{K/N}^{\big(\frac{x_t-x_0}{c-x_0}\big)}(c-x_0)\,\sigma_{K/N}^{\big(\frac{x_1-c}{x_1-x_0}\big)}(x_1-x_0) \bigg] f_N(x_0) \\
    + \sigma_{K/N}^{\big(\frac{x_t-x_0}{c-x_0}\big)}(c-x_0)\,\sigma_{K/N}^{\big(\frac{c-x_0}{x_1-x_0}\big)}(x_1-x_0) f_N(x_1).
\end{split}
\end{equation}
On the other hand it is easy to realize that 
\begin{equation*}
    \sigma_{K/N}^{\big(\frac{x_t-x_0}{c-x_0}\big)}(c-x_0)\,\sigma_{K/N}^{\big(\frac{c-x_0}{x_1-x_0}\big)}(x_1-x_0) = \sigma_{K/N}^{\big(\frac{x_t-x_0}{x_1-x_0}\big)}(x_1-x_0).
\end{equation*}
As an example, we consider $K<0$. Then, it holds that
\begin{eqnarray*} 
   &\sigma_{K/N}^{\big(\frac{x_t-x_0}{c-x_0}\big)}(c-x_0)\,\sigma_{K/N}^{\big(\frac{c-x_0}{x_1-x_0}\big)}(x_1-x_0) =& \\ &\frac{\sin(\sqrt{K/N}(x_t-x_0))}{\sin(\sqrt{K/N}(c-x_0))} \cdot  \frac{\sin(\sqrt{K/N}(c-x_0))}{\sin(\sqrt{K/N}(x_1-x_0))}=\frac{\sin(\sqrt{K/N}(x_t-x_0))}{\sin(\sqrt{K/N}(x_1-x_0))}=&\\& \sigma_{K/N}^{\big(\frac{x_t-x_0}{x_1-x_0}\big)}(x_1-x_0).&
\end{eqnarray*}
Moreover, with an explicit computation, using the sum-to-product trigonometric formulas, it is also possible to prove that 
\begin{equation*}
    \sigma_{K/N}^{\big(\frac{c-x_t}{c-x_0}\big)}(c-x_0)+ \sigma_{K/N}^{\big(\frac{x_t-x_0}{c-x_0}\big)}(c-x_0)\,\sigma_{K/N}^{\big(\frac{x_1-c}{x_1-x_0}\big)}(x_1-x_0) = \sigma_{K/N}^{\big(\frac{x_1-x_t}{x_1-x_0}\big)}(x_1-x_0).
\end{equation*}
Combining the previous  trigonometric identities with inequality \eqref{eq:intermidiate} we obtain the $(K,N)$-convexity inequality.
\end{proof}

\noindent An immediate corollary follows from the fact that $f_N(-\infty)=0$.

\begin{cor}
Let $f:I\to \bar \R$ be a function on the interval $I:=[a,b]$. Assume that there exists $c\in(a,b)$ such that ${f}|_{[a,c]}$ and ${f}|_{[c,b]}$ are $(K,N)$-convex and that $f(c)=-\infty$, then $f$ is $(K,N)$-convex.
\end{cor}

Now, we present an alternative version of Proposition \ref{prop:ohtaexamples}, in which we do not need to assume regularity of the weight function $\psi$ at the price of restricting to the case $M=\R$.

\begin{prop}\label{prop:varexamples}
Let $\psi:\R\to [-\infty,\infty)$ be a $(K,N)$-convex with $N<-1$, such that $\mathcal{L}^1(\{\psi = -\infty\}) = 0$. Then the metric measure space $(\R, |\cdot|, e^{-\psi}\mathcal{L}^1)$ is a $\CD(K,N+1)$ space.
\end{prop}

\begin{proof}
In this proof we denote with $\m$ the reference measure $e^{-\psi}\mathcal{L}^1$. 
In order to prove the $\CD$-condition we fix two absolutely continuous measures $\mu_0=\rho_0\m,\mu_1=\rho_1 \m\in \PX_2(\R)$. Notice that the assumption $\Leb^1(\{\psi =-\infty\})=0$ avoids ambiguities and ensures that $\mu_0,\mu_1 \ll \Leb^1$, we are going to call $\tilde\rho_0$ and $\tilde\rho_1$ (respectively) their densities, that is $\mu_0=\tilde \rho_1 \Leb^1$ and $\mu_0=\tilde \rho_1 \Leb^1$. Now, Brenier theorem
ensures that there exists a unique optimal transport plan between $\mu_0$ and $\mu_1$, and it is induced by a map $T$, which is differentiable $\mu_0$-almost everywhere. It is also well known 
that the map $T$ is increasing, thus $T'(x)$ will be non-negative when defined. Moreover, the unique Wasserstein geodesic connecting $\mu_0$ and $\mu_1$ is given by $\mu_t=(T_t)_\#\mu_0$, where $T_t= (1-t) \text{id} + t T$. Then, calling $\tilde \rho_t$ the density of $\mu_t$ with respect to the Lebesgue measure $\Leb^1$, the Jacobi equation holds and gives that
\begin{equation*}
     \tilde \rho_0(x)= \tilde\rho_t(T_t(x)) J_{T_t} (x) = \tilde\rho_t(T_t(x)) (1+ t(T'(x)-1)),
\end{equation*}
for $\mu_0$-almost every $x$. On the other hand it is obvious that $\tilde \rho_t= e^{-\psi}\rho_t$ for every $t\in[0,1]$, therefore
\begin{equation}\label{eq:jacobi}
    e^{-\psi(x)}  \rho_0(x) = e^{-\psi(T_t(x))}\rho_t(T_t(x)) (1+ t(T'(x)-1))
\end{equation}
for every $t\in[0,1]$ and $\mu_0$-almost every $x$. 
On the other hand notice that for every $N'<-1$
\begin{equation}\label{eq:pointwiseconv}
\begin{split}
    S_{N'+1}(\mu_t)= \int \rho_t^{-\frac{1}{N'+1}} \de \mu_t = \int \rho_t^{-\frac{1}{N'+1}} \de [(T_t)_\#\mu_0]= \int \rho_t(T_t(x))^{-\frac{1}{N'+1}}  \de \mu_0(x) \\
    = \int \big( e^{-\psi(x)}  \rho_0(x)\big)^{-\frac{1}{N'+1}} \big( e^{-\psi(T_t(x))} (1+ t(T'(x)-1)) \big)^\frac{1}{N'+1} \de \mu_0(x).
\end{split}    
\end{equation}
We can then prove the convexity pointwise, using the $(K,N)$-convexity of $\psi$. In particular, $\psi$ is $(K,N')$-convex for every $N'\in[N,0)$ (cfr. \cite[Lemma 2.9]{Ohta16}). Therefore, calling $A(x)=T'(x)-1$ in order to ease the notation, for every $N'\in[N,-1)$, it holds that
\[
\begin{split}
    \big(e^{-\psi(T_t(x))}& (1+ tA(x)) \big)^\frac{1}{N'+1}= \big[e^{-\psi(T_t(x))/N'}\big]^\frac{N'}{N'+1} (1+ tA(x)) ^\frac{1}{N'+1} \\
    &\leq (1+ tA(x)) ^\frac{1}{N'+1} \bigg[ \sigma_{K/N'}^{(1-t)}(|T(x)-x|) e^{-\psi(x)/N'} + \sigma_{K/N'}^{(t)}(|T(x)-x|) e^{-\psi(T(x))/N'} \bigg]^\frac{N'}{N'+1}.
    \end{split}
    \]
Thus, by rewriting the last term, we obtain
    \[
    \begin{split}
    &\bigg[ (1+ tA(x)) ^\frac{1}{N'}\sigma_{K/N'}^{(1-t)}(|T(x)-x|) e^{-\psi(x)/N'} + (1+ tA(x)) ^\frac{1}{N'} \sigma_{K/N'}^{(t)}(|T(x)-x|) e^{-\psi(T(x))/N'} \bigg]^\frac{N'}{N'+1} \bigskip \\
   & \leq \tau_{K/(N'+1)}^{(1-t)}(|T(x)-x|) e^{-\psi(x)/(N'+1)} +  \tau_{K/(N'+1)}^{(t)}(|T(x)-x|) (1+ A(x)) ^\frac{1}{N'+1} e^{-\psi(T(x))/(N'+1)},
\end{split}
\]
 where the last inequality is a consequence of Jensen's inequality and of the definition of $\tau_{K,N'}^{(t)}$ (see \eqref{eq:deftau}).
Replacing the above inequality in \eqref{eq:pointwiseconv} and using \eqref{eq:jacobi} with $t=1$, we obtain that
\begin{equation*}
    \begin{split}
    S&_{N'+1}(\mu_t)\leq\int \tau_{K/(N'+1)}^{(1-t)}(|T(x)-x|) \rho_0(x)^{-\frac{1}{N'+1}} \de \mu_0(x) \\
    &\qquad+ \int \tau_{K/(N'+1)}^{(t)}(|T(x)-x|) (1+ A(x)) ^\frac{1}{N'+1} \big( e^{-\psi(x)}  \rho_0(x)\big)^{-\frac{1}{N'+1}} e^{-\psi(T(x))/(N'+1)} \de \mu_0(x)\\
    &= \int \tau_{K/(N'+1)}^{(1-t)}(|T(x)-x|) \rho_0(x)^{-\frac{1}{N'+1}} \de \mu_0(x) + \int \tau_{K/(N'+1)}^{(t)}(|T(x)-x|) \rho_1(T(x)) ^{-\frac{1}{N'+1}} \de \mu_0(x)\\
    &= \int \tau_{K/(N'+1)}^{(1-t)}(|T(x)-x|) \rho_0(x)^{-\frac{1}{N'+1}} \de \mu_0(x) + \int \tau_{K/(N'+1)}^{(t)}(|T(x)-x|) \rho_1(x) ^{-\frac{1}{N'+1}} \de \mu_1(x),
\end{split}  
\end{equation*}
for every $N'\in[N,-1)$, which is the desired inequality.
\end{proof}

\noindent
A direct application of the previous result leads to the following refinement of Example \ref{ex:nice}:
\begin{example}\label{ex:ugly}
{\bf (ii')} For any pair of real numbers $K > 0, N < -1$, the weighted space
\[
\Big(\R, |\cdot|, V \Leb^1\Big) \quad \text{ with } V(x) = \sinh\bigg(  x \sqrt{- \dfrac{K}{N}} \bigg)^N,
\]
obtained gluing two copies of the half-line space in Example \ref{ex:nice}-(ii),
satisfies the curvature-dimension condition $\CD(K, N+1)$ with singular set $\mathcal{S}_{V \Leb^1} = \{ 0 \}$.\\
{\bf (iii')} Similarly, for any $N < -1$ the space $(\R, |\cdot|, |x|^{N} \Leb^1)$ is a $\CD(0, N+1)$ space with singular set $\mathcal{S}_{|x|^{N} \Leb^1} = \{ 0 \}$.\\
{\bf (iv')} For any pair of real numbers $K < 0, N < -1$ the space which is obtained gluing $J$-copies of the interval in Example \ref{ex:nice}-(iv), for example by considering $\Big( \bigcup_{j=1}^J I_j,  |\cdot |, V \Leb^1 \Big)$ with
\[
{ I_j := \Bigg[ \dfrac{(2j-1)\pi}{2} \sqrt{\dfrac{K}{N}},  \dfrac{(2j+1) \pi}{2} \sqrt{\dfrac{K}{N}} \Bigg]} \, \text{ and } \, V := \sum_{j=1}^J \mathbbm{1}_{I_j} \cdot\cos \bigg(  (x - x_j ) \sqrt{\dfrac{K}{N}} \bigg)^N, \,\, x_j := {j \pi} \sqrt{\dfrac{K}{N}}, 
\]
satisfies the curvature-dimension condition $\CD(K, N+1)$ with singular set given by 
\[\mathcal{S}_{V \Leb^1} = \Bigg\{ \dfrac{(2j-1)\pi}{2} \sqrt{\dfrac{K}{N}} : j = 1, \dots, J+1  \Bigg\}.
\]
\end{example}

\noindent We end this section by pointing out that there exist unbounded $\CD$ spaces with negative dimension, for every value of the curvature. In particular, unlike to what happens for positive dimensional CD spaces, it is never possible to obtain a bound on the diameter of the space. Actually, this not only happens for singular spaces, as in Example \ref{ex:ugly} (iv'), since for example, also the hyperbolic plane satisfies the $\CD(-1,N)$-condition for any $N < 0$ (recall that every $\CD(K, N)$ space with $N \ge1$ is automatically a $\CD(K, N)$ space for any $N <0$). Therefore, there exists no counterpart of the generalized Bonnet-Myers theorem \cite[Corollary 2.6]{Sturm06II} for negative dimensional $\CD$ spaces. 

\subsection{Approximate $\CD$-condition and regularity assumptions}\label{sec:assumption}

Examples \ref{ex:nice}  and \ref{ex:ugly} exhibit that $\mm$-singular sets associated to $\CD(K, N)$ spaces are not necessarily empty sets. Nevertheless, as already noticed, the geometric behaviour in these examples can be extremely different: in opposition to Examples \ref{ex:nice}, where points in $\mathcal{S}_\mm$   appear only as terminal  points of geodesics in the space,  
singular points in Examples \ref{ex:ugly} occur as inner points of geodesics. This observation shows that, in particular, the $k$-th regular set $\mathcal{R}^k$ of a space $X$, which was introduced in \eqref{def:Rk}, is not necessarily geodesically convex. This turns out to produce major difficulties in the proof of our main result. Indeed, we wish to approximate metric measure spaces by considering their $k$-th regular sets but the $\CD$-condition is precisely a condition made on geodesics. We point out that this type of  issues do not show up for positive values of the dimensional parameter and, in fact, to surpass the problems arising from the non-geodesically convexity will  be the main challenge to overcome in the proof of the Stability Theorem \ref{th:stab}. In order to do so we present the next two definitions.

\begin{dfn}[Approximate $\CD$-condition]
We say that a metric measure space $(\X, \dis, \m)$ satisfies the  \emph{approximate curvature-dimension condition}, $\CD^a(K, N)$ in short,  if the $\CD$-condition \ref{def:curvcond} is satisfied by further requiring that the supports of the  measures  $\mu_{0},\, \mu_{1} \in \PX_{2}^{ac}(\X)$ satisfy $\supp(\mu_0),\supp(\mu_1) \subset \mathcal R^{k}$, for some $k\in\N$.
\end{dfn}
Note that in the definition above $k$ is not fixed.


As discussed above, we need to carefully approach the topic of the non-geodesic convexity of $\m$-singular sets. The following concept directs us in this direction by quantifying in term of masses - and thus, controlling - to which extent convexity of the $k$-th cuts is unsatisfied. 

\begin{dfn}[$\omega$-uniformly convexity]
A metric measure space $(\X, \di, \mm)$ is \emph{$\omega$-uniformly convex} if there exists a function $\omega \colon \N \times \N \times \R^+ \to [0, 1]$ with the following properties: 
\begin{itemize}
\item for any $\mu_{0},\, \mu_{1} \in \PX_{2}(\X)$, with $S_{N,\m}(\mu_0), S_{N,\m}(\mu_1)\le M$ and $\supp(\mu_0), \supp(\mu_1) \subseteq \mathcal R^{k}$, every $t$-middle point of any geodesic $\{\mu_t\}_{t \in [0, 1]}\subset \PX_2(\X)$ between $\mu_0$ and $\mu_1$ is such that
\[ \mu_t(\mathcal R^h) \ge 1 - \omega(k, h, M); \]
\item for any $k \in \N$, $M \in \R^+$	
\begin{equation}\label{eq:omega0}
\lim_{h \to \infty} \omega (k, h, M) = 0.
\end{equation}
\end{itemize}
\end{dfn}
Figure \ref{fig:unifconvexity2} shows an schematic representation of the $\omega$-convexity.

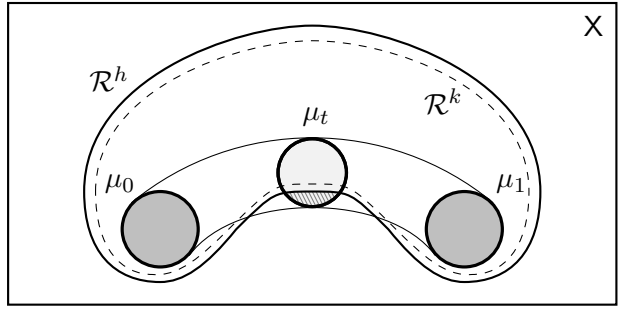
\begin{figure}
\begin{center}

\begin{tikzpicture} 
\filldraw[ fill=black!5!white, very thick](0,0.75) circle (0.45);
\draw[black!50!white] (0,0.5)--(0.1,0.3);
\draw[black!50!white] (0.05,0.5)--(0.15,0.3);
\draw[black!50!white] (0.1,0.5)--(0.18,0.34);
\draw[black!50!white] (0.15,0.5)--(0.23,0.34);
\draw[black!50!white] (0.2,0.5)--(0.26,0.38);
\draw[black!50!white] (0.25,0.5)--(0.3,0.4);
\draw[black!50!white] (-0.05,0.5)--(0.05,0.3);
\draw[black!50!white] (-0.1,0.5)--(0,0.3);
\draw[black!50!white] (-0.15,0.5)--(-0.05,0.3);
\draw[black!50!white] (-0.2,0.5)--(-0.1,0.3);
\draw[black!50!white] (-0.25,0.5)--(-0.17,0.34);
\draw[black!50!white] (-0.3,0.5)--(-0.24,0.38);
\draw[thick](-4,3)--(4,3)--(4,-1)--(-4,-1)--cycle;
\filldraw[ fill=black!25!white, very thick](2,0) circle (0.5);
\draw (-1.62,-0.35) .. controls (-1,0.5) and (1,0.5)..(1.62,-0.35);
\draw(-2.38,0.35).. controls (-1,1.5) and (1,1.5)..(2.38,0.35);
\draw[ very thick](0,0.75) circle (0.45);
\draw[thick] (-3,0.5)..controls (-3,2) and (-1,2.7).. (0,2.7)..controls (1,2.7) and (3,2).. (3,0.5)..controls (3,-0.3) and (2.6,-0.7)..(2,-0.7).. controls (1.7,-0.7) and (1.4,-0.5)..(1,0).. controls (0.5,0.5)..(0,0.5) .. controls (-0.5,0.5).. (-1,0).. controls (-1.4,-0.5) and (-1.7,-0.7)..(-2,-0.7).. controls (-2.6,-0.7) and (-3,-0.3)..(-3,0.5);
\draw[dashed] (-2.85,0.5)..controls (-2.85,1.9) and (-0.9,2.5).. (0,2.5)..controls (1,2.5) and (2.85,1.9).. (2.85,0.5)..controls (2.85,-0.27) and (2.5,-0.6)..(2.1,-0.6).. controls (1.75,-0.63) and (1.45,-0.35)..(1.1,0).. controls (0.5,0.6)..(0,0.6) .. controls (-0.5,0.6).. (-1.1,0).. controls (-1.45,-0.35) and (-1.75,-0.63) ..(-2.1,-0.6).. controls (-2.5,-0.6) and (-2.85,-0.27)..(-2.85,0.5);
\filldraw[ fill=black!25!white, very thick](-2,0) circle (0.5);
\node at (-3,0.6)[label=east:$\mu_{0}$] {};
\node at (2.15,0.6)[label=east:$\mu_{1}$] {};
\node at (-0.4,1.5)[label=east:$\mu_{t}$] {};
\node at (3.3,2.7)[label=east:$\X$] {};
\node at (-3.2,2)[label=east:$\mathcal R^h$] {};
\node at (1.2,1.7)[label=east:$\mathcal R^k$] {};

\end{tikzpicture}

\caption{A visual representation of the property of $\omega$-uniform convexity. In particular the shaded set has $\mu_t$-mass bounded above by $\omega(k,h,M)$, if $S_{N,\m}(\mu_0),S_{N,\m}(\mu_1) \le M$.}
\label{fig:unifconvexity2}
\end{center}
\end{figure}

\begin{rmk}
We illustrate with some examples the concept of $\omega$-uniform convexity.

\begin{itemize}
\item[\textasteriskcentered] If the $k$-th regular set $\mathcal R^k$ is geodesically convex, then we can choose $\omega(\ell,h,M)=0$, for all entropy bounds $M>0$ and $\ell \le k \le h$. This is the case, for example, of metric measure spaces with empty $\m$-singular set and the spaces presented in Example \ref{ex:nice}. 
\item[\textasteriskcentered] Conversely,  if $\mu_{0},\, \mu_{1} \in \mathcal{P}_{2}(\X)$  are supported in $\mathcal R^{k}$ with bounded entropies and the support of $\mu_t$, a geodesic joining $\mu_0$ and $\mu_1$ evaluated at time $t$, is contained in the complement of $\mathcal R^h$, for some $t \in (0, 1)$ and $h\in \N$, then $\omega (k, h, M) = 1$.\\ In particular, the metric measure space $([-1,1],|\cdot|,\m)$, with $\d\m= \delta_{-1} + \delta_{1} + 1/x^2 \d \mathcal{L}^1$ serves as example of a metric measure space which is not  $\omega$-uniformly convex. Indeed, at time $t=1/2$, the support of the unique 2-Wasserstein geodesic $(\delta_t)_{t\in[0,1]}$ is contained inside $X\setminus \mathcal R^h$, for any $h\in N$, while its terminal points have entropy equal to $1$. 
\item[\textasteriskcentered] Lastly, a more interesting behaviour occurs in a convex subset of Example \ref{ex:ugly}, given by 
\[
\left ( \left [0,\pi \right ], \,\big |\cdot \big |\, , \mathbb{I}_{\left [0,\frac{\pi}{2}  \right ]}\mathrm{cos}\left (x \right )^{-2} \mathcal{L}^1 + \mathbb{I}_{\left[\frac{\pi}{2}  , \pi \right ]} \mathrm{cos}\left (x-\pi \right )^{-2}\mathcal{L}^1 \right ),
\]
whose singular set is $S=\{\frac{\pi}{2}  \}$. This is a $\CD(-2,-1)$ space as well as an $\omega$-uniformly convex space for a non-trivial function $\omega(k,h,M)$. Note that since there exist Wasserstein geodesics which are, at some time $t$, entirely contained  in the complement of $\mathcal R^k$, there are actually some values  $k,h\in\N,\,M\in\R^+$ for which $\omega(k,h,M)=1$. Also, for fixed values of $k \in \N$ and $M \in \R^+$, $\omega(k,h,M)\to 0$ as $h\to\infty$, since $W_2$-geodesics are absolutely continuous 
and $X\setminus\mathcal R^{h} \to 0$ as $h\to\infty$. Indeed, the key observation here is that we can not force arbitrarily large amounts of mass to transit through $X\setminus\mathcal R^{h}$, at a given time, without losing the upper bound on the entropy of the terminal points. 
Intuitively, to produce such  geodesics, 
we would have to consider measures with arbitrarily small supports or which accumulate arbitrarily large masses around a point. However, these type of measures have  large entropy. \\
\end{itemize} 
\end{rmk}

A concrete and useful property which $\omega$-uniformly convex metric measure spaces enjoy is that we are able to quantify interpolated mass outside the set $\mathcal R^h$, even if the marginals are not necessarily supported on $\mathcal R^k$, granted they supply sufficient mass to the $k$-th regular sets.

\begin{prop}\label{prop:Omega}
Let $(\X, \d, \mm)$ be a $\omega$-uniformly convex space. Then there exists a function $\Omega \colon  \N \times \N \times \R^+ \times [0,1] \to \R$ such that:
\begin{itemize}
\item[i)] for any $\mu_{0},\, \mu_{1} \in \PX_{2}(\X)$ with $S_{N,\m}(\mu_0),S_{N,\m}(\mu_1)\le M$ and $\mu_0(\mathcal{R}^k), \mu_1(\mathcal{R}^k) \ge 1-\delta$, then any $t$-middle point of the geodesic $\{\mu_t\}_{t \in [0, 1]}$ is such that $\mu_t(\mathcal R^h) \ge 1 - \Omega(k, h, M, \delta)$,
\item[ii)] for every $0 \leq \delta< \frac 14$ it holds that 
\begin{equation}\label{eq:Omega}
    \limsup_{h \to \infty} \Omega(k, h, M, \delta) \le 2 \delta,
\end{equation}
for every fixed $k \in \N$ and $M \in \R^+$.
\end{itemize}
\end{prop}

\begin{figure}
\begin{center}

\begin{tikzpicture} 

\filldraw[ fill=black!5!white, very thick](0,0.487) circle (0.55);
\filldraw[ fill=black!15!white, very thick](-2.15,-0.4) circle (0.5);
\filldraw[ fill=black!15!white, very thick](2.15,-0.4) circle (0.5);
\draw[thick] (-3,0.5)..controls (-3,2) and (-1,2.7).. (0,2.7)..controls (1,2.7) and (3,2).. (3,0.5)..controls (3,-0.5) and (2.6,-0.75)..(2,-0.7).. controls (1.7,-0.7) and (1.4,-0.5)..(1,0).. controls (0.7,0.3) and (0.5,0.4)..(0,0.4) .. controls (-0.5,0.4) and (-0.7,0.3).. (-1,0).. controls (-1.4,-0.5) and (-1.7,-0.7)..(-2,-0.7).. controls (-2.6,-0.75) and (-3,-0.5)..(-3,0.5);

\draw[black!50!white] (-2.25,-0.7)--(-2.25,-0.9);
\draw[black!50!white] (-2.2,-0.7)--(-2.2,-0.9);
\draw[black!50!white] (-2.15,-0.7)--(-2.15,-0.9);
\draw[black!50!white] (-2.1,-0.7)--(-2.1,-0.9);
\draw[black!50!white] (-2.05,-0.7)--(-2.05,-0.9);
\draw[black!50!white] (-2,-0.7)--(-2,-0.88);
\draw[black!50!white] (-1.95,-0.7)--(-1.95,-0.86);
\draw[black!50!white] (-1.9,-0.7)--(-1.9,-0.84);
\draw[black!50!white] (-1.85,-0.69)--(-1.85,-0.8);
\draw[black!50!white] (-1.8,-0.67)--(-1.8,-0.77);
\draw[black!50!white] (-2.3,-0.7)--(-2.3,-0.89);
\draw[black!50!white] (-2.35,-0.69)--(-2.35,-0.87);
\draw[black!50!white] (-2.4,-0.67)--(-2.4,-0.85);
\draw[black!50!white] (-2.45,-0.67)--(-2.45,-0.81);
\draw[black!50!white] (-2.5,-0.65)--(-2.5,-0.77);
\draw[black!50!white] (-2.55,-0.62)--(-2.55,-0.72);

\draw[black!50!white] (2.25,-0.7)--(2.25,-0.9);
\draw[black!50!white] (2.2,-0.7)--(2.2,-0.9);
\draw[black!50!white] (2.15,-0.7)--(2.15,-0.9);
\draw[black!50!white] (2.1,-0.7)--(2.1,-0.9);
\draw[black!50!white] (2.05,-0.7)--(2.05,-0.9);
\draw[black!50!white] (2,-0.7)--(2,-0.88);
\draw[black!50!white] (1.95,-0.7)--(1.95,-0.86);
\draw[black!50!white] (1.9,-0.7)--(1.9,-0.84);
\draw[black!50!white] (1.85,-0.69)--(1.85,-0.8);
\draw[black!50!white] (1.8,-0.67)--(1.8,-0.77);
\draw[black!50!white] (2.3,-0.7)--(2.3,-0.89);
\draw[black!50!white] (2.35,-0.69)--(2.35,-0.87);
\draw[black!50!white] (2.4,-0.67)--(2.4,-0.85);
\draw[black!50!white] (2.45,-0.67)--(2.45,-0.81);
\draw[black!50!white] (2.5,-0.65)--(2.5,-0.77);
\draw[black!50!white] (2.55,-0.62)--(2.55,-0.72);

\draw[black!50!white] (0,-0.05)--(0,0.3);
\draw[black!50!white] (0.05,-0.05)--(0.05,0.3);
\draw[black!50!white] (0.1,-0.05)--(0.1,0.3);
\draw[black!50!white] (0.15,-0.05)--(0.15,0.3);
\draw[black!50!white] (0.2,-0.02)--(0.2,0.3);
\draw[black!50!white] (0.25,-0.02)--(0.25,0.3);
\draw[black!50!white] (0.3,0)--(0.3,0.28);
\draw[black!50!white] (0.35,0.05)--(0.35,0.27);
\draw[black!50!white] (0.4,0.1)--(0.4,0.27);
\draw[black!50!white] (0.45,0.15)--(0.45,0.25);
\draw[black!50!white] (-0.05,-0.05)--(-0.05,0.3);
\draw[black!50!white] (-0.1,-0.05)--(-0.1,0.3);
\draw[black!50!white] (-0.15,-0.05)--(-0.15,0.3);
\draw[black!50!white] (-0.2,-0.02)--(-0.2,0.3);
\draw[black!50!white] (-0.25,-0.02)--(-0.25,0.3);
\draw[black!50!white] (-0.3,0)--(-0.3,0.28);
\draw[black!50!white] (-0.35,0.05)--(-0.35,0.27);
\draw[black!50!white] (-0.4,0.1)--(-0.4,0.27);
\draw[black!50!white] (-0.45,0.15)--(-0.45,0.25);

\draw[  very thick](0,0.487) circle (0.55);
\draw[ very thick](-2.15,-0.4) circle (0.5);
\draw[ very thick](2.15,-0.4) circle (0.5);
\draw[thick](-4,3.5)--(4,3.5)--(4,-1.5)--(-4,-1.5)--cycle;
\draw[thick] (-3,0.5)..controls (-3,2) and (-1,2.7).. (0,2.7)..controls (1,2.7) and (3,2).. (3,0.5)..controls (3,-0.5) and (2.6,-0.75)..(2,-0.7).. controls (1.7,-0.7) and (1.4,-0.5)..(1,0).. controls (0.7,0.3) and (0.5,0.4)..(0,0.4) .. controls (-0.5,0.4) and (-0.7,0.3).. (-1,0).. controls (-1.4,-0.5) and (-1.7,-0.7)..(-2,-0.7).. controls (-2.6,-0.75) and (-3,-0.5)..(-3,0.5);
\draw[thick] (-3.3,0.5)..controls (-3.3,2.2) and (-1.1,3).. (0,3)..controls (1.1,3) and (3.3,2.2).. (3.3,0.5)..controls (3.3,-0.5) and (2.8,-1.1)..(2,-1).. controls (1.65,-0.95) and (1.4,-0.6)..(1,-0.15).. controls (0.7,0.2) and (0.5,0.3)..(0,0.3) .. controls (-0.5,0.3) and (-0.7,0.2).. (-1,-0.15).. controls (-1.4,-0.6) and (-1.65,-0.95)..(-2,-1).. controls (-2.8,-1.1) and (-3.3,-0.5)..(-3.3,0.5);
\draw(-2.53,-0.05).. controls (-1,1.4) and (1,1.4)..(2.53,-0.05);
\draw (-1.865,-0.83) .. controls (-1,0.2) and (1,0.2)..(1.865,-0.83);

\node at (-3,0.4)[label=east:$\mu_{0}$] {};
\node at (2.15,0.4)[label=east:$\mu_{1}$] {};
\node at (-0.4,1.3)[label=east:$\mu_{t}$] {};
\node at (3.3,3.2)[label=east:$\X$] {};
\node at (-3.4,2.2)[label=east:$\mathcal R^h$] {};
\node at (1.3,1.7)[label=east:$\mathcal R^k$] {};

\end{tikzpicture}

\caption{A visual representation of the property provided by the function $\Omega$, that is i) in Proposition \ref{prop:Omega}. In particular the shaded set in the center has $\mu_t$-mass bounded above by $\Omega(k,h,M,\delta)$, if the shaded set on the left and the one on the right have $\mu_0$-mass and $\mu_1$-mass (respectively) less than $\delta$ and $S_{N,\m}(\mu_0),S_{N,\m}(\mu_1) \le M$.  }
\label{fig:unifconvexity}
\end{center}
\end{figure}
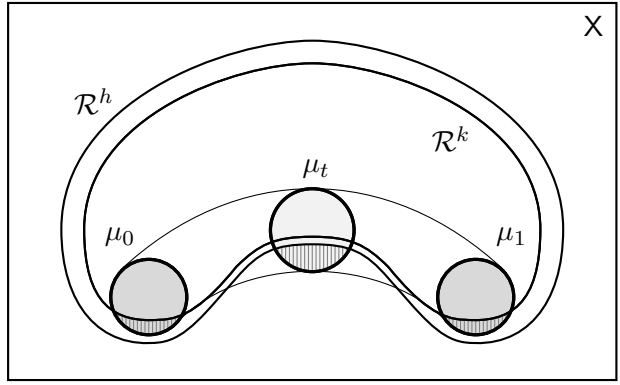

\begin{proof}
Notice that we can limit ourselves to the case when $0\leq\delta<\frac 14$, because we can simply put $\Omega(k,h,M,\delta)= 1$ if $\delta \geq \frac 14 $. Fixed $\mu_{0},\, \mu_{1} \in \PX_{2}(\X)$ such that $S_{N,\m}(\mu_0),S_{N,\m}(\mu_1)\le M$ and $\mu_0(\mathcal{R}^k), \mu_1(\mathcal{R}^k) \ge 1-\delta$, consider a $t$-middle point of a geodesic $\{\mu_t\}_{t \in [0, 1]}$, connecting $\mu_0$ and $\mu_1$. Call $\eta\in \PX(\text{Geo}(\X))$ the representation of $\{\mu_t\}_{t \in [0, 1]}$ and define 
\begin{equation*}
    \tilde \eta := \frac{1}{\eta (G)} \cdot \eta|_G \in \PX(\text{Geo}(\X)),
\end{equation*}
where 
\begin{equation*}
    G:= \{\gamma \in \text{Geo}(\X)\,:\, \gamma(0),\gamma(1)\in \mathcal R^k\}.
\end{equation*}
Notice that $\tilde \eta $ is actually well-defined, since our condition on $\mu_0$ and $\mu_1$ ensures that $\eta(G)\geq1-2\delta>0$. Moreover,
\begin{equation}\label{eq:eta}
    \eta=\eta(G)\cdot \tilde\eta + \bar \eta \quad \text{for some $\bar \eta\in\M(\text{Geo}(\X))$ with $\bar \eta(\text{Geo}(\X))\leq 2\delta$.}
\end{equation}
Observe that $\{\tilde\mu_t=(e_t)_\# \tilde\eta\}_{t\in[0,1]}$ is a Wasserstein geodesic connecting two measures $\tilde \mu_0$ and $\tilde\mu_1$, which are supported on $\mathcal R^k$ and satisfy
\begin{equation*}
\begin{split}
   \max \{ S_{N,\m}(\tilde\mu_0),S_{N,\m}(\tilde\mu_1)\}  &\leq  \bigg[ \frac{1}{\eta (G)}\bigg]^{1-\frac 1N} \max\{S_{N,\m}(\mu_0),S_{N,\m}(\mu_1)\}  \\ & \leq \bigg[ \frac{1}{1-2\delta}\bigg]^{1-\frac 1N} M \leq 2^{1-\frac 1N} M.
   \end{split}
\end{equation*}
Then, the $\omega$-uniform convexity of $(\X,\di,\m)$ ensures that, for every $h$,
\begin{equation*}
    \tilde\mu_t(\mathcal{R}^h)\geq 1- \omega(k,h,2^{1-\frac 1N} M).
\end{equation*}
Moreover, taking into account \eqref{eq:eta}, we can conclude that 
\begin{equation*}
    \mu_t(\mathcal R^h) \geq (1-2\delta)\cdot\tilde\mu_t(\mathcal{R}^h)\geq 1- \omega(k,h,2^{1-\frac 1N} M)- 2 \delta.
\end{equation*}
Therefore, to satisfy i), we can set
\begin{equation*}
    \Omega(k,h,M,\delta):=\omega(k,h,2^{1-\frac 1N} M)+ 2 \delta.
\end{equation*}
With this definition, ii) is a straightforward convergence of condition \eqref{eq:omega0} on $\omega(k,h,M)$.
\end{proof}

\section{Stability of $\CD$-condition}
The aim of the last section is to prove the main result,
\begin{thm}[Stability]\label{th:stab} 
Let $K\in\R$, $N\in (\infty,0)$, and $\big\{ (\X_n,\dis_n,\m_n, \mathcal S_{\mm_n}, p_n)\big \}_{n\in \N}\subset \mathcal{M}^{qR}_{\bar k}$ be a sequence of metric measure spaces converging to $(\X_\infty,\dis_\infty,\m_\infty, \mathcal S_{\mm_\infty}, p_\infty)\in \mathcal{M}^{qR}_{\bar k}$ in the  $\mathsf{iKRW}$-distance, for some $\bar k \in \N$. Assume further that:
\begin{itemize}
    \item[(i)] $(\X_n,\dis_n,\m_n)$ is a $\CD(K,N)$ space for every $n \in \N$;
    \item[(ii)] there exists $\omega:\N \times \N \times \R\to [0,1]$, 
    for which $(\X_n,\dis_n,\m_n)$ is $\omega$-uniformly convex, for every $n \in \N$; and
    \item[(iii)] $\sup_{i\in \N \cup \{\infty\}} \diam(\X_i,\di_i) < \pi\sqrt{\frac{1}{|K|}} $, if $K<0$.
\end{itemize}
Then $(\X_\infty, \dis_\infty, \m_\infty)$ is a $\CD(K,N)$ space.
\end{thm}
As a matter of fact, Theorem \ref{th:stab} is concluded from the slightly more general statement below, since Proposition \ref{pr:krw} provides an effective realization for an $\mathsf{iKRW}$-converging sequence of metric measure spaces. Recall that the extrinsic convergence of metric measure spaces is presented in Definition \ref{de:ext}.  
\begin{thm}[Extrinsic Stability]\label{th:ext}
Let $K\in\R$, $N\in (\infty,0)$. Then the $\CD(K,N)$-condition is stable under the extrinsic convergence of metric measure spaces, granted conditions \emph{(i)-(iii)} from Theorem \ref{th:stab}. are satisfied by the converging sequence. 
\end{thm}
The following is an immediate result of Theorem \ref{th:ext}.
\begin{cor}
Let $K\in\R$, $N\in (\infty,0)$, and $\big\{ (\X_n,\dis_n,\m_n, \mathcal S_{\mm_n}, p_n)\big \}_{n\in \N}\subset \mathcal{M}^{qR}_{\bar k}$ be a sequence converging to $(\X_\infty,\dis_\infty,\m_\infty, \mathcal S_{\mm_\infty}, p_\infty)\in \mathcal{M}^{qR}_{\bar k}$, in the extrinsic or intrinsic manner. Assume that the regular sets $\mathcal{R}_n^k$ are geodesically convex, for all $k,n\in\N$. If for every $n \in \N$ the space $(\X_n, \dis_n, \m_n)$ satisfies the $\CD(K,N)$ condition (with $\sup_{i\in \N \cup \{\infty\}} \diam(\X_i,\di_i) < \pi|K|^{-1/2} $, if $K<0$), then also $(\X_\infty, \dis_\infty, \m_\infty)$ is a $\CD(K,N)$ space.

\end{cor} 

When compared with \emph{Stability Theorem} \ref{th:stab}, the advantage of \emph{Extrinsic Stability Theorem} \ref{th:ext}, is that no assumptions have to be made, regarding the limiting behaviour of singular sets along the sequence. This contrasts with  the \emph{Stability Theorem}, which is stated in terms of the intrinsic $\dint$-convergence, since the $\dint$-distance controls the Hausdorff distance between singular sets. Therefore, with the latter Theorem one gains some flexibility to study the aforementioned sets; nevertheless, there is a price to pay in exchange. Namely, it is  necessary to be in possession of an effective realization for the convergence. In this sense, we find that both results complement very well each other. \\

We fix some notation prior to outlining the argument in the proof of Theorem \ref{th:ext}. \\ In the remainder, we use the adjective \emph{horizontal} to refer to the approximations we construct inside a fixed space $\mathbb{X}_n$, for $n\in\N$. Respectively, we denote as \emph{vertical} approximations  those approximations made over the sequence $\mathbb{X}_n\to \mathbb{X}_\infty$, when we let $n\to \infty$. 
%
Our objective is, naturally,  to demonstrate, for every pair of measures $\mu_0,\mu_1\in \mathscr{P}_2^{ac}(\X_\infty)$, the existence of a 2-Wasserstein geodesic $\{\mu_{t}\}_{t \in [0, 1]} \subset \PX_2(\X_\infty)$ and an optimal plan $q\in \mathsf{Opt}(\mu_0,\mu_1)$, for which the curvature-dimension inequality \eqref{def:CD} is satisfied. 
We accomplish this by following the next steps. \\

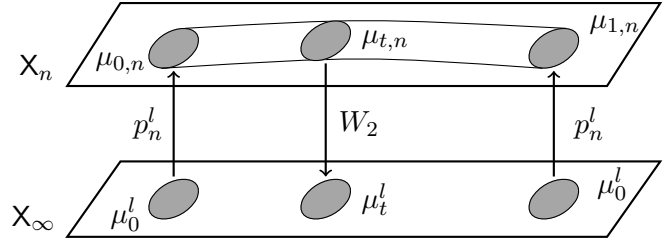
\begin{figure}
\begin{center}

\begin{tikzpicture} 

\draw[thick](-3.9,3)--(3.2,3)--(3.9,4.1)--(-3.2,4.1)--cycle;
\draw[thick](-3.9,1)--(3.2,1)--(3.9,2)--(-3.2,2)--cycle;
\filldraw[black!35!white] (-2.8,3.5).. controls (-2.6,3.85) and (-2.05,3.85)..(-2.2,3.5)..controls (-2.4,3.15) and (-2.95,3.15).. (-2.8,3.5);
\draw(-2.8,3.5).. controls (-2.6,3.85) and (-2.05,3.85)..(-2.2,3.5)..controls (-2.4,3.15) and (-2.95,3.15).. (-2.8,3.5);
\filldraw[black!35!white] (2.2,3.5).. controls (2.4,3.85) and (2.95,3.85)..(2.8,3.5)..controls (2.6,3.15) and (2.05,3.15).. (2.2,3.5);
\draw (2.2,3.5).. controls (2.4,3.85) and (2.95,3.85)..(2.8,3.5)..controls (2.6,3.15) and (2.05,3.15).. (2.2,3.5);
\draw (-2.66,3.235)..controls (0,3.4)..  (2.35,3.235);
\draw (-2.35,3.765) ..controls (0,3.9).. (2.66, 3.765);
\filldraw[black!35!white] (-2.8,1.5).. controls (-2.6,1.85) and (-2.05,1.85)..(-2.2,1.5)..controls (-2.4,1.15) and (-2.95,1.15).. (-2.8,1.5);
\draw(-2.8,1.5).. controls (-2.6,1.85) and (-2.05,1.85)..(-2.2,1.5)..controls (-2.4,1.15) and (-2.95,1.15).. (-2.8,1.5);
\filldraw[black!35!white] (2.2,1.5).. controls (2.4,1.85) and (2.95,1.85)..(2.8,1.5)..controls (2.6,1.15) and (2.05,1.15).. (2.2,1.5);
\draw (2.2,1.5).. controls (2.4,1.85) and (2.95,1.85)..(2.8,1.5)..controls (2.6,1.15) and (2.05,1.15).. (2.2,1.5);
\filldraw[black!35!white] (-0.8,3.6).. controls (-0.6,3.95) and (-0.05,3.95)..(-0.2,3.6)..controls (-0.4,3.25) and (-0.95,3.25).. (-0.8,3.6);
\draw (-0.8,3.6).. controls (-0.6,3.95) and (-0.05,3.95)..(-0.2,3.6)..controls (-0.4,3.25) and (-0.95,3.25).. (-0.8,3.6);
\filldraw[black!35!white] (-0.8,1.5).. controls (-0.6,1.85) and (-0.05,1.85)..(-0.2,1.5)..controls (-0.4,1.15) and (-0.95,1.15).. (-0.8,1.5);
\draw (-0.8,1.5).. controls (-0.6,1.85) and (-0.05,1.85)..(-0.2,1.5)..controls (-0.4,1.15) and (-0.95,1.15).. (-0.8,1.5);
\draw[thick,->](-2.5,1.8)--(-2.5,3.2);
\draw[thick,->](2.5,1.8)--(2.5,3.2);
\draw[thick,<-](-0.5,1.8)--(-0.5,3.3);
\node at (-4.9,1.25)[label=east:$\X_\infty$] {};
\node at (-4.8,3.25)[label=east:$\X_n$] {};
\node at (-3.8,3.3)[label=east:$\mu_{0,n}$] {};
\node at (-3.6,1.3)[label=east:$\mu_{0}^l$] {};
\node at (-3.3,2.5)[label=east:$p_n^l$] {};
\node at (2.5,2.5)[label=east:$p_n^l$] {};
\node at (2.7,3.8)[label=east:$\mu_{1,n}$] {};
\node at (2.8,1.7)[label=east:$\mu_{0}^l$] {};
\node at (-0.3,3.6)[label=east:$\mu_{t,n}$] {};
\node at (-0.3,1.5)[label=east:$\mu_{t}^l$] {};
\node at (-0.6,2.5)[label=east:$W_2$] {};

\end{tikzpicture}

\caption{Approximation procedure for the midpoints}
\label{fig:example}
\end{center}
\end{figure}

\begin{enumerate}

\item We assume that $\supp(\mu_i)\subseteq\mathcal{R}^k_\infty$, for $i\in\{0,1\}$ and fixed $k\in\N$, and construct a geodesic $(\mu_t)_{t \in [0,1]}\subset \PX_2^{ac}(\X_\infty)$ between $\mu_0$ and $\mu_1$ and an optimal plan $q\in\mathsf{Opt}(\mu_0,\mu_1)$, for which the $\CD$-inequality \eqref{def:CD} is fulfilled, relying on the following vertical approximation argument. (Above, and in the following, we write $\mathcal{R}^k_n\subset \X_n$ to denote  the $k$-th $\m_n$-regular set of $\X_n$, $k$-regular set in short, for $n\in\N \cup \{\infty\}$.)

The assumption on the supports allows us to approximate vertically the marginal measures $\mu_0$ and $\mu_1$, by employing a canonical map between Wasserstein spaces $P_n^{k}:\PX_{2}^{ac}(\X_\infty)\to \PX_2^{ac}(\X_n)$, induced via an optimal coupling of the normalized reference measures $p_n^k\in\mathsf{Opt}(\bar{\m}_\infty^k,\bar{\m}_n^k)$. Let us denote these approximations by $(\mu_{n,i})_{n\in\N}$, for $i\in\{0,1\}$. 

At this point we construct the pair $(\mu_t,\,q)$ as the vertical limits of a sequence of geodesics  $(\mu_{n,t})_{t\in[0,1]} \subset \PX_2^{ac}(\X_n)$, between  $\mu_{n,0}$ and $\mu_{n,1}$, and a sequence of optimal plans $q_n \in$ $\mathsf{Opt}(\mu_{n,0},\mu_{n,1})$, both indexed by $n\in\N$. Furthermore, we provide these sequences using the $\CD$-hypothesis on $(\X_n,\di_n,\m_n)$, so in particular we can guarantee that, for $n\in\N$, each pair $\left ( \mu_{n,t}, \;q_n \right)$ satisfies the $\CD$-inequality, for every $t\in[0,1]$.

After demonstrating the lower semicontinuity $S_{\m_\infty}(\mu_t)\leq \liminf S_{\m_n}(\mu_{n,t})$ and upper semicontinuity  $\limsup T_{K,N}^{(t)} (q_n|\m_n) \leq T_{K,N}^{(t)} (q|\m_\infty)$, along our sequences as $n\to\infty$, we conclude the validity of the $\CD$-inequality \eqref{def:CD} for $(\mu_t,\,q)$, for every $t\in[0,1]$.

(Look at Figure \ref{fig:example} for a schematic representation)
\item Additionally, we produce, for $i\in\{0,1\}$, favorable horizontal approximations $(\mu_i^k)_{k\in \N}\subset \PX_2^{ac}(\X_\infty)$, $W_2$-converging to $\mu_i$, whose supports satisfy $\supp(\mu_i^k)\subseteq\mathcal{R}^k_\infty$, for every $k\in\N$. 
Subsequently, by approximating with the pairs constructed in Step 1, we show the existence of the sought geodesic $\{\mu_{t}\}_{t \in [0, 1]} \subset \PX_2^{ac}(\X_\infty)$ and optimal plan $q\in \mathsf{Opt}(\mu_0,\mu_1)$. After showing that the appropriate semicontinuity of the functionals $S_{N',\m_\infty}(\cdot)$ and\\ $T_{K,N'}^{(t)} (\cdot|\m_\infty)$ hold, we are able to verify the $\CD$-condition and conclude.
\end{enumerate}

Let us now indicate where the complications arise. \\By now  the rough idea behind a proof of geometric stability in Wasserstein spaces is well known. More precisely, for well-behaved measures, as a first step one shows  that $\PX(\X_\infty)$ inherits the $\CD$-convexity from the stability of the geometry of $\PX(\X_n)$ under vertical approximations. Afterwards, one can conclude the same property for more general measures, by approximating them horizontally with these well-behaved measures. On the way, the two main challenges that one encounters are, of course, semicontinuity  and precompactness. 

Inspired by techniques used in \cite{Lott-Villani09}, we are able to provide a Legendre-type representation formula for the entropy, which handles one of the functionals in question. At this point, inspired by arguments of Sturm in \cite{Sturm06II}, we conclude  the upper semicontinuity of $T_{K,N'}^{(t)}(\cdot|\m_\infty)$. 

The very challenging obstacles appear then when approaching the problem of the existence of limits and of the convergence of inner points of geodesics. The  general class of metric measure structures which we consider is not even locally  compact, while the ``wildness'' of quasi-Radon measures prohibit us to control the reference measures in any uniform way, preventing us to recover any tightness results from them. Thus,  to overcome these problems original arguments have to be provided here.\\ The crucial ingredient to get back into track will be the control of the mass given by Wasserstein geodesics to $\m$-singular sets when taking the limits and this can be extracted from the $\omega$-uniform convexity.\\

We advance to the presentation of some auxiliary results in the next Section. The vertical approximation argument is presented afterwards in Section \ref{section:approx}, while Step 2. above is discussed in the final Section  \ref{section:CD}.

\subsection{Auxiliary Results}\label{sec:pre}

Collected in this Section are preliminary results, which are needed to prove Theorem \ref{th:stab}. 
We present in a first part results which prove useful in the approximation of $t$-midpoints of geodesics, and semicontinuity results are included in a consecutive Subsection. 

\subsubsection{Approximation and compactness results}
We start by exhibiting the existence of well-behaved horizontal approximations to measures. Recall that for a reference measure $\m \in \M^{qR}(\X)$, $\m^k$ is its $k$-cut defined by \eqref{eq:k-cut}.
\begin{lemma}\label{prop:mk}
Let $(\X,\dis,\m)$ be an metric measure space, $\m \in \M^{qR}(\X)$, and $\mu \in \PX_2^{ac}(\X)$. Then:
\begin{enumerate}
\item[$1)$] The sequence of measures  $\{\m^k\}_{k\in\N}$ approximates $\m$, in the sense of quasi-Radon measures: 
\begin{equation*}
\m^k\weakto \m.
\end{equation*}
\item[$2)$] There exists a sequence of measures $\{ \mu^k \}_{k \in \N} \subset \PX_2^{ac}(\X)$, $\mu^k \ll \m^k$ for any $k \in \N$, converging to $\mu$ in the $W_2$-distance. In particular, for every $k\in\N$, we have that $\supp(\mu^k)\subseteq \mathcal{R}^k:= B_{2^{k+1}}(p)\setminus \mathcal{N}_{2^{-(k+1)}}(\mathcal{S^\m})$, thus these measures have bounded support.
\end{enumerate}
\end{lemma}

\begin{proof}
We start noticing that $1)$ follows directly from the definition of weak convergence because $\supp(f)\subseteq \mathcal{R}^k$ holds eventually, for any function $f\in C_{bs}(X)\cap C_S(X)$.

As for $2)$, let us consider $\mu^k := c^k f^k \mu$, where $f^k$ is the cut-off function defined in \eqref{eq:fk}, $c^k$ is a normalization constant providing $\mu^k(X)=1$, and $k$ is a sufficiently large number, the estimate of which will be determined along the proof. Clearly, $\mu^k \ll \m^k$. At this point we recall that $\supp (f^k) \subset \mathcal R^k$ with $0 \le f^k \le 1$ for any $k \in \N$, and that $f^k \to 1$ pointwise $\m$-almost everywhere as $k \to \infty$. As a consequence,  $f^k \to 1$ pointwise $\mu$-almost everywhere, and $\supp(\mu^k)$ is bounded $\supp(\mu^k)\subset \supp(\m^k)\subset \mathcal{R}^k$.\\ 
By choosing $k_0$ sufficiently large, we can assume that $\supp(\mu)\cap \mathcal{R}^k \not= \emptyset$, for all $k \ge k_0$. (Although the particular choice of $k_0$ does depend on $\mu$, there is no loss of generality, since such a bound exists for every measure $\mu$ and we are interested exclusively in the limit behavior of $\mu^k$.) Let $c^k:= (\int_X f^k \, \d\mu)^{-1}$: in view of the previous remarks, $c^k$ is well-defined and monotone decreasing in $k \in \N$ and $\lim_{k\to \infty}c^k=1$. \\
We can then conclude using the dominated convergence theorem, recalling that a sequence of measures is $W_2$-convergent if and only if it is weakly convergent and the sequence of its second moments is also convergent.
\end{proof}

Recall that, given a coupling $p\in\mathsf{Adm}(\m_X,\m_Y)$, between probability measures $\m_X$ on $(\X,\di_\X)$ and $\m_Y$ on $(\mathrm{Y},\di_\mathrm{Y})$, we can consider a canonical map between its Wasserstein spaces $P:\PX_{2}^{\ll \m_\X}(\X)\to \PX_2^{\ll \m_\mathrm{Y}}(\mathrm{Y})$, which is induced by pushing forward weighted versions of the coupling $p$. We refer to these maps as \emph{Weighted Marginalization}s and we will use them  to produce vertical approximations. 

In detail, for each  $n, k\in \N$ we consider (and fix) an optimal coupling $p_n^k \in  \mathsf{Opt}({\bar\m}_\infty^k,{\bar\m}_n^k)$. Here $\bar\n = \n(\mathsf{Y} )^{-1}\n$ denotes the normalization of a finite measure $\n\in\M(\mathsf{Y})$.
We then write $\{P_{n,k}(x)\}_{x\in Z} \in \PX(x\times \X_n)  \approx\PX(\X_n)$ and $\{P'_{n,k}(y)\}_{y\in Z}\in \PX(\X_\infty \times y) \approx \PX(\X_\infty)$ the disintegrations kernels of the coupling $p_n^k$ with respect to the projection maps $\p_1\colon\X_\infty\times \X_n \to \X_\infty $ and $\p_2\colon\X_\infty\times \X_n \to \X_n $, respectively. 
More precisely,  for $\bar\m_\infty^k$-a.e. $x\in \X_\infty$, and $\bar\m_n^k$-a.e. $y\in \X_n$, we let $P_{n,k}(x)$ and  $P'_{n,k}(y)$ be the measures given by the Disintegration Theorem, which are characterized by
\begin{equation*}
p_n^k(A) = \int_{\X_\infty} P_{n,k}(x)(A) \,\de\bar\m_\infty^k( x) = \int_{\X_n} P'_{n,k}(y)(A)\,\de\bar\m_n^k( y).  
\end{equation*}
for every measurable $A\subset X_\infty \times X_n$. Furthermore, we can define the Weighted Marginalization maps between Wasserstein spaces, via the push forward, along the coordinate projections, of  weighted couplings $f\,p_n^k$. As an abuse of notation, we denote again these maps as $P'_{n,k}$ 
 and $P_{n,k}$. 
 Specifically, let
\begin{equation}\label{def:Pmu}
\begin{aligned}
P'_{n,k}\colon\PX_2(\X_\infty, \dis_\infty, \m_\infty^k) &\to \PX_2(\X_n, \dis_n, \m_n^k) \\
\mu = \rho\, \bar \m_\infty^k  &\mapsto P'_{n,k} (\mu) := (\p_2)_\sharp\, \rho \, p_n^k = \rho'\, \bar \m_n^k, \\
& \text{ with }\; \rho'(y) =  \int_{X_\infty} \rho(x) P'_{n,k}(y)(\d x).
\end{aligned}
\end{equation}
The map $P_{n,k}\colon\PX_2(\X_n, \dis_n, \m_n^k) \to \PX_2(\X_\infty, \dis_\infty, \m_\infty^k)$ is defined in an analogous manner. Note that, in particular, $\rho\,p_n^k \in \mathsf{Adm}(\mu,P'_{n,k}(\mu))$.
The following lemma shows that the well-known properties of the weighted marginalization map $P'_{n,k}$ are still valid in our framework.

\begin{lemma}\label{lemma:I4.19}
Let  $\mu = \rho \m^k_\infty \in \PX_2(\X_\infty)$, then $P_{n,k}'$ satisfies the following properties:
\begin{itemize}
    \item[(i)] For every $N<0$ the functional $S_{N,\cdot}(\cdot)$ satisfies the contraction property:
    \begin{equation}\label{eq:entcontr}
    \begin{split}
         S_{N, \m^k_n}(P'_{n,k}(\mu)) &= \m^k_n(\X_n)^{-\frac 1N}  S_{N, \bar\m^k_n}(P'_{n,k}(\mu))\\ &\leq \m^k_n(\X_n)^{-\frac 1N} S_{N, \bar\m^k_\infty} (\mu)  = \bigg[\frac{\m^k_n(\X_n)}{\m^k_\infty(\X_\infty)}\bigg]^{-\frac 1N} S_{N, \m^k_\infty} (\mu).
    \end{split}
    \end{equation}
    \item[(ii)] If the density $\rho$ of $\mu$ is bounded, then this Wasserstein convergence holds: 
    \begin{equation*}
       W^2_2(\mu, P'_n(\mu)) \le  \displaystyle{\int \dis^2(x, y) \tilde \rho(x)\; \d p_n^k(x, y)} \to 0,
    \end{equation*}
    where $\tilde \rho= \m^k_\infty(\X_\infty) \,\rho$ is the density of $\mu$ with respect to the normalized measure $\bar \m_\infty^k$.
\end{itemize}
\end{lemma}

\begin{proof}
Observe that the two equalities in \eqref{eq:entcontr} are obvious, then we just have to prove the inequality. Consequently (i) follows directly from Jensen's inequality applied to the convex function $\psi(r) := r^{1-\frac{1}{N}}$. Indeed,
\[ \begin{split}
S_{N, \bar\m^k_n}(P'_{n,k}(\mu)) &= \int_{\X_n} \bigg[ \int_{\X_\infty} \tilde \rho(x) P'_{n,k} (y)(\d x) \bigg]^{1-\frac{1}{N}} \, \d \bar\m^k_n(y) \\
&\le \int_{\X_n} \int_{\X_\infty} \tilde{\rho}(x)^{1-\frac{1}{N}} P'_{n,k} (y)(\d x) \, \d \bar \m^k_n(y) = S_{N, \bar \m^k_\infty} (\mu).
\end{split}     \] 
Regarding (ii), notice that the same holds for $\tilde \rho$, since $\rho$ is bounded. Moreover, we have that $\tilde \rho\, p_n^k \in \mathsf{Adm}(\mu,P'_n(\mu))$, and consequently
\begin{equation*}
W^2_2(\mu, P'_n(\mu)) \le  \displaystyle{\int \dis^2(x, y) \tilde \rho(x)\, \d p_n^k(x, y)}\leq  ||\tilde\rho ||_{L^{\infty}(\m_\infty ^k)} W^2_2(\bar\m_\infty^k,\bar\m_n^k) \to 0.
\end{equation*} 
\end{proof}

The last result we are going to prove in this subsection is useful to conclude tightness for a sequence of measures, provided that we have a uniform bound on their R\'enyi entropies, and a tightness condition on the reference measures. The analogous result stated for the relative entropy functional was proved in \cite[Proposition 4.1]{GigMonSav}, and this proof can be easily adapted.

\begin{lemma}\label{lem:tightness}
Let $\{\n_n\}_{n\in \N}, \{\mu_n\}_{n\in \N}\subset \PX(X)$ be two sequences of measures on a complete and separable metric space $(X,d)$,  such that $\{\n_n\}_{n\in \N}$ is tight and $\sup_{n\in\N}S_{N,\n_n}(\mu_n) < \infty$. Then $\{\mu_n\}_{n\in \N}$ is tight.
\end{lemma}
\begin{proof}
First of all we observe that, being the entropy bounded, we can write $\mu_n=\rho_n\n_n$. Thus, a direct application of Jensen's inequality gives that, for every $n\in \N$ and for every measurable set $E\subset X$,
\begin{eqnarray*}
\frac{\mu_n(E)^{1-\frac{1}{N}}}{\n_n(E)^{1-\frac{1}{N}}}
\leq \frac{1}{\n_n(E)}\int_E \rho_n ^{1-\frac{1}{N}} d\n_n 
\leq \frac{S_{N,\n_n}(\mu_n) }{\n_n(E)}.
\end{eqnarray*}
The tightness of $\{\n_n\}_{n\in\N}$ assures the existence of a sequence of compact sets $\{D_l\}_{l\in \N}$ such that $\sup_{n\in\N}\n_n(X\setminus D_l)\to 0$ as $l\to\infty$. We write $E_l=X\setminus D_l$ and we conclude from the above inequality that $\{\mu_n\}_{n\in\N}$ is tight, since
\begin{equation*}
\sup_{n\in \N} \mu_n(E_l)^{1-\frac{1}{N}} \leq \sup_{n\in\N}\n_n(E_l)^{-\frac{1}{N}} \sup_{n\in\N} S_{N,\n_n}(\mu_n)\overset{l\to\infty}{\to} 0.
 \end{equation*}
\end{proof}

\noindent This result can be applied to the extrinsic converging sequence $(\X_n,\di_n,\m_n^k) \xrightarrow{n\to\infty} (\X_\infty,\di_\infty,\m_\infty^k)$, with a straightforward normalization argument by recalling that $\m_n^k(\X_n)$ approaches $\m_\infty^k (\X_\infty)$, as $n\to\infty$.
\begin{cor}\label{lem:tight}
Given a fixed $k\in\N$, and $\{\mu_n\}_{n\in \N}\subset \PX(Z)$ a sequence of probability measures, such that $\sup_{n\in\N}S_{N,\m_n^k}(\mu_n) < \infty$, then $\{\mu_n\}_{n\in \N}$ is tight.
\end{cor}

\subsubsection{Semicontinuity properties}

We present semicontinuity properties of $S_{N, \cdot}(\cdot)$ and $T^{(t)}_{K,N}(\cdot | \cdot)$ conditioned to their domain of definition. We start by proving the lower semicontinuity of $S_{N, \cdot}(\cdot)$ that we anticipated in Section \ref{sec:CD1}. This property is well known in classical frameworks, that is, for positive values of $N$ and well-behaved reference measures. For example, lower semicontinuity for a big class of functionals, in which the R\'enyi entropy is included, was proved in \cite{Lott-Villani09} for locally compact spaces endowed with reference measures having  uniformly bounded volume growth. Inspired by the techniques used in \cite{Lott-Villani09}, we provide below a Legendre-type representation formula for the entropy to attain our result. With this aim, let us write 
\begin{equation*}
 \PX^{\mathcal S}(\X):= \{\mu\in \PX_2(\X) \,: \, \mu(\mathcal S)=0\}.
 \end{equation*}
 
\begin{prop}\label{prop:almostlscofSN}
Let $(\X,\di,p)$ be a pointed Polish space and $\mathcal S\subset \X$ a closed subset with empty interior. Then the R\'enyi entropy functional $S_{N, \n}(\nu)$ is a lower semicontinuous function of $(\n,\nu)\in \M_\mathcal S(\X)\times\PX^{\mathcal S}(\X)$. 
  Specifically, for sequences 
\begin{equation*}
  (\n_n)_{n\in\N \cup \{\infty\}}\subset \M_\mathcal S(\X) \text{ and } (\nu_n)_{n\in\N \cup \{\infty\}}\subset\PX^{\mathcal{S}}(\X),
\end{equation*}    
   such that $\nu_n\overset{*}{\rightharpoonup}\nu$ as quasi-Radon measures and $\nu_n\rightharpoonup\nu$, we have that,
\begin{eqnarray*}
S_{N,\m}(\mu)\leq \liminf_{n\to\infty} S_{N,\m_n}(\mu_n).
\end{eqnarray*}
In particular, the conclusion remains valid under $W_2$-convergence in the second argument.
\end{prop}

\begin{proof}
The semicontinuity of the R\'enyi entropy functional is verified by exhibiting $S_{N, \mm}$ as the supremum of a set of continuous functions on $\M_\mathcal S(\X)\times \PX^{\mathcal S}(\X)$ endowed with the corresponding product convergence. 
In particular we define $\mathcal{R}^k := B_{2^{k+1}}(p) \setminus \mathcal{N}_{2^{-(k+1)}}(\mathcal{S})$ and we show that, for every pair $(\m,\mu) \in \M_\mathcal S(\X)\times \PX^{\mathcal S}(\X)$, 
\begin{equation}\label{eq:lsc}
\begin{aligned}
    S_{N, \mm}(\mu)& = \\ 
  &  \sup \bigg\{ \int F \de \mu &- \int f^* (F) \de \m \,:\,   F\in  C_b(\X) \text{ supported on }\mathcal R^k \text{, for some }k \in \N \bigg\},
\end{aligned}
\end{equation}
where $f^*$ is the convex conjugate function of  $f(x)=|x|^{1-\frac 1N} $, for $x\in\R$. That is, 
\begin{eqnarray*}
f^* : \R &\to & [0,\infty)\\
y &\mapsto & f^*(y):= \sup_{ x \in \R}( y\,x - f(x)) = \left(\frac{N}{N-1}\right)^{1-N} |x|^{1-N}. 
\end{eqnarray*}
For simplicity, let us denote the expression on the right-hand side of \eqref{eq:lsc}  by $\tilde{\mathtt{S}}_{N, \mm}(\mu)$. And do note that, $G(\m,\mu)=\int F \de \mu - \int f^* (F) \de \m $ is indeed continuous in $\M_\mathcal S(\X)\times \PX^{\mathcal S}(\X)$ since, $f^* (F) \in C_{bs}(\X) \cap C_\mathcal S(\X)$, for all  $F\in  C_b(\X)$ with support in $\mathcal R^k \text{, for some }k \in \N$.

We first verify that $S_{N, \mm}(\mu) \geq \tilde{\mathtt{S}}_{N, \mm}(\mu)$, for every $(\m,\mu) \in \M_S(\X)\times \PX^{\mathcal S}(\X)$. 
We  assume that $\mu=\rho\m$ since the aforementioned inequality  is trivially satisfied when $\mu \not\ll \m$. We get the desired result, integrating  the expression $f(z)\geq z\,y^*-f^{*}(y^*)$ with respect to $\m$ which, by definition of $f^{*}$, holds for any $z,y^*\in\R$, after replacing $z=\rho(x)$ and $y^*=F(x)$, for  $ F \in C_b(\X)$ with support inside some $\mathcal R^k$. 

Before proceeding with the converse inequality let us point out  that $\tilde{\mathtt{S}}_{N,\m}(\mu)=\infty$ granted $\mu \not\ll \m$. Indeed, in this case, there exists a Borel set $A\subset \X$ with $\m(A)=0$ and $\mu(A)>0$. At this point, recall that every Borel finite measure in  a Polish spaces is inner regular with respect to compact sets and outer regular with respect to open sets (see for instance \cite[Theorem 7.1.7]{Bogachev07}). 
For this reason, since $\mu \in \PX^\mathcal{S}(\X)$, it is possible to assume that $A \cap S = \emptyset$, and being $A$ a compact set. Observe also that compactness grants the existence of $k\in \N$ for which $A\subset \mathcal R^k$.   Since $\m$ and $\mu$ restricted to $\mathcal R^{k+1}$ are finite measures, there exists a sequence of compact sets $(K_{n})_{n\in\N}$ and one of open sets $(A_{n})_{n\in\N}$, such that $K_{n}\subset A\subset A_n \subset \mathcal{R}^{k+1}$ and $(\mu+\m)(A_n\setminus K_n)<1/n$, for any $n\in\N$. Then, for $M>0$, Tietze's Theorem ensures the existence of a sequence of approximating functions $\left(F^M_n\right)_{n\in\N} \subset C_b(X)$ satisfying: $0\le F^M_n\le M$, $F^M_n=M$ on $K_n$, and $F^M_n=0$ on $X \setminus A_n$. Therefore, as $n\to\infty$, the functions $F^M_n$ converge, in $L^1(\mu+\m)$, to the scaled characteristic function $ M\chi_{A}$. On the other hand, since for every $n$, $F^M_n$ is an admissible function for the supremum in $ \tilde{\mathtt{S}}_{N, \mm}$, we have that
\begin{equation}\label{eq:cond}
\int F^M_n\de\mu-\int f^{*}(F^M_n)\de\m\leq \tilde{\mathtt{S}}_{N, \mm}(\mu)<+\infty.
\end{equation}
for any $n\in\N$. Passing now to the limit as $n$ goes to infinity in \eqref{eq:cond}, we obtain that $M \cdot\mu (A) \leq\tilde{\mathtt{S}}_{N, \mm}(\mu)$. The arbitrariness of $M$ implies then that $\tilde{\mathtt{S}}_{N,\m}(\mu)=\infty$. 

We proceed now to prove that $S_{N, \mm}(\mu) \leq \tilde{\mathtt{S}}_{N, \mm}(\mu)$ and we will assume that $\tilde{\mathtt{S}}_{N, \mm}(\mu)<+\infty$, since otherwise there is nothing to prove.  
The paragraph above enables us to write $\mu=\rho\m$. We then have the following expression for $\tilde{\mathtt{S}}_{N, \mm}(\mu)$:
\begin{equation*}
\tilde{\mathtt{S}}_{N, \mm}(\mu)= \sup \left\{ \int F \rho - f^* (F) \,\de \m \, :  F\in C_b(\X) \text{ supported on }\mathcal R^k \text{, for some }k \in \N \right\}.
\end{equation*}
And, recalling that $f(x)=(f^*)^*(x)=\sup_{y\in \R} \{x\,y - f^*(y)\}$ because $f$ is finite, convex and continuous, it follows that  
\begin{equation*}
S_{N, \mm}(\mu) = \int \sup_{s^*\in\mathbb Q}\{\rho(x)s^*-f^*(s^*)\}\de\m(x).
\end{equation*}  
By fixing $\mathbb Q=\{q_n\}_{n\in\N}$, an enumeration of rational numbers with $q_0=0$, we introduce the family of approximating functionals
\begin{equation*}
\left \{ S_{N, \mm}^{h}(\mu):=\int \sup_{s^*\in\{q_0,\dots,q_h\}}\{\rho(x)s^*-f^*(s^*)\}\de\m(x)\right\}_{h\in \N}.
\end{equation*}
Observe that the integrands are monotone increasing in $h$ and that $0 =  \rho(x) q_0-f^{*}(q_0)$. In particular, Beppo-Levi's Theorem ensures that $S_{N, \mm}^{h}(\mu) \to S_{N, \mm}(\mu)$, as $h\to \infty$.  Therefore, it suffices to show that $ S_{N, \mm}^{h}\leq \tilde{\mathtt{S}}_{N, \mm} $, for any fixed $h\in\N$. To this aim, one confirms directly that 
\begin{equation}\label{eq:kenergy}
S_{N, \mm}^{h}=\sup\left\lbrace \int(\rho F-f^*(F))\de \m\,:\, F\text{ is a step function with values in $\{q_0,\dots,q_h\}$}\right\rbrace.
\end{equation}
Note that the fact that $S$ is an $\m$-null set guarantees that we can further require that the aforementioned functions are supported in $\mathcal R^k$ for some $k\in \N$ without modifying the supremum, as an approximation argument using the Monotone Convergence Theorem shows. Finally, since $\m$ and $\mu$ are finite measures when restricted to $\mathcal R^{k+1}$, every step function with support in $\mathcal R^k$ can be obtained as the $L^1(\mu+\m)$-limit of continuous and uniformly bounded functions implying that $S_{N, \mm}^{h}\leq \tilde{\mathtt{S}}_{N, \mm}  $, which concludes the proof.
\end{proof}

A couple of observations suffice to extract a useful corollary for extrinsic converging sequences of metric measure spaces $\{(\X,\di_n,\m_n)\}_{n\in \N\cup\{\infty\}}$. From Lemma \ref{lemma:I4.19} we deduce that, for any $k\in\N$, the sequence of vertical projections $(\m_n^k)_{n\in\N}$ approximates $\m_\infty^k\in\M( \X)\subset \M_\emptyset( \X)$ in narrow convergence. Moreover, Remark \ref{re:weak_narrow} states that, granted we restrict ourselves to the set  $\M( Z)$ on a Polish space $(Z,\dis_Z)$ and set $S=\emptyset$, then weak convergence in the sense of quasi-Radon measures  coincides with the usual weak one. 
\begin{cor}\label{le:lsc_ent}
Given a fixed $k\in \N$ and $\{\mu_n\}_{n\in \N}\subset \PX_2(Z)$ a sequence converging weakly to $\mu\in\PX_2(Z)$. Then it holds that  
\begin{equation}\label{eq:lsc_ent}
S_{N,\m^k_\infty}(\mu)\leq \liminf_{n\to \infty}S_{N,\m^k_n}(\mu_n).
\end{equation}
\end{cor}

We conclude the section with a corresponding continuity result for the functional $T^{(t)}_{K,N}$. We stress that although, it would be sufficient for the proof of Theorem \ref{th:stab} to have the upper semicontinuity of $T^{(t)}_{K,N}$, we prefer to present a more general statement.

\begin{prop}\label{prop:continuityofT}
Let $K\geq 0$ and $N<0$ and $(\X, \di, \m)$ be a metric measure space. Furthermore, set $\mu_0=\rho_0\m,\mu_1=\rho_1\m\in\PX(\X)$ to be absolutely continuous with respect to the quasi-Radon reference measure $\m$, with $S_{N, \mm}(\mu_0), S_{N, \mm}(\mu_1)<\infty$. Consider a sequence $(\pi_n)_{n\in \N}\subset \PX(\X\times \X)$, weakly converging to $\pi\in \PX(\X\times \X)$ and such that 
\begin{equation*}
    (\p_1)_\# \pi_n = \mu_0 \quad \text{and} \quad (\p_1)_\# \pi_n = \mu_1 \qquad \text{for every }n\in \N.
\end{equation*}
Then, for $t\in[0,1]$, it holds that 
\begin{equation*}
    \lim_{n\to \infty} T^{(t)}_{K,N}(\pi_n|\m)= T^{(t)}_{K,N}(\pi|\m).
\end{equation*}
Additionally, the conclusion remains valid for $K<0$, granted $\diam(\X)<\pi \sqrt{\frac{N-1}{K}}$.
\end{prop}

\begin{proof}
Set $t\in(0,1)$, since the statement is clearly ture for the remaining values. We only proceed to prove that, 
\begin{equation}\label{eq:contofT}
   \lim_{n\to \infty} \int_{X\times X} \tau_{K,N}^{(1-t)} (\di(x,y)) \rho_0(x)^{-\frac 1N} \de \pi_n = \int_{X\times X} \tau_{K,N}^{(1-t)} (\di(x,y)) \rho_0(x)^{-\frac 1N} \de \pi ,
\end{equation}
since the other term of $T^{(t)}_{K,N}(\cdot |\m)$ can be treated analogously. Notice that, since $S_{N, \mm}(\mu_0)<\infty$, then $\rho_0(x)^{- 1/N}\in L^1(\mu_0)$ thus, according to Lemma \ref{lemma:dens},
for every fixed $\varepsilon>0$ there exists $f^\varepsilon\in C_b(\X)$ such that $\norm{\rho_0^{- 1/N}-f^\varepsilon}_{L^1(\mu_0)}<\varepsilon$. Moreover, notice that  the coefficients $\tau_{K,N}^{(1-t)}(\cdot | \m)$ are bounded and continuous.  Indeed, this is always the case for $K\geq 0$, and since $\diam(X) < \pi \sqrt{\frac{N-1}{K}}$ is bounded by our assumptions, this holds as well for $K<0$. Therefore,
\begin{equation*}
    \tau_{K,N}^{(1-t)} (\di(x,y)) f^\varepsilon(x)
\end{equation*}
is itself bounded and continuous. Consequently, the weak convergence $(\pi_n)_n \rightharpoonup \pi$ shows that,
\begin{equation*}
    \lim_{n\to \infty} \int_{X\times X} \tau_{K,N}^{(1-t)} (\di(x,y)) f^\varepsilon(x) \de \pi_n = \int_{X\times X} \tau_{K,N}^{(1-t)} (\di(x,y)) f^\varepsilon(x) \de \pi.
\end{equation*}
Furthermore, the boundedness of $\tau_{K,N}^{(1-t)}$ allows to deduce the following estimate
\begin{equation*}
    \begin{split}
        \limsup_{n\to \infty} \int_{X\times X} \tau_{K,N}^{(1-t)} (\di(x,y)) \rho_0(x)^{-\frac 1N} \de \pi_n &\leq \lim_{n\to \infty} \int_{X\times X} \tau_{K,N}^{(1-t)} (\di(x,y)) f^\varepsilon(x) \de \pi_n + \varepsilon  \norm{\tau_{K,N}^{(1-t)}}_{L^\infty} \\
        &= \int_{X\times X} \tau_{K,N}^{(1-t)} (\di(x,y)) f^\varepsilon(x) \de \pi + \varepsilon  \norm{\tau_{K,N}^{(1-t)}}_{L^\infty}\\
        &\leq \int_{X\times X} \tau_{K,N}^{(1-t)} (\di(x,y)) \rho_0(x)^{-\frac 1N} \de \pi + 2  \varepsilon  \norm{\tau_{K,N}^{(1-t)}}_{L^\infty}.
    \end{split}
\end{equation*}
Analogously,  it follows that 
\begin{equation*}
    \liminf_{n\to \infty} \int_{X\times X} \tau_{K,N}^{(1-t)} (\di(x,y)) \rho_0(x)^{-\frac 1N} \de \pi_n \geq \int_{X\times X} \tau_{K,N}^{(1-t)} (\di(x,y)) \rho_0(x)^{-\frac 1N} \de \pi - 2  \varepsilon  \norm{\tau_{K,N}^{(1-t)}}_{L^\infty},
\end{equation*}
and since $\varepsilon>0$ and can be chosen arbitrarily, equation	\eqref{eq:contofT} holds true. We conclude by recalling the arbitrariness of $t$.\\
\end{proof}

\subsection{Proof of the Approximate CD Condition}\label{section:approx}

The objective of this section is to prove the next partial result

\begin{thm}\label{thm:approxCD}
Let $K\in\R$, $N\in (\infty,0)$, and $\big\{ (\X_n,\dis_n,\m_n,  \mathcal S_{\mm_n}, p_n)\big \}_{n\in \N \cup \{\infty\}} \subset \mathcal M^{qR}_{\bar k}$ be a sequence of metric measure spaces   
satisfying the assumptions of the Stability Theorem \ref{th:stab}, for some $\bar k \in \N$. 
Then $(\X_\infty,\dis_\infty,\m_\infty)$ is a $\CD^a(K,N)$ space.
\end{thm}
We follow the plan discussed at the beginning of the section and argue, in steps, using vertical approximations. Specifically, Steps 1 and 2 serve the purpose of constructing useful approximations of the marginals $\mu_0$ and $\mu_1$. 
Next, we follow the argumentation of Sturm in \cite{Sturm06II} in Steps 3 to 6, to exhibit the upper semicontinuity of $T^{(t)}_{K,N}$ along a sequence of optimal couplings, provided by the curvature-dimension assumption.  Additionally, we demonstrate the existence of a favorable limiting optimal coupling. Step 7 focuses in proving the convergence of inner points of a vertical sequence of Wasserstein geodesics, as well as, the lower semicontinuity of the Renyi entropy along this sequence.\\

%

Let us fix first the notation. Set $k \ge \bar k$ and assume that $\mu_0, \mu_1 \in \PX_2^{ac}(\X_\infty)$ have supports satisfying $\supp(\mu_0),\,\supp(\mu_1) \subset \mathcal R_\infty^{k-1}$. We denote by $\rho_i$ the density of $\mu_i$ with respect to $\m$, for $i=\{0,1\}$. Define the set, 
\begin{equation*}
    I:= \{ N'\in [N,0)\, :\, S_{N',\m_\infty}(\mu_0),S_{N',\m_\infty}(\mu_1)<\infty\},
\end{equation*}
and observe that $I$ is an interval, as a consequence of Jensen's inequality. Then, surely, we are able to assume that (and set)
\begin{equation*}
  (\, M:=\,)\;\max \left\{S_{N,\m_\infty}(\mu_0),\,S_{N,\m_\infty}(\mu_1)\right\}<\infty,
\end{equation*}
and that, for every $q\in \mathsf{Opt}(\mu_0,\mu_1)$, 
\begin{equation}\label{eq:finitenessofT}
    T^{(t)}_{K, N'}(q | \m_\infty) < \infty,
\end{equation}
since the $\CD$-condition is trivial in  failure of any of these inequalities.  Therefore, the arbitrariness of $k$ and initial measures shows that, in order to demonstrate Theorem \ref{thm:approxCD}, we are required to validate the $\CD$-inequality \eqref{def:CD}, for every $N'\in I$, and every $t\in[0,1]$.

Additionally, we fix as before an optimal coupling $p_n\in\mathsf{Opt}\left(\bar\m_\infty^k,\bar\m_n^k\right)$ between the normalized $k$-cuts of the reference measures, for every $n\in \N$. And, as defined in Section \ref{sec:pre}, we consider $\{P_{n}(x)\}_{x\in X^\infty} \in  \PX(\X_n)$ and $\{P'_{n}(y)\}_{y\in X_n}\in  \PX(\X_\infty)$  the disintegration kernels of $p_n$ with respect to the projections $\p_1$ and $\p_2$ respectively, and consider the map $P'_n\colon\PX_2(\X_\infty, \dis_\infty, \m_\infty^k) \to \PX_2(\X_n, \dis_n, \m_n^k)$. Note that, in contrast to Section \ref{sec:pre}, here we have omitted the dependence on the number $k$, since it is fixed for now. \\
%

\textbf{STEP 1: Horizontal approximation with bounded densities}\\
During the argument it proves useful to work with bounded-density measures. Therefore we construct here a horizontal approximation of $\mu_0$ and $\mu_1$, for which its elements enjoy this property. \\ For the construction, we fix an arbitrary optimal coupling $\tilde{q} \in\mathsf{Opt}(\mu_0,\mu_1)$ and define, for every $r>0$, 
\begin{equation}
    E_r := \{ (x_0, x_1) \in X_\infty\times X_\infty \,:\, \rho_0(x_0) < r, \rho_1(x_1) < r  \}
\end{equation}
and consequently, for sufficiently large $r$,
\begin{equation*}
    \tilde{q}^{(r)} := \alpha_r^{-1} \tilde{q}(\, . \cap E_r),
\end{equation*}
where $\alpha_r := \tilde{q}(E_r)$.
The measure $\tilde{q}^{(r)} \in \PX(\X_\infty \times \X_\infty)$ has marginals given by
\begin{equation}\label{margqtilde} 
\mu_0^{(r)} :=  (\p_{ 1})_{ \sharp} \, \tilde{q}^{(r)} \quad \text{ and } \quad \mu_1^{(r)} :=  (\p_{ 2})_{ \sharp} \, \tilde{q}^{(r)}.
\end{equation}
 Notice that both $\mu_0^{(r)}$ and $\mu_1^{(r)}$ have bounded densities and that $\mu_i^{(r)}$ converges to $\mu_i$ in $(\PX^2(\X_\infty),W_2)$, for $i=0, 1$. Moreover, notice that $S_{N,\m_\infty}(\mu_i^{(r)}) \to S_{N,\m_\infty}(\mu_i)$ as $r\to \infty$, for $i=0,1$. Then we fix $\varepsilon>0$ and find $r =r (\varepsilon)$ such that $\alpha_r\geq 1-\varepsilon$ and that the following estimates holds:
\begin{equation}\label{mutildeclose}
\max_{i\in\{1,2\}}W_2 (\mu_i, \mu_i^{(r)}) \le \varepsilon \quad \text{ and }\quad  \max_{i\in\{1,2\}}S_{N,\m_\infty^k} (\mu_i^{(r)})=\max_{i\in\{1,2\}}S_{N,\m_\infty} (\mu_i^{(r)})\leq M +\frac 12.
\end{equation}
We point out that the parameter $r$ depends on $\varepsilon$, but we won't be explicit on this dependence for the sake of the presentation. \\

\textbf{STEP 2: Vertical approximation}\\
Once we have identified the horizontal approximations $\mu_0^{(r)}$ and $\mu_1^{(r)}$, we may proceed to their vertical approximation. First of all, observe that $\mu_0^{(r)}$ and $\mu_1^{(r)}$ are absolutely continuous with respect to the normalized reference measure $\bar\m_\infty^k$, so we denote by $\tilde{\rho}_0^{(r)}$ and $\tilde{\rho}_1^{(r)}$ their bounded densities. Then, for every $n \in \N$, we define $\mu_{0,n}, \,\mu_{1,n} \in \PX_2(\X_n, \dis_n, \m_n^k)$ as
\begin{equation}\label{muin} 
\mu_{i, n} := P_n'(\mu_i^{(r)}) = \rho_{i, n} \bar\m_n^k,
\end{equation}
where $\rho_{i, n}(y) = \int \tilde{\rho}_i^{(r)} (x)P_n'(y)(\d x)$.
Notice that $\mu_{0,n}$ and $\mu_{1,n}$ depend on $r$ (and ultimately on $\varepsilon$), but, once again, we prefer not to make this dependence explicit, in order to maintain an easy notation in the following. Anyway, we invite the reader to keep in mind that every object we are going to define depends only on $\varepsilon$.
Now, since $\m^k_n(\X_n)\to\m^k_\infty(\X_\infty)$, observe that Lemma \ref{lemma:I4.19} guarantees the existence of an $\bar{n} \in \N$, such that if $n \ge \bar{n}$ it holds that
\begin{equation}\label{muclose}
\max_{i\in\{1,2\}} W_2(\mu_i^{(r)}, \mu_{i, n})  \leq \varepsilon.
\end{equation}
and that  
\begin{equation}\label{eq:entropyboundM}
    \max_{i\in\{1,2\}} S_{N,\m_n}(\mu_{i,n})\leq \max_{i\in\{1,2\}} S_{N,\m_n^k}(\mu_{i,n})\leq M+1.
\end{equation}
Moreover, according to Lemma \ref{lemma:I4.19}, for every $N'\in I$, it holds that
\begin{equation*}
      S_{N', \m_n}(\mu_{i,n}) \leq S_{N', \m^k_n}(\mu_{i,n})\leq \bigg[\frac{\m^k_n(\X_n)}{\m^k_\infty(\X_\infty)}\bigg]^{-\frac 1{N'}} S_{N', \m^k_\infty} (\tilde\mu^{(r)}_i) <\infty.
\end{equation*}
 Therefore, for every $n\in \N$ large enough, since $(\X_n, \di_n, \m_n)$ is a $\CD(K,N)$ space,  there exist an optimal plan $\pi_n\in \mathsf{Opt}(\mu_{0,n},\mu_{1,n})$ and a 2-Wasserstein geodesic $(\mu_{t,n})_{t \in [0,1]} \subset  \mathscr{P}_{2}^{ac}(\X_n) $ connecting $\mu_{0,n}$ and $\mu_{1,n}$, for which,
\begin{equation}\label{eq.CDatn}
    S_{N',\m_n}(\mu_{t,n}) \leq T_{K,N'}^{(t)}(\pi_n | \m_n)
\end{equation}
holds, for every $t\in [0,1]$ and every $N'\in I$. Note that Remark \ref{re:bounded} together with the assumption that $\sup_{i\in \N \cup \{\infty\}} \diam(\X_i,\di_i) < \pi|K|^{-1/2} $, if $K<0$, assures that the geodesic $\mu_{t,n}$ is absolutely continuous with respect to $\m_n$.\\

\textbf{STEP 3: Estimate for $T_{K,N'}^{(t)}$ }\\
In this step we start  the proof of the upper semicontinuity of the functional $T_{K,N'}^{(t)}$.
In particular, we fix $N'\in [N, 0)$ and a time $t \in [0, 1]$ and we call $Q_n$ and $Q'_n$ be the disintegrations of $\pi_n$ with respect to $\mu_{0, n}$ and $\mu_{1, n}$ respectively. Then we define the following two functions
\[  v_0(y_0) =  \int_{\X_n} \tau^{(1-t)}_{K, N'} (\dis(y_0, y_1)) Q_n(y_0, d y_1)  \]
and 
\[ v_1(y_1) = \int_{\X_n} \tau^{(t)}_{K, N'} (\dis(y_0, y_1)) Q'_n(y_1, d y_0).  \] 
A direct application of Jensen's theorem leads to 
\[
\begin{split}
 T_{K, N'}^{(t)}(\pi_n |\bar \m_n^k) &= \displaystyle \sum_{i = 0}^1 \, \int_{\X_n} \rho_{i, n}(y_i)^{1-1/N'} \cdot v_i (y_i) \, \d \bar\m^k_n(y_i)\\
&= \sum_{i = 0}^1 \, \int_{\X_n} \bigg[ \int_{\X_\infty} \tilde{\rho}^{(r)}_{i}(x_i) P'_n(y_i, \d x_i)\bigg]^{1-1/N'} \cdot v_i (y_i) \, \d \bar\m^k_n(y_i)\\
&\le \sum_{i=0}^1 \, \int_{\X_n} \int_{\X_\infty} \tilde{\rho}^{(r)}_{i}(x_i)^{1-1/N'} P'_n(y_i, \d x_i)  \cdot v_i (y_i) \, \d \bar\m^k_n(y_i)\\
&= \sum_{i=0}^1 \, \int_{\X_\infty} \tilde{\rho}^{(r)}_{i}(x_i)^{1-1/N'} \bigg[ \int_{\X_n} v_i(y_i) P_n(x_i, \d y_i)   \bigg] \, \d\bar\m^k_\infty(x_i).
\end{split}
\]
\noindent At this point we see that 
\begin{align*}
\displaystyle \int_{\X_n}  v_0(y_0) & P_n(x_0, \d y_0) = \int_{\X_n \times \X_n} \tau_{K, N'}^{(1-t)}(\dis(y_0, y_1)) Q_n(y_0, \d y_1) P_n(x_0, \d y_0) &\\
&=\int_{\X_n \times \X_n\times \X_\infty} \tau_{K, N'}^{(1-t)}(\dis(y_0, y_1))\dfrac{\tilde{\rho}_1^{(r)}(x_1)}{\rho_{1, n}(y_1)}P_n'(y_1, \d x_1) Q_n(y_0, \d y_1) P_n(x_0, \d y_0) \\
& \le \int_{\X_n \times \X_n \times \X_\infty} \big[ \tau_{K, N'}^{(1-t)}(\dis(x_0, x_1)) + C \cdot |\dis(y_0, y_1) - \dis(x_0, x_1)| \big]& \\ 
& \hspace{4.5cm} \dfrac{\tilde{\rho}_1^{(r)}(x_1)}{\rho_{1, n}(y_1)}P_n'(y_1, \d x_1)Q_n(y_0, \d y_1) P_n(x_0, \d y_0) &\\
&\le \int_{\X_n \times \X_n \times \X_\infty} \big[ \tau_{K, N'}^{(1-t)}(\dis(x_0, x_1)) + C \cdot (\dis(x_0, y_0) + \dis(x_1, y_1)) \big]& &\\ 
& \hspace{4.5cm} \dfrac{\tilde{\rho}_1^{(r)}(x_1)}{\rho_{1, n}(y_1)}P_n'(y_1, \d x_1) Q_n(y_0, \d y_1) P_n(x_0, \d y_0) ,
\end{align*} 
and analogously that
\begin{align*}
\displaystyle \int_{\X_n}  v_1(y_1)  P_n(x_1, \d y_1) & \le \int_{\X_n \times \X_n \times \X_\infty} \big[ \tau_{K, N'}^{(1-t)}(\dis(x_0, x_1)) + C \cdot (\dis(x_0, y_0) + \dis(x_1, y_1)) \big]& &\\ 
& \hspace{4.5cm} \dfrac{\tilde{\rho}_0^{(r)}(x_0)}{\rho_{0, n}(y_0)}P_n'(y_0, \d x_0) Q_n(y_1, \d y_0) P_n(x_1, \d y_1), \vspace{-0.3cm}
\end{align*} 
where $C := \max_{\theta \in [0, M], s\in [0, 1]} \frac{\partial}{\partial \theta} \tau_{K, N'}^{(s)}(\theta)$ and $ M$ is the maximum between $\dis (x_0, x_1)$ and $\dis(x_1, y_1)$. 
Observe that the constant $C$ is indeed finite because we know that  $\sup_{i\in \N \cup \{\infty\}} \text{diam}(\X_i,\di_i) < \pi\sqrt{\frac{1}{|K|}}$, if $K<0$, from assumption (iii) in Theorem \ref{thm:approxCD}.\\
Moreover we notice that 
\begin{equation}\label{eq:estimate1}
\begin{split}
    \int_{\X_\infty} \tilde{\rho}^{(r)}_{i}(x_0)^{1-1/N'}   &\int_{\X_n \times \X_n \times \X_\infty} \dis(x_0, y_0) \dfrac{\tilde{\rho}_1^{(r)}(x_1)}{\rho_{1, n}(y_1)}P_n'(y_1, \d x_1) Q_n(y_0, \d y_1) P_n(x_0, \d y_0)\, \d\bar\m^k_\infty(x_0)\\
    &=\int_{\X_\infty} \tilde{\rho}^{(r)}_{i}(x_0)^{1-1/N'}   \int_{\X_n } \dis(x_0, y_0)  P_n(x_0, \d y_0)\, \d\bar\m^k_\infty(x_0) \\
    &\leq r^{1-1/N'} \int_{\X_n\times \X_\infty} \dis(x_0, y_0) \de p_n^k(x_0,y_0) \leq r^{1-1/N'} W_2(\bar\m_n^k,\bar\m_\infty^k)
\end{split}
\end{equation}
and
\begin{equation}\label{eq:estimate2}
\begin{split}
    \int_{\X_\infty} \tilde{\rho}^{(r)}_{0}(x_0)^{1-1/N'}   &\int_{\X_n \times \X_n \times \X_\infty} \dis(x_1, y_1) \dfrac{\tilde{\rho}_1^{(r)}(x_1)}{\rho_{1, n}(y_1)}P_n'(y_1, \d x_1) Q_n(y_0, \d y_1) P_n(x_0, \d y_0)\, \d\bar\m^k_\infty(x_0)\\
    &\leq r^{-1/N'}  \int_{\X_n \times \X_n \times \X_\infty\times \X_\infty}   \dis(x_1, y_1) \dfrac{\tilde{\rho}_1^{(r)}(x_1)}{\rho_{1, n}(y_1)}P_n'(y_1, \d x_1) \\
    & \qquad\qquad\qquad\qquad\qquad\qquad\qquad Q_n(y_0, \d y_1) P_n(x_0, \d y_0)\tilde{\rho}^{(r)}_{0}(x_0)\, \d\bar\m^k_\infty(x_0) \\
    &=r^{-1/N'}  \int_{\X_n \times \X_n \times \X_\infty}   \dis(x_1, y_1) \dfrac{\tilde{\rho}_1^{(r)}(x_1)}{\rho_{1, n}(y_1)}P_n'(y_1, \d x_1) Q_n(y_0, \d y_1) \mu_{0,n}(\de y_0)\\
     &=r^{-1/N'}  \int_{\X_n \times \X_\infty}   \dis(x_1, y_1) \tilde{\rho}_1^{(r)}(x_1)P_n'(y_1, \d x_1)\de \bar\m_n^k(y_1)\\
     &\leq r^{1-1/N'}  \int_{\X_n \times \X_\infty} \di(x_1,y_1)  \de p_n^k(x_1, y_1) \leq r^{1-1/N'} W_2(\bar\m_n^k,\bar\m_\infty^k),
\end{split}
\end{equation}
where the last inequality in both chains follows by the Jensen's inequality.
Consequently, for every $n\in\N$, we define a -- not necessarily optimal -- coupling $\bar{q}_n^{(r)}\in \mathsf{Adm}(\mu_0^{(r)},\mu_1^{(r)})$  by imposing that
\[
\begin{split}
\d \bar{q}_n^{(r)}(x_0, x_1) &= \displaystyle \int_{\X_n \times \X_n} \dfrac{\tilde{\rho}_0^{(r)}(x_0) \tilde{\rho}_1^{(r)}(x_1)}{\rho_{0, n}(y_0) \rho_{1, n}(y_1)} P_n'(y_1, \d x_1) P_n'(y_0, \d x_0) \, \d\pi_n(y_0, y_1)\\
&= \int_{\X_n \times \X_n} \dfrac{\tilde{\rho}_0^{(r)}(x_0) \tilde{\rho}_1^{(r)}(x_1)}{ \rho_{1, n}(y_1)} P_n'(y_1, \d x_1)  Q_n(y_0, \d y_1) P_n(x_0, \d y_0) \, \d \bar\m_\infty^k(x_0)\\
&= \int_{\X_n \times \X_n} \dfrac{\tilde{\rho}_0^{(r)}(x_0) \tilde{\rho}_1^{(r)}(x_1)}{ \rho_{0, n}(y_0)} P_n'(y_1, \d x_1)  Q'_n(y_1, \d y_0) P_n(x_0, \d y_0) \, \d \bar\m_\infty^k(x_1).
\end{split}
\]\\
With this definition of $\bar{q}^{(r)}$ and keeping in mind \eqref{eq:estimate1} and \eqref{eq:estimate2}, we end up with
\[
T_{K, N'}^{(t)}(\pi_n | \bar\m_n^k)\le T^{(t)}_{K, N'}(\bar{q}^{(r)} | \bar\m_\infty^k) + 4 C r^{1-1/N'}W_2(\bar\m_\infty^k, \bar\m_n^k).
\]
Now, up to taking a greater $\bar n$, we can require that for every $n\geq \bar n$ it holds that 
\[
W_2(\bar\m_\infty^k, \bar\m_n^k) \le  \dfrac{\varepsilon}{4 C r^{\frac{N'-1}{N'}}} ,
\]
for every $N'\in [N, \varepsilon)$.
As a consequence we obtain that
\begin{equation}\label{eq:Tfixmarg}
T_{K, N'}^{(t)}(\pi_n |\bar \m^k_n)  \le T_{K, N}^{(t)}(\bar{q}_n^{(r)} | \bar\m^k_\infty) + \varepsilon,
\end{equation}
for every $n\geq \bar n$ and every $N'\in [N, \varepsilon)$.\\

\textbf{STEP 4: $q^{(r)}_n$ converges to an optimal plan}\\
The objective now is to prove that 
\begin{equation}\label{eq:aimofstep}
    \int \di^2(x_0,x_1) \de \bar q^{(r)}_n(x_0,x_1)\to W^2_2(\mu_0^{(r)},\mu_1^{(r)}) \qquad \text{as }n \to \infty.
\end{equation}

First of all notice that, since every $\bar q^{(r)}_n$ is an admissible plan between $\mu_0^{(r)}$ and $\mu_1^{(r)}$, then for every $n\in \N$ it holds that
\begin{equation}\label{eq:boundbelow}
    \int \di^2(x_0,x_1) \de \bar q^{(r)}_n(x_0,x_1)\geq W^2_2(\mu_0^{(r)},\mu_1^{(r)}).
\end{equation}
On the other hand the triangular inequality ensures that
\begin{equation*}
    \di(x_0,x_1) \leq \di (x_0,y_0)+\di (y_0,y_1) + \di (x_1,y_1)
\end{equation*}
and consequently, since $\di(y_0,y_1)<\text{diam}(\mathcal R ^k_n)\leq 2^{k+2}$ for $\pi_n$-almost every pair $(y_0,y_1)$, we have that
\begin{equation*}
    \di^2(x_0,x_1)-\di^2(y_0,y_1) \leq 2\di^2 (x_0,y_0)+ 2\di^2 (x_1,y_1) + 2^{k+3} \di (x_0,y_0)+ 2^{k+3} \di (x_1,y_1)
\end{equation*}
for $\pi_n$-almost every pair $(y_0,y_1)$. It is then possible to perform the following estimate
\begin{align*}
    \int_{\X_\infty \times \X_\infty}\di^2(x_0,x_1) &\de \bar q^{(r)}_n (x_0,x_1) \\
    &= \int_{\X_\infty \times \X_\infty}\di^2(x_0,x_1) \int_{\X_n \times \X_n} \dfrac{\tilde{\rho}_0^{(r)}(x_0) \tilde{\rho}_1^{(r)}(x_1)}{\rho_{0, n}(y_0) \rho_{1, n}(y_1)} P'_n(y_1, \d x_1) P'_n(y_0, \d x_0) \, \d\pi_n(y_0, y_1)\\
    &\leq \int_{\X_n \times \X_n} \di^2(y_0,y_1) \de \pi_n (y_0,y_1) \\
    &\quad+ \int_{\X_\infty } \int_{\X_n \times \X_n} 2\di^2(x_0,y_0) \dfrac{\tilde{\rho}_0^{(r)}(x_0) }{\rho_{0, n}(y_0)}  P'_n(y_0, \d x_0) \, \d\pi_n(y_0, y_1)\\
    &\quad+\int_{\X_\infty} \int_{\X_n \times \X_n} 2^{k+3} \di(x_0,y_0) \dfrac{\tilde{\rho}_0^{(r)}(x_0) }{\rho_{0, n}(y_0)}  P'_n(y_0, \d x_0) \, \d\pi_n(y_0, y_1) \\
    &\quad+ \int_{\X_\infty } \int_{\X_n \times \X_n} 2\di^2(x_1,y_1) \dfrac{\tilde{\rho}_1^{(r)}(x_1) }{\rho_{1, n}(y_1)}  P'_n(y_1, \d x_1) \, \d\pi_n(y_0, y_1)\\
    &\quad+\int_{\X_\infty } \int_{\X_n \times \X_n} 2^{k+3} \di(x_1,y_1) \dfrac{\tilde{\rho}_1^{(r)}(x_1) }{\rho_{1, n}(y_1)}  P'_n(y_1, \d x_1) \, \d\pi_n(y_0, y_1) 
\end{align*}
We can now consider one term at the time and start by noticing that, according to Lemma \ref{lemma:I4.19},
 \begin{align*}
     \int_{\X_\infty } \int_{\X_n \times \X_n} 2\di^2(x_0,y_0) \dfrac{\tilde{\rho}_0^{(r)}(x_0) }{\rho_{0, n}(y_0)}  &P'_n(y_0, \d x_0) \, \d\pi_n(y_0, y_1) \\
     &= \int_{\X_\infty } \int_{\X_n } 2\di^2(x_0,y_0) \tilde{\rho}_0^{(r)}(x_0)  P'_n(y_0, \d x_0) \, \d\bar\m_n^k(y_0)\\
     &=\int_{\X_\infty } \int_{\X_n} 2\di^2(x_0,y_0) \tilde{\rho}_0^{(r)}(x_0) \, \de p^k_n(x_0,y_0)  \to 0,
 \end{align*}
 and similarly 
 \begin{equation*}
     \int_{\X_\infty } \int_{\X_n \times \X_n} 2\di^2(x_1,y_1) \dfrac{\tilde{\rho}_1^{(r)}(x_1) }{\rho_{1, n}(y_1)}  P'_n(y_1, \d x_1) \, \d\pi_n(y_0, y_1) \to 0.
 \end{equation*}
Moreover H\"older's inequality ensures that 
\begin{align*}
    \int_{\X_\infty} &\int_{\X_n \times \X_n} 2^{k+3} \di(x_0,y_0) \dfrac{\tilde{\rho}_0^{(r)}(x_0) }{\rho_{0, n}(y_0)}  P'_n(y_0, \d x_0) \, \d\pi_n(y_0, y_1) \\
    &\leq 2^{k+3} \left[ \int_{\X_\infty } \int_{\X_n \times \X_n} \di^2(x_0,y_0) \dfrac{\tilde{\rho}_0^{(r)}(x_0) }{\rho_{0, n}(y_0)}  P'_n(y_0, \d x_0) \, \d\pi_n(y_0, y_1)\right] ^\frac12 \to 0,
\end{align*}
and analogously that 
\begin{equation*}
    \int_{\X_\infty } \int_{\X_n \times \X_n} 2^{k+3} \di(x_1,y_1) \dfrac{\tilde{\rho}_1^{(r)}(x_1) }{\rho_{1, n}(y_1)}  P'_n(y_1, \d x_1) \, \d\pi_n(y_0, y_1) \to 0.
\end{equation*}
Therefore, putting together the estimates on every term, we can conclude that
\begin{align*}
    \limsup_{n\to \infty} \int \di^2(x_0,x_1) \de \bar q^{(r)}_n(x_0,x_1) &\leq \limsup_{n\to \infty} \int_{\X_n \times \X_n} \di^2(y_0,y_1) \de \pi_n (y_0,y_1)  \\
    &=\limsup_{n\to \infty} W^2_2(\mu_{0,n}, \mu_{1,n}) = W^2_2(\mu_0^{(r)},\mu_1^{(r)}),
\end{align*}
where we used that $\pi_n$ is an optimal plan and that $\mu_{0,n}\xrightarrow{W_2}\mu_0^{(r)}$, $\mu_{
1,n}\xrightarrow{W_2}\mu_1^{(r)}$. This last inequality, combined with \eqref{eq:boundbelow}, allows us to conclude \eqref{eq:aimofstep}.\\

\textbf{STEP 5: Definition of approximating plan with fixes marginals}\\
We have shown the existence of $n(\varepsilon)\geq \bar n$ such that 
\begin{equation}\label{eq:convtooptimal}
     \bigg |\int \di^2(x_0,x_1) \de \bar q^{(r)}_{n(\varepsilon)}(x_0,x_1)- W^2_2(\mu_0^{(r)},\mu_1^{(r)})\bigg | <\varepsilon.
\end{equation}
Recalling the properties of $\bar n$ proven in the previous steps, we also know that 
\begin{equation}\label{eq:estT1}
T_{K, N'}^{(t)}(\pi_{n(\varepsilon)} | \bar\m^k_{n(\varepsilon)})  \le T_{K, N'}^{(t)}(\bar{q}_{n(\varepsilon)}^{(r)} | \bar\m^k_\infty) + \varepsilon,
\end{equation}
for every $n\geq \bar n$ and every $N'\in [N, \varepsilon)$.
At this point, using $\bar{q}_{n(\varepsilon)}^r$, we define a coupling $q^\varepsilon$ between $\mu_0$ and $\mu_1$ by
\begin{equation}
q^\varepsilon ( . ) := \alpha_r \bar{q}_{n(\varepsilon)}^{(r)} + \tilde{q} ( . \cap (\X_\infty^2 \setminus E_r)).
\end{equation}
First of all, notice that 
\begin{equation*}
    \bigg | \int \di^2(x_0,x_1) \de \bar q^{(r)}_{n(\varepsilon)}(x_0,x_1) - \int \di^2(x_0,x_1) \de  q^\varepsilon(x_0,x_1) \bigg |\leq  (1-\alpha_r) \text{diam} (\mathcal R^k_\infty)^2 \leq \varepsilon 2^{2k+4}.
\end{equation*}
Consequently, putting together this last estimate with \eqref{mutildeclose} and \eqref{eq:convtooptimal}, we can conclude that 
\begin{equation}\label{eq:limitW2}
    \int \di^2(x_0,x_1) \de  q^\varepsilon(x_0,x_1) = W^2_2(\mu_0,\mu_1) + O(\varepsilon).
\end{equation}
On the other hand, it is immediate from the definition of $q^\varepsilon$ that 
\begin{equation}\label{eq:estT2}
(1-\varepsilon)^{1-1/N'}T^{(t)}_{K, N'}(\bar{q}^{r}_{n(\varepsilon)} | \bar\m^k_\infty) \leq \alpha_r^{1-1/N'} T^{(t)}_{K, N'}(\bar{q}^{r}_{n(\varepsilon)} | \bar\m^k_\infty) \leq  T^{(t)}_{K, N'}(q^{\varepsilon} | \bar\m^k_\infty)  . 
\end{equation}

\textbf{STEP 6: Convergence of plans}\\
We turn now to prove the weak convergence of the plans $q^\varepsilon$ introduced in the previous step, as $\varepsilon\to 0$, and consequently the upper semicontinuity of $T^{(t)}_{K,N'}$.\\ 
We first note that, since for every $\varepsilon>0$, it holds that $q^\varepsilon \in \mathsf{Adm}(\mu_0, \mu_1)$, then the family $(q^\varepsilon)_{\varepsilon>0}$ is tight and Prokhorov Theorem ensures the existence of a sequence $(\varepsilon_m)_{m\in \N}$ converging to $0$ such that $q^{\varepsilon_m}\rightharpoonup q \in \mathsf{Adm}(\mu_0,\mu_1)$. Equation \eqref{eq:limitW2} ensures the optimality of $q \in \mathsf{Opt} (\mu_0,\mu_1)$. Furthermore, putting together the estimates \eqref{eq:estT1} (that holds definitely for every $N'\in[N,0)$) and \eqref{eq:estT2}, we conclude that, for every $N'\in [N,0)$ and $t \in [0,1]$,

\begin{equation}
    \begin{split}
    \limsup_{m \to \infty} T_{K, N'}^{(t)}(\pi_{n(\varepsilon_m)} | \m_{n(\varepsilon_m)}) &\leq \limsup_{m \to \infty} T_{K, N'}^{(t)}(\pi_{n(\varepsilon_m)} | \m_{n(\varepsilon_m)}^k) \\
    &= \limsup_{m \to \infty} \frac{1}{\m_{n(\varepsilon_m)}^k(\X_n)^{-1/N'}}  \cdot T_{K, N'}^{(t)}(\pi_{n(\varepsilon_m)} | \bar\m_{n(\varepsilon_m)}^k)\\
    &= \frac{1}{\m_{\infty}^k(\X_\infty)^{-1/N'}}\limsup_{m \to \infty}  T_{K, N'}^{(t)}(\pi_{n(\varepsilon_m)} | \bar\m_{n(\varepsilon_m)}^k)\\
    &\le \frac{1}{\m_{\infty}^k(\X_\infty)^{-1/N'}}\limsup_{m \to \infty} T_{K, N'}^{(t)}(\bar{q}_{n(\varepsilon_m)}^{(r)} | \bar\m^k_\infty) + \varepsilon_m \\
    &=\frac{1}{\m_{\infty}^k(\X_\infty)^{-1/N'}} \limsup_{m \to \infty} T_{K, N'}^{(t)}(\bar{q}_{n(\varepsilon_m)}^{(r)} | \bar\m^k_\infty) \\
    &\leq \frac{1}{\m_{\infty}^k(\X_\infty)^{-1/N'}} \limsup_{m \to \infty} \frac{1}{(1-\varepsilon_m)^{1-1/N'}} T^{(t)}_{K, N'}(q^{\varepsilon_m} | \bar\m^k_\infty)\\
    &= \limsup_{m \to \infty} T^{(t)}_{K, N'}(q^{\varepsilon_m} | \m^k_\infty).
\end{split}
\end{equation}
Now, notice that every $q^{\varepsilon_m}$ has as marginals $\mu_0$ and $\mu_1$, which are supported on $\mathcal{R}^{k-1}_\infty$ and therefore  
\begin{equation*}
    T^{(t)}_{K, N'}(q^{\varepsilon_m} | \m^k_\infty) = T^{(t)}_{K, N'}(q^{\varepsilon_m} | \m_\infty).
\end{equation*}
Thus we can apply Proposition \ref{prop:continuityofT} to  $T^{(t)}_{K, N'}(q^{\varepsilon_m} | \m_\infty)$, for $N'\in I$, which together with the above estimate guarantees that 
\begin{equation}\label{eq:uscofT}
    \limsup_{m \to \infty} T_{K, N'}^{(t)}(\pi_{n(\varepsilon_m)} | \m_{n(\varepsilon_m)}) \leq \limsup_{m \to \infty} T^{(t)}_{K, N'}(q^{\varepsilon_m} | \m_\infty)
   = T^{(t)}_{K, N'}(q | \m_\infty)
\end{equation}
holds for every $N'\in I$ and $t \in [0,1]$.\\

\textbf{STEP 7: Convergence of midpoints}\\
The goal of this step is to show the existence of a limit geodesic $\{\mu_t\}_{t \in [0, 1]}$, such that for any $t \in [0,1]$, $\mu_{t,n(\varepsilon_m)}$ $W_2$-converges (up to subsequences) to $\mu_t$, as $m\to\infty$. Furthermore, we are going to prove a suitable lower semicontinuity of the Renyi entropies that will allow us to pass to the limit the $\CD$ inequality. In order to ease the notation we will denote the Renyi entropy $S_{N,\m_{n(\varepsilon_m)}} $ by $S_{N,n(\varepsilon_m)} $.

\begin{clm*}
For every fixed $t\in [0,1]$, the sequence $(\mu_{t,n(\varepsilon_m)})_{m\in \N}$ converges (up to subsequences) to a measure $\mu_t\in\PX(\X_\infty)$.
\end{clm*}
First of all, notice that estimate \eqref{eq:uscofT}, the $\CD$-condition
\eqref{eq.CDatn} and assumption \eqref{eq:finitenessofT} together ensure that the entropies $S_{N,n(\varepsilon_m)}(\mu_{t,n(\varepsilon_m)})$ are uniformly bounded above by a constant $M'$, for $m\in \N$. Moreover, for every $n(\varepsilon_m)$, the approximation Lemma \ref{prop:mk} provides the existence of a sequence $(\mu_{t,n(\varepsilon_m)}^l)_{l\in \N}$, that $W_2$-converges to $\mu_{t,n(\varepsilon_m)}$, as $l\to \infty$, and such that $\supp(\mu_{t,n(\varepsilon)}^l)\subseteq \mathcal R_{n(\varepsilon)} ^l$. From the proof of Lemma \ref{prop:mk} we recall that $\mu_{t,n(\varepsilon_m)}^l=c^l\,f^l \,  \mu_{t,n(\varepsilon_m)}$ so, we can easily notice that, 
\begin{equation*}
    (c^l)^{-1} =\int f^l \, \de \mu_{t,n(\varepsilon_m)} \geq \mu_{t,n(\varepsilon_m)} (\mathcal R^{l-1}_{n(\varepsilon_m)}) \geq 1- \omega(k,l-1,M+1).
\end{equation*}
Consequently, for sufficiently large $l$ it holds that, as measures,
\begin{equation}\label{eq:ineqofmeasures}
    \mu_{t,n(\varepsilon_m)}^l \leq \frac{1}{1- \omega(k,l-1,M+1)}\cdot  \mu_{t,n(\varepsilon_m)} .
\end{equation}
Notice that we took into account that $\supp (\mu_{0,n(\varepsilon_m)}),\supp (\mu_{1,n(\varepsilon_m)}) \subseteq \mathcal R_{n(\varepsilon_m)}^k$ and we have used the $\omega$-uniform convexity assumption, keeping in mind that $M+1$ bounds from above the terminal entropies \eqref{eq:entropyboundM}.
In turn, inequality \eqref{eq:ineqofmeasures}  implies that,
\begin{equation}\label{eq:entropyboundl}
    S_{N,n(\varepsilon_m)}(\mu_{t,n(\varepsilon_m)}^l)\leq \frac{1}{(1- \omega(k,l-1,M+1))^{1-1/N}} S_{N,n(\varepsilon_m)}(\mu_{t,n(\varepsilon_m)}).
\end{equation}
 Moreover, the fact that $\supp (\mu_{0,n(\varepsilon_m)}),\supp (\mu_{1,n(\varepsilon_m)}) \subseteq \mathcal R _{n(\varepsilon_m)}^k$  shows that the measure $\mu_{t,n(\varepsilon_m)}$ has bounded support. In particular, for every $l\in\N$ (and every $m\in\N$)
\begin{equation*}
    \supp (\mu_{t,n(\varepsilon_m)}^l) \subseteq \supp (\mu_{t,n(\varepsilon_m)}) \subseteq B(p_{n(\varepsilon_m)}, 2^{k+2}).
\end{equation*}
As a consequence of this bound, it's easy to deduce that, 
\begin{equation}\label{eq:W2estimatel}
    W_2^2(\mu_{t,n(\varepsilon_m)}^l,\mu_{t,n(\varepsilon_m)})\leq (2\cdot 2^{2k+2})^2 \,\omega(k,l-1,M+1),
\end{equation}
because $\mu_{t,n(\varepsilon_m)}\leq \mu_{t,n(\varepsilon_m)}^l$ when restricted to $\mathcal R_{n(\varepsilon_m)}^{l-1}$, and $\mu_{t,n(\varepsilon_m)} (X\setminus \mathcal R_{n(\varepsilon_m)}^{l-1})\geq \omega(k,l-1,M+1)$, by $\omega$-uniform convexity. Now, for every fixed $l$ sufficiently large, such that $\omega(k,l-1,M+1)<1$, observe that, according to \eqref{eq:entropyboundl}, the entropies $S_{N,n(\varepsilon_m)}(\mu_{t,n(\varepsilon_m)}^l)$ are uniformly bounded above by the constant
\begin{equation*}
    \frac{1}{(1- \omega(k,l-1,M+1))^{1-1/N}} M',
\end{equation*}
for all $m\in \N$. Notice also that, since $\mu_{t,n(\varepsilon_m)}^l$ is supported on $\mathcal R_{n(\varepsilon_m)}^l$, it holds that
\begin{equation*}
    S_{N,\m^{l+1}_{n(\varepsilon_m)}}(\mu_{t,n(\varepsilon_m)}^l)=S_{N,n(\varepsilon_m)}(\mu_{t,n(\varepsilon_m)}^l).
\end{equation*}
Therefore, Corollary \ref{lem:tight} shows that  $\mu_{t,n(\varepsilon_m)}^l$ weakly converges to some $\mu_{t}^l\in \PX(X_\infty)$ as $m\to \infty$, for some choice of a subsequence. Moreover, we extract the bound $S_{N,\m_\infty^{l+1}}(\mu_t^l)<\infty$  from Corollary \ref{le:lsc_ent}, which guarantees the lower semicontinuity of $S_{N, \cdot}(\cdot)$ along our sequence. Consequently, this implies that the support of $\mu_t^l$ is contained in $\mathcal R_\infty^{l+1}$. 
Finally, note that then the sequence of measures $(\mu_{t,n(\varepsilon_m)}^l)_{m\in\N}$ is supported in a uniformly bounded set, since $\supp (\mu_{t,n(\varepsilon_m)}^l) \subseteq B\left (p_{n(\varepsilon_m)},2^{k+2} \right )$ and $p_{n(\varepsilon_m)}\to p_\infty$, as $m\to\infty$. Thus, we are able to conclude, up to picking again subsequence,  that for every sufficiently large $l \in \N$ 
\[ \mu_{t,n(\varepsilon_m)}^l \xrightarrow{W_2} \mu_{t}^l  \quad \text{ as } \quad m\to \infty.\]
As a matter of fact, we can show using inequality \eqref{eq:W2estimatel} that, 
\begin{equation*}
    W_2^2(\mu_t^i,\mu_t^j) \leq 2^{2k+7} \big[\omega(k,i-1,M+1)) + \omega(k,j-1,M+1)) \big],
\end{equation*}
for every (large enough) $i,j\in \N$. Then, our assumption on $\omega$ ensures that  $(\mu_t^l)_{l\in\N}$ is a Cauchy sequence, which therefore $W_2$-converges to $\mu_t\in \PX(\X_\infty)$. We conclude by noting that the uniform estimate \eqref{eq:W2estimatel} guarantees that $\mu_{t,n(\varepsilon_m)}\to \mu_t$. 

\begin{clm*}
For every $t \in [0,1]$ the measure $\mu_t$ does not give mass to the set $S$ of singular points.
\end{clm*}

For every $m \in \N $ and every $l\in \N$ sufficiently large, let us introduce the measures 
\begin{equation*}
    \tilde\mu_{t,n(\varepsilon_m)}^l =[1- \omega(k,l-1,M+1)] \mu_{t,n(\varepsilon_m)}^l,
\end{equation*}
and notice that, for every $l\in \N$ sufficiently large, 
\begin{equation*}
    \tilde\mu_{t,n(\varepsilon_m)}^l \rightharpoonup \tilde\mu_{t}^l:= [1-\omega(k,l-1,M+1)] \mu_{t}^l.
\end{equation*}
Observe also that all measures $\tilde \mu_{t,n(\varepsilon_m)}^l$ have total mass equal to $[1-\omega(k,l-1,M+1)]$, as $m$ varies. Thus, $\tilde \mu_{t}^l$ also has total mass equal to $[1-\omega(k,l-1,M+1)]$.
On the other hand it follows from the uniform convexity properties (and in particular from \eqref{eq:ineqofmeasures}) that for every $m \in \N $ and every $l\in \N$ sufficiently large, there exists a positive measure $\bar \mu_{t,n(\varepsilon_m)}^l$ such that
\begin{equation*}
    \mu_{t,n(\varepsilon_m)} = \tilde \mu_{t,n(\varepsilon_m)}^l+\bar \mu_{t,n(\varepsilon_m)}^l.
\end{equation*}
Notice that, since the sequences $ \mu_{t,n(\varepsilon_m)}$ and $(\tilde \mu_{t,n(\varepsilon_m)}^l)_{m\in \N}$ are weakly converging, the sequence $\bar \mu_{t,n(\varepsilon_m)}^l$ is also weakly converging to a (positive) measure $\bar \mu_{t}^l$, such that
\begin{equation*}
    \mu_t= \tilde \mu_t^l + \bar \mu_t^l.
\end{equation*}
As pointed out before $\mu_t^l$ is supported on $\mathcal R_\infty^{l+1}$, thus the same holds for $\tilde \mu_t^l$, and therefore $$\mu_t(\mathcal R_\infty^{l+1})\geq 1-\omega(k,l-1,M+1).$$ 
Finally observe that this is sufficient to prove the claim, because of the arbitrariness of $l$.

\begin{clm*}
The lower semicontinuity of the Renyi entropies holds, that is for every $N'\in[N,0)$ 
\begin{equation}\label{eq:step3}
    S_{N', \m_\infty}(\mu_{t}) \leq \liminf_{m\to \infty} S_{N', n(\varepsilon_m)}(\mu_{t, n(\varepsilon_m)}).
\end{equation}
\end{clm*}

First of all notice that, the result of Claim 2 combined with Proposition \ref{prop:almostlscofSN} yields that 
\begin{equation}\label{eq:lscatinfty}
   S_{N', \m_\infty}(\mu_{t}) \leq \liminf_{l\to \infty} S_{N', \m_\infty}(\mu_{t}^l).
\end{equation}
On the other hand Lemma \ref{le:lsc_ent} ensures that for every $l\in\N$ large enough
\begin{equation}\label{eq:lscatl}
    S_{N', \m_\infty}(\mu_{t}^l)= S_{N', \m_\infty^{l+2}}(\mu_{t}^l)\leq  \liminf_{m\to \infty} S_{N', \m^{l+2}_{n(\varepsilon_m)}}(\mu_{t, n(\varepsilon_m)}^l)=\liminf_{m\to \infty} S_{N', n(\varepsilon_m)}(\mu_{t, n(\varepsilon_m)}^l).
\end{equation}
Moreover, we  deduce as in Claim 1 the following estimate for every $N'\in[N,0)$
\begin{equation*}
    S_{N',n(\varepsilon_m)}(\mu_{t,n(\varepsilon_m)}^l)\leq \frac{1}{(1- \omega(k,l-1,M+1))^{1-1/N'}} S_{N',n(\varepsilon_m)}(\mu_{t,n(\varepsilon_m)})
\end{equation*}
and consequently for every $l\in \N$
\begin{equation*}
\begin{split}
    \liminf_{m\to \infty} S_{N',n(\varepsilon_m)}(\mu_{t,n(\varepsilon_m)}) &\geq \liminf_{m\to \infty} (1-\omega(k,l-1,M+1))^{1-1/N'} S_{N',n(\varepsilon_m)}(\mu_{t,n(\varepsilon_m)}^l)\\
    &\geq (1-\omega(k,l-1,M+1))^{1-1/N'}S_{N', \m_\infty}(\mu_{t}^l),
\end{split}
\end{equation*}
where the last passage follows from \eqref{eq:lscatl}. Then, since this last inequality holds for every $l\in\N$, we can conclude that
\begin{equation*}
\begin{split}
     \liminf_{m\to \infty} S_{N',n(\varepsilon_m)}(\mu_{t,n(\varepsilon_m)})&\geq \liminf_{l\to \infty}(1-\omega(k,l-1,M+1))^{1-1/N'}S_{N', \m_\infty}(\mu_{t}^l)\\
    &\geq  S_{N', \m_\infty}(\mu_{t}),
\end{split}
\end{equation*}
where we used \eqref{eq:lscatinfty}. This is exactly what we wanted to prove.\\

\textbf{CONCLUSION}\\
So far we were able to prove that for every fixed $t\in[0,1]$, the sequence $(\mu_{t,n(\varepsilon_m)})_{m\in \N}$ converges (up to subsequences) to a measure $\mu_t\in\PX_2(\X_\infty)$ and 
\begin{equation}\label{eq:lsconQ}
      S_{N', \m_\infty}(\mu_{t}) \leq \liminf_{m\to \infty} S_{N', n(\varepsilon_m)}(\mu_{t, n(\varepsilon_m)}),
\end{equation}
for every $N'\in[N,0)$.
Now, a diagonal argument ensures that, by selecting a suitable subsequence (that we do not rename for sake of simplicity), 
\begin{equation*}
    (\mu_{t,n(\varepsilon_m)}) \overset{W_2}{\longrightarrow} \mu_t,
\end{equation*}
and that estimate \eqref{eq:lsconQ} holds for every $t\in[0,1]\cap \mathbb Q$.
Our approximation ensures also that $\mu_{0,n(\varepsilon_m)}\overset{W_2}{\longrightarrow}\mu_0 $ and $\mu_{1,n(\varepsilon_m)}\overset{W_2}{\longrightarrow}\mu_1 $ therefore, since $\mu_{t,n(\varepsilon_m)}$ is a $t$-midpoint of $\mu_{0,n(\varepsilon_m)}$ and $\mu_{1,n(\varepsilon_m)}$, for every $t\in[0,1]\cap \mathbb Q$ the limit point $\mu_t$ is a $t$-midpoint of $\mu_0$ and $\mu_1$. Now it is easy to realize that we can extend by continuity $\mu_t$ to a Wasserstein geodesic (connecting $\mu_0$ and $\mu_1$) on the whole interval $[0,1]$, obtaining also that 
\begin{equation*}
    (\mu_{t,n(\varepsilon_m)}) \overset{W_2}{\longrightarrow} \mu_t, \qquad \text{for every }t\in [0,1].
\end{equation*}
Moreover, we know from the proof of Claim 2, that for every $l\in\N$ 
\begin{equation*}
    \mu_t(\mathcal R_\infty^{l+1})\geq 1-\omega(k,l-1,M+1),
\end{equation*}
for every $t\in[0,1]\cap \mathbb Q$. Then by continuity we can conclude the same inequality for every $t\in[0,1]$, and consequently we know that $\mu_t$ gives no mass to the set of singular points.\\
Finally, inequality \eqref{eq:lsconQ}, combined with \eqref{eq:uscofT}, allows to pass to the limit as $m\to \infty$ inequality \eqref{eq.CDatn} at every rational time and obtaining that
\begin{equation*}
    S_{N', \m_\infty}(\mu_{t}) \leq T^{(t)}_{K, N'}(q | \m_\infty)
\end{equation*}
holds for every $t\in[0,1]\cap \mathbb Q$ and every $N'\in I$. Finally, the lower semicontinuity of the entropy (ensured by the fact that $\mu_t$ gives no mass to the set of singular points) and the continuity of $T_{K,N'}^{(t)}(q | \m_\infty)$ in $t$ (which is a straightforward consequence of the dominated convergence theorem), allow to extend this last inequality to every $t\in[0,1]$, concluding the proof of the approximate $\CD$-condition.

\subsection{Proof of the CD Condition}\label{section:CD}
This final section is dedicated to the proof of our main result, that is Theorem \ref{th:ext}.. As already mentioned, the proof of the approximate $\CD$-condition and the approximation argument that made it possible are the foundation to prove the $\CD$-condition. As the reader will notice, we are going to use basically the same techniques, but refining them a little bit to achieve the more general result. We specify that we could prove the $\CD$-condition directly, but we preferred to divide the proof in order to be clearer.

Before going on, we prove a preliminary lemma, that will help us in the following. Notice that a result of this type is now needed because the marginals may not have bounded support.
\begin{lemma}\label{lem:geodesic}
Given a metric space $(X,d)$, for every $n\in \N$ let $(\nu_t^n)_{t\in[0,1]}\subset \PX_2(X)$ be a Wasserstein geodesic. Assume that for every $t\in[0,1]$ the family $(\nu_t^n)_{n\in \N}$ is tight and that there exist $\nu_0,\nu_1\in\PX_2(X)$ such that 
\begin{equation*}
    \nu_0^n \overset{W_2}{\longrightarrow} \nu_0 \quad \text{ and } \quad  \nu_1^n \overset{W_2}{\longrightarrow} \nu_1 \qquad \text{ as }n \to \infty.
\end{equation*}
Then there exists a Wasserstein geodesic $(\nu_t)_{t\in[0,1]}\subset \PX_2(X)$ connecting $\nu_0$ and $\nu_1$ such that, up to subsequences, 
\begin{equation*}
    \nu_t^n \rightharpoonup \nu_t \qquad \text{as } n \to \infty\text{, for every $t\in [0,1]\cap \mathbb Q$}.
\end{equation*}
\end{lemma}

\begin{proof}
First of all, notice that applying Prokhorov theorem we deduce that, for every fixed $t$, the sequence $(\nu_t^n)_{n\in \N}$ is weakly convergent, up to subsequences. Thus the diagonal argument ensures that, up to take a suitable subsequence which we do not recall for simplicity, for every $t\in [0,1]\cap \mathbb Q$ there exists $\nu_t\in \PX_2(X)$ such that
\begin{equation*}
    \nu_t^n\rightharpoonup \nu_t \qquad\text{as } n \to \infty\text{, for every $t\in [0,1]\cap \mathbb Q$}.
\end{equation*}
It is well-known that the Wasserstein distance is lower semicontinuous with respect to the weak convergence (see for example Proposition 7.1.3 in \cite{AmbrosioGigliSavare08}, then
\begin{equation*}
    W_2(\nu_0, \nu_t)\leq \liminf_{n\to \infty} W_2(\nu_0^n, \nu_t^n)= \liminf_{n\to \infty} t\cdot W_2(\nu_0^n, \nu_1^n)= t \cdot W_2(\nu_0, \nu_1)
\end{equation*}
and analogously
\begin{equation*}
    W_2(\nu_t, \nu_1)\leq (1-t) \cdot W_2(\nu_t, \nu_1). 
\end{equation*}
Combining this two inequalities with the triangular inequality we deduce that 
\begin{equation*}
     W_2(\nu_0, \nu_t)=  t \cdot W_2(\nu_0, \nu_1)\quad \text{and}\quad W_2(\nu_t, \nu_1)= (1-t) \cdot W_2(\nu_t, \nu_1),
\end{equation*}
which means that $\nu_t$ is a $t$-midpoint of $\nu_0$ and $\nu_1$. The lower semicontinuity of the Wasserstein distance also ensures that for every $s,t\in[0,1]\cap \mathbb Q$ it holds that
\begin{equation*}
    W_2(\nu_t,\nu_s) \leq \liminf_{n\to \infty} |t-s|\cdot W_2(\nu_t^n, \nu_s^n)= |t-s|\cdot\liminf_{n\to \infty}  W_2(\nu_0^n, \nu_1^n)=|t-s|\cdot W_2(\nu_0,\nu_1).
\end{equation*}
Finally, since for every $r \in [0,1]\cap \mathbb Q$ $\nu_r$ is an $r$-midpoint of $\nu_0$ and $\nu_1$, the triangular inequality allow us conclude that
\begin{equation*}
    W_2(\nu_t,\nu_s) =|t-s|\cdot W_2(\nu_0,\nu_1), \qquad \text{for every $s,t\in[0,1]\cap \mathbb Q$},
\end{equation*}
then we can extend $\nu_t$ to the whole interval $[0,1]$, finding a Wasserstein geodesic $(\nu_t)_{t\in[0,1]}$ connecting $\nu_0$ and $\nu_1$.
\end{proof}

Now that we have this last result at our disposal we can proceed to the proof of Theorem \ref{th:stab}.
To this aim, we fix $\mu_0, \mu_1 \in \PX_2^{ac}(\X_\infty)$. In analogy with the previous section, we can assume that $S_{N,\m_\infty}(\mu_0),S_{N,\m_\infty}(\mu_1)<\infty$ and introduce the constant  
\begin{equation*}
  M:=\max \{S_{N,\m_\infty}(\mu_0),S_{N,\m_\infty}(\mu_1)\}.
\end{equation*}
 We can also define the interval
\begin{equation*}
    I:= \{ N'\in [N,0)\, :\, S_{N',\m_\infty}(\mu_0),S_{N',\m_\infty}(\mu_1)<\infty\},
\end{equation*}
in particular we will need to prove \eqref{def:CD} for every $N'\in I$ and every $t\in[0,1]$. Now, according to Lemma \ref{prop:mk} there exist two sequences $(\mu_0^l)_{l\in\N}$ and $(\mu_1^l)_{l\in\N}$, $W_2$-converging to $\mu_0$ and $\mu_1$ respectively and such that
\begin{equation*}
    \supp(\mu_0^l),\supp(\mu_1^l)\subseteq \mathcal R_\infty^{l-1} \qquad \text{for every }l\in \N.
\end{equation*}
Moreover, keeping in mind the definition of $\mu_0^l$ and $\mu_1^l$, it is easy to realize that for $l$ sufficiently large
\begin{equation*}
    S_{N',\m_\infty}(\mu_0^l),S_{N',\m_\infty}(\mu_1^l)<\infty \quad \text{ for every $N'\in I$}
\end{equation*}
and that the dominated convergence theorem ensures that
\begin{equation*}
    \lim_{l\to\infty}S_{N,\m_\infty}(\mu_0^l) = S_{N,\m_\infty}(\mu_0) \qquad \text{ and }\qquad \lim_{l\to\infty}S_{N,\m_\infty}(\mu_1^l) = S_{N,\m_\infty}(\mu_1).
\end{equation*}
Thus, for every $l$ large enough 
\begin{equation*}
    S_{N,\m_\infty}(\mu_0^l),S_{N,\m_\infty}(\mu_1^l)\leq \max \big\{S_{N,\m_\infty}(\mu_0),S_{N,\m_\infty}(\mu_1) \big\} +1 :=M+1
\end{equation*}
 and then we can apply the argument presented in the last section and deduce the existence of an optimal plan $q^l\in \mathsf{Opt} (\mu^l_0,\mu^l_1)$ and of a Wasserstein geodesic $(\mu^l_t)_{t\in[0,1]}$ connecting $\mu_0^l$ and $\mu_1^l$, such that
\begin{equation}\label{eq:approxCD}
    S_{N',\m_\infty}(\mu_t^l) \leq T_{K,N'}^{(t)}(q^l | \m_\infty)
\end{equation}
holds for every $t\in[0,1]$ and every $N'\in I$.
Now, we divide the proof in two steps, the first dedicated to the convergence of the plans $(q^l)_{l\in\N}$ and to the upper semicontinuity of $T_{K,N'}^{(t)}$, the second dedicated to the convergence of the measures $(\mu_t^l)_{l\in\N}$ and  the lower semicontinuity of $S_{N',\m_\infty}$.\\

\textbf{Step 1: Upper semicontinuity for $T_{K,N'}^{(t)}$}

Notice that $(q^l)_{l\in\N}$ is a sequence of probability measures having as marginals two sequences of converging, and thus tight, probability measures. As a consequence the sequence $(q^l)_{l\in\N}$ is itself tight, then up to subsequences it weakly converges to a plan $q\in \PX(\X_\infty \times \X_\infty)$. It is well know and easy to prove that $q\in \mathsf{Opt}(\mu_0,\mu_1)$. We are now going to prove that 
\begin{equation}\label{eq:uscofT2}
    \limsup_{l\to \infty} T_{K,N'}^{(t)}(q^l|\m_\infty) \leq T_{K,N'}^{(t)}(q|\m_\infty)
\end{equation}
for every $t\in[0,1]$ and every $N'\in I$. The argument we are going to use is essentially the same as the one explained in the proof of Proposition \ref{prop:continuityofT}, nevertheless we briefly recall it for the sake of completeness, avoiding to repeat all the details. \\
In particular, for every $l\in \N$ let us call $\rho_0^l$ and $\rho_1^l$ the densities of $\mu_0^l$ and $\mu_1^l$ with respect to the reference measure $\m_\infty$, we just need to prove that
\begin{equation*}
   \lim_{l\to \infty} \int \tau_{K,N'}^{(1-t)} (\di(x,y)) \rho_0^l(x)^{-\frac 1{N'}} \de q^l = \int \tau_{K,N'}^{(1-t)} (\di(x,y)) \rho_0(x)^{-\frac 1{N'}} \de q.
\end{equation*}
 Notice that, the particular definition of $\mu_0^l$ (check Lemma \ref{prop:mk}), ensures that the density $\rho_0^l$ is a suitable renormalization of $f^l\rho_0$, then for a fixed $\varepsilon>0$ we can find $\bar l\in \N$ such that
 \begin{equation*}
    \norm{\big(\rho_0^l \big)^{ -1/N'}-\rho_0^{- 1/N'}}_{L^1(\mu_0^l)}<\varepsilon \, \, \text{ for every }l\geq \bar l.
 \end{equation*}
 Furthermore, for the same reason combined with Lemma \ref{lemma:dens} (up to possibly change $\bar l$) we can find $f^\varepsilon\in C_b(X)$ such that
\begin{equation*}
        \norm{\rho_0^{- 1/N'}-f^\varepsilon}_{L^1(\mu_0)}<\varepsilon \qquad\text{and} \qquad \norm{\rho_0^{- 1/N'}-f^\varepsilon}_{L^1(\mu_0^l)}<\varepsilon \,\,\,\text{ for every }l\geq \bar l.
\end{equation*}
Putting together this last two estimates we end up proving that 
\begin{equation*}
    \norm{\big(\rho_0^l \big)^{ -1/N'}-f^\varepsilon}_{L^1(\mu_0^l)}<2\varepsilon \, \, \text{ for every }l\geq \bar l.
 \end{equation*}
  On the other hand, since the function
\begin{equation*}
    \tau_{K,N}^{(1-t)} (\di(x,y)) f^\varepsilon(x)
\end{equation*}
is bounded and continuous, the weak convergence of $(q^l)_l$ to $q$ yields that
\begin{equation*}
    \lim_{l\to \infty} \int \tau_{K,N'}^{(1-t)} (\di(x,y)) f^\varepsilon(x) \de q^l = \int \tau_{K,N'}^{(1-t)} (\di(x,y)) f^\varepsilon(x) \de q.
\end{equation*}
Then, since definitely $l\geq \bar l$, we can deduce the following estimate
\begin{equation*}
    \begin{split}
        \limsup_{l\to \infty} \int \tau_{K,N'}^{(1-t)} (\di(x,y)) \rho_0^l(x)^{-\frac 1{N'}} \de q^l &\leq \lim_{l\to \infty} \int \tau_{K,N'}^{(1-t)} (\di(x,y)) f^\varepsilon(x) \de q^l + 2\varepsilon  \norm{\tau_{K,N'}^{(1-t)}}_{L^\infty} \\
        &= \int \tau_{K,N'}^{(1-t)} (\di(x,y)) f^\varepsilon(x) \de q + 2\varepsilon  \norm{\tau_{K,N'}^{(1-t)}}_{L^\infty}\\
        &\leq \int \tau_{K,N'}^{(1-t)} (\di(x,y)) \rho_0(x)^{-\frac 1{N'}} \de \pi + 3  \varepsilon  \norm{\tau_{K,N'}^{(1-t)}}_{L^\infty}.
    \end{split}
\end{equation*}
and since $\varepsilon>0$ can be chosen arbitrarily, \eqref{eq:uscofT2} holds true.\\

\textbf{Step 2: Lower semicontinuity for $S_{N',\m_\infty}$}\\
In this second step we prove an additional property on $\mu^l_t$, which is fundamental to prove the $\CD$-condition. Let us start with a preliminary lemma.

\begin{lemma}\label{lem:limitinl}
Fix $k\leq h \in \N$ and let $\nu\in \PX_{ac}(\X_\infty,\m_\infty)$ with bounded density be such that $\supp(\nu)\subseteq\mathcal R_\infty^{k-1}$. Then, for every $\epsilon>0$, there exists $\tilde n\in \N$ large enough such that, $P'_{n,h}(\nu)(\mathcal R_n^{k+1})\geq 1-\epsilon$ for every $n\geq \tilde n$.
\end{lemma}

\begin{proof}
Notice that, accordingly to Lemma \ref{lemma:I4.19}, both the sequences $(P'_{n,k}(\nu))_{n\in \N}$ and $(P'_{n,h}(\nu))_{n\in \N}$ $W_2$-converge to $\nu^k$. Assume that $P'_{m,h}(\nu)(\mathcal R_m^{k+1})< 1-\epsilon$ for some $m\in \N$. Observe that 
\begin{equation*}
    \inf \big\{ \di(x,y) \,:\, x\in\mathcal R_m^{k}, \, y \in \mathcal (R_m^{k+1})^c \big\}= 2^{-(k+2)},
\end{equation*}
as a consequence, since $P'_{m,k}(\nu^k)$ is supported on $\mathcal R^k_m$, we obtain that
\begin{equation}\label{eq:W2geq}
    W_2^2(P'_{m,k}(\nu),P'_{m,h}(\nu)) \geq \epsilon \cdot 2^{-(2k+4)}.
\end{equation}
On the other, since the sequences $(P'_{n,k}(\nu^k))_{n\in \N}$ and $(P'_{n,h}(\nu^k))_{n\in \N}$ have the same limit, it holds that
\begin{equation*}
     W_2^2(P'_{n,k}(\nu^k),P'_{n,h}(\nu^k))\to 0 \qquad \text{as }n\to \infty.
\end{equation*}
Then definitely \eqref{eq:W2geq} cannot hold, proving the desired result.
\end{proof}
Fix $\epsilon>0$ and take $k(\epsilon)\in \N$ such that $\mu_0(\mathcal R_\infty^{k(\epsilon)-1}),\mu_1(\mathcal R_\infty^{k(\epsilon)-1})>1-\frac\epsilon 2$. Then we take $l> k(\epsilon)$ and repeat the argument of the previous section on $\mu_0^l$ and $\mu_1^l$. We are also going to use the same notation, forgetting for the moment the dependence on $l$. 
It is easy to realize that there exist two measures $\nu^{k(\epsilon)}_0$ and $\nu_1^{k(\epsilon)}$ with $\supp(\nu^{k(\epsilon)}_0),\supp(\nu^{k(\epsilon)}_1)\subseteq\mathcal R_\infty^{k(\epsilon)-1}$ and $\nu^{k(\epsilon)}_0(\X_\infty),\nu^{k(\epsilon)}_1(\X_\infty)> 1-\epsilon$, such that for $r$ sufficiently large (and thus for $\varepsilon$ sufficiently small) $\tilde\mu_0^{(r)}\geq \nu^{k(\epsilon)}_0$ and $\tilde\mu_1^{(r)}\geq \nu^{k(\epsilon)}_1$ (in particular this tells us that $\nu^{k(\epsilon)}_0$ and $\nu_1^{k(\epsilon)}$ have bounded density). Then we can apply Lemma \ref{lem:limitinl} to the probability measures
\begin{equation*}
    \frac{1}{\nu^{k(\epsilon)}_0(\X_\infty)}\nu^{k(\epsilon)}_0 \quad \text{and} \quad \frac{1}{\nu^{k(\epsilon)}_1(\X_\infty)}\nu^{k(\epsilon)}_1,
\end{equation*}
obtaining that for $m$ sufficiently large (in particular such that $n(\varepsilon_m)\geq \tilde n$) it holds that
\begin{equation*}
    \mu_{0,n(\varepsilon_m)} (\mathcal R_{n(\varepsilon_m)}^{k(\epsilon)+1}), \mu_{1,n(\varepsilon_m)} (\mathcal R_{n(\varepsilon_m)}^{k(\epsilon)+1}) \geq (1-\epsilon)^2\geq 1-2 \epsilon.
\end{equation*}
Consequently our uniform convexity assumption ensures that 
\begin{equation*}
    \mu_{t,n(\varepsilon_m)} (\mathcal R_{n(\varepsilon_m)}^h) \geq 1- \Omega(k(\epsilon)+1, h, M+2,2\epsilon),
\end{equation*}
for every $t\in[0,1]$ and every $h\in \N$. Proceeding as in Step 7 and Conclusion of the previous section we can actually conclude that
\begin{equation}\label{eq:boundXh}
    \mu_t^l (\mathcal R_\infty^{h+1}) \geq 1- \Omega (k(\epsilon)+1, h-1,M+2,2\epsilon),
\end{equation}
for every $t\in[0,1]$ and every $h\in \N$ sufficiently large. \\

\begin{clm*}
For a fixed $t>0$, the family $(\mu_t^l)_{l\in \N}$ is tight.
\end{clm*}
Given a fixed $\delta>0$, we have to find a compact set $K_\delta$, such that $\mu_t^l(K_\delta)\geq 1-\delta$ for every $l\in\N$. To this aim we take suitable $\epsilon$ and $h$ such that \eqref{eq:boundXh} ensures that
\begin{equation}\label{eq:tightness}
    \mu_t^l(\mathcal R_\infty^{h+1})\geq 1-\frac\delta2.
\end{equation}
Moreover, combining the result of Step 1 (that is \eqref{eq:uscofT2}) with \eqref{eq:approxCD}, we conclude that $S_{N,\m_\infty}(\mu_t^l)$ is definitely bounded. Then, since $\m_\infty|_{\mathcal R_\infty^{h+1}}$ is a finite Radon measure, we can argue as in the proof of Lemma \ref{lem:tightness} and prove the tightness of the family of measure $(\mu_t^l|_{\mathcal R_\infty^{h+1}})_{l\in \N}$. As a consequence, keeping in mind \eqref{eq:tightness}, there exists a compact set $K_\delta$ such that 
\begin{equation*}
    \mu_t^l(K_\delta) \geq \mu_t^l|_{\mathcal R_\infty^{h+1}}(K_\delta)\geq 1-\delta,
\end{equation*}
proving the claim.
\\

Now, we can apply Lemma \ref{lem:geodesic} and find a Wasserstein geodesic $(\mu_t)_{t\in[0,1]}\subset\PX(\X_\infty)$ connecting $\mu_0$ and $\mu_1$ such that (up to subsequences) 
\begin{equation*}
    \mu_t^l \rightharpoonup \mu_t \in\PX(\X_\infty) \qquad \text{as $l\to\infty$ for every $t\in[0,1]\cap\mathbb Q$}
\end{equation*}
Then, since the bound \eqref{eq:boundXh} is uniform in $l$, we can conclude that
\begin{equation}\label{eq:massRh}
    \mu_t (\mathcal R_\infty^{h+1}) \geq 1- \Omega (k(\epsilon)+1, h-1,M+2,2\epsilon),
\end{equation}
for every $t\in[0,1]\cap \mathbb Q$ and every $h\in \N$ sufficiently large. Moreover, by continuity we can deduce \eqref{eq:massRh} for every time $t\in[0,1]$ (and every $h\in \N$ sufficiently large).
This is sufficient to conclude that $\mu_t$ gives no mass to the set $\mathcal S$ of singular points, for every $t\in[0,1]$. In fact, assume by contradiction that $\mu_t(\mathcal S)=\delta>0$, then condition \eqref{eq:Omega}
on $\Omega$ ensures that there exist $\epsilon$ and $h\in \N$ such that 
\begin{equation*}
    \Omega (k(\epsilon)+1, h-1,M+2,2\epsilon)) < \delta,
\end{equation*}
 and consequently 
 \begin{equation*}
    \mu_t (\mathcal R_\infty^{h+1}) \geq 1- \delta,
 \end{equation*}
which contradicts $\mu_t(\mathcal S)=\delta$. At his point we know from Proposition \ref{prop:almostlscofSN} that 
\begin{equation}\label{eq:lscofSN}
    \liminf_{l\to \infty} S_{N',\m_\infty} (\mu_t^l) \leq S_{N',\m_\infty}(\mu_t),
\end{equation}
for every $t\in[0,1]\cap \mathbb Q$ and $N'\in[N,0)$.\\

Finally, we can use \eqref{eq:lscofSN} and \eqref{eq:uscofT2} to pass at the limit as $l\to \infty$ the inequality \eqref{eq:approxCD} and deduce that
\begin{equation}\label{eq:last}
    S_{N',\m_\infty}(\mu_t) \leq T_{K,N'}^{(t)}(q | \m_\infty)
\end{equation}
holds for every $t\in[0,1]\cap\mathbb Q$ and every $N'\in I$. Then the lower semicontinuity of $ S_{N',\m_\infty}$ (granted by \eqref{eq:massRh}) and the continuity of $T_{K,N'}^{(t)}(q | \m_\infty)$ in $t$ (which is a straightforward consequence of the dominated convergence theorem), allow to conclude \eqref{eq:last} for every $t\in[0,1]$, finishing the proof.\\



\bibliographystyle{siam}
\bibliography{Stability_negN}
\end{document}